\documentclass[11pt,a4paper,oneside]{amsart}
\pdfoutput=1
\usepackage{amssymb}
\usepackage{esint}
\usepackage{mathtools}
\usepackage{fullpage} 
\usepackage{hyperref}
\usepackage{enumerate}
\usepackage[english]{babel}
\usepackage[utf8]{inputenc} 
\usepackage[pdftex]{graphicx}
\usepackage{siunitx} 
\usepackage{subcaption}

\theoremstyle{plain}
\newtheorem{theorem}{Theorem}
\newtheorem{lemma}[theorem]{Lemma}
\newtheorem{corollary}[theorem]{Corollary}
\newtheorem{proposition}[theorem]{Proposition}

\theoremstyle{definition}

\theoremstyle{remark}
\newtheorem{remark}[theorem]{Remark}

\newcommand{\T}{\mathbb{T}}
\newcommand{\R}{\mathbb{R}}
\newcommand{\Z}{\mathbb{Z}}
\newcommand{\C}{\mathbb{C}}
\newcommand{\N}{\mathbb{N}}

\newcommand{\xrt}{\mathcal{I}}

\newcommand{\iip}[2]{\left(#1,#2\right)}
\newcommand{\vev}[1]{\left\langle#1\right\rangle}
\newcommand{\der}{\mathrm{d}}
\newcommand{\eps}{\varepsilon}
\renewcommand{\phi}{\varphi}
\newcommand{\abs}[1]{\left| #1 \right|}
\newcommand{\norm}[1]{\Vert #1 \Vert}
\newcommand{\aabs}[1]{\left\| #1 \right\|}

\newcommand{\dt}{\der t}
\DeclareMathOperator{\id}{id}
\DeclareMathOperator*{\argmin}{arg\,min}

\title{Torus computed tomography}

\author{Joonas Ilmavirta}
\address{Department of Mathematics and Statistics, University of Jyv\"askyl\"a, P.O.Box~35 (MaD) FI-40014 University of Jyv\"askyl\"a, Finland}
\email{joonas.ilmavirta@jyu.fi}
\email{jesse.t.railo@jyu.fi}
\author{Olli Koskela}
\address{HAMK Smart Research Unit, Häme University of Applied Sciences, Hämeenlinna, FI-13100, Finland and BioMediTech Institute and Faculty of Medicine and Health Technology, Tampere University, FI-33014 Tampere University, Finland}
\email{olli.koskela@hamk.fi}
\author{Jesse Railo}
\date{\today}
\thanks{This is the accepted manuscript of the work first published in \emph{SIAM J. Appl. Math.}, 80 (4), 1947--1976. \href{https://doi.org/10.1137/19M1268070}{10.1137/19M1268070}.}

\begin{document}

\begin{abstract} We present  a  new  computed  tomography  (CT)  method  for  inverting  the Radon  transform  in  2D.  The idea relies on the geometry of the flat torus, hence we call the new method Torus CT. We prove new inversion formulas for integrable functions, solve a minimization problem associated to Tikhonov regularization in Sobolev spaces and prove that the solution operator provides an admissible regularization strategy with a quantitative stability estimate. This regularization is a simple post-processing low-pass filter for the Fourier series of a phantom. We also study the adjoint and the normal operator of the X-ray transform on the flat torus. The X-ray transform is unitary on the flat torus. We have  implemented the Torus CT method using Matlab and  tested it with simulated data with promising results. The inversion method is meshless in the sense that it gives out a closed form function that can be evaluated at any point of interest. 
\end{abstract}


\maketitle

\section{Introduction}

We present a new computed tomography (CT) method for X-ray tomography in 2D, {based on a relatively recent theoretical approach to X-ray tomography.}
The method reconstructs the Fourier series of a phantom via the projection of X-ray data into X-ray data on the flat torus which has a remarkably simple inverse X-ray transform.
Therefore we call the new method \textit{Torus CT}.

{
We emphasize that Torus CT is new as a numerical approach, so it is beyond the scope of a single article to give a full and detailed answer to all the natural questions that arise.
This article should be seen as a proof of concept, exploring a new approach and showing that it is indeed feasible.
}

We have developed new mathematical theory and computational implementations.
{
We found a new inversion method and proved that it provides a regularization strategy with respect to suitable Sobolev norms; see sections~\ref{sec:results}.}
The numerical implementation was used to demonstrate the potential of Torus CT method in various simulations and tests, including data simulation in torus geometry and traditional experimental projections.
Torus CT provided an efficient basis for inverse solution and its efficacy is shown in this work.

The article is organized as follows. In section~\ref{sec:ct} we give an overview of computed tomography and regularization, in section~\ref{sec:radon} we discuss works related to X-ray tomography on torus, and in section~\ref{sec:results} we state the main theoretical results in this paper. Section~\ref{sec:torusCTall} includes mathematical preliminaries, proofs of theorems and numerical analysis for Torus CT method. Section~\ref{sec:implementation} contains mathematical formulation of computational forward and inverse models. Section~\ref{sec:numerics} presents numerical experiments and their analysis. Conclusions are given in section~\ref{sec:conclusions}.
We have included a short note on supplementary material in the end of the article.

\subsection{Overview of computed tomography and regularization methods}
\label{sec:ct}

We give here an overview of X-ray tomography. Practical CT imaging was first introduced by Cormack and Hounsfield in 1970s based on the theoretical work of Cormark \cite{C63,C64} in 1960s. The mathematical theory itself was in fact earlier studied by Radon~\cite{R17} in 1917. We give here only a narrow sample of topics and references in X-ray computed tomography. More references can be found in the cited works.

CT has many applications in medical imaging and engineering utilizing computerized axial tomography (CAT), positron-emission tomography (PET) and single-photon emission computed tomography (SPECT)~\cite{K14}. Possible applications include imaging of patients in medicine and nondestructive testing in engineering. The most common inversion method for CT imaging is based on the filtered back-projection (FBP) algorithms \cite{N01, KS01}. 

{The FBP algorithms work well if there is sufficiently dense set of measurements. The filter in the standard FBP algorithm stands for applying the fractional Laplacian $(-\Delta)^{1/2}$ to the back-projected Radon transform data, and this does not include any regularization \cite{H99,N01}. In addition, regularization is required to control the errors in reconstructions caused by a measurement noise. See for example \cite{MS12,K14}.}

Usually a regularization method is applied for a discretized X-ray tomography model as in the examples listed next. The most common regularization methods include Tikhonov regularization and truncated singular value decomposition (TSVD) which promote smoothness of reconstructions~\cite{MS12}. {Other common regularization approaches include total variation (TV) regularization which promotes $\ell^1$ sparsity and, therefore, piecewise constant reconstructions \cite{Setal03,Hetal13,Netal16,Hetal17}. Another approach is to encode a priori information as a probability distribution and think the reconstruction problem as an Bayesian inverse problem for finding a posterior distribution \cite{Setal03,Ketal03,KS05,Hetal13,Hetal17}.}

The main difference of our proposed Tikhonov regularization approach, stated in theorem~\ref{thm:regularization}, compared to the usual regularization methods is that we do not discretize a phantom and regularization takes a form of a simple low-pass filter on the Fourier side. This also reflects the fact that Torus CT method is meshless (or meshfree) method. Theorem~\ref{thm:strategy} states that the proposed regularization method is an admissible regularization strategy. Details are given in the subsequent sections.

\subsection{The X-ray transform on torus, the Radon transform and the geodesic X-ray transform}
\label{sec:radon}

In this paper we consider application of the X-ray transform on the flat torus $\T^n = \R^n /\Z^n$ to the usual CT in the case when $n = 2$. In this section we give an account of theoretical works on the X-ray transforms on tori. As expected, the \textit{$d$-plane Radon transform} of a function~$f$ on~$\T^n$ encodes the integrals of~$f$ over all periodic $d$-planes. The X-ray transform corresponds to the case when $d = 1$ and is in fact the~\textit{geodesic X-ray transform} on~$\T^n$ over closed geodesics. It is described in section~\ref{sec:red} how the usual CT reconstruction on~$\R^2$ can be reduced to a reconstruction on~$\T^2$.

Injectivity, reconstruction and certain stability estimates of the $d$-plane Radon transform on~$\T^n$ were proved for distributions by Ilmavirta in~\cite{I15}. The first injectivity result for the geodesic X-ray transform on~$\T^2$ was obtained by Strichartz in~\cite{S82}, and generalized to~$\T^n$ by Abouelaz and Rouvière in~\cite{AR11} if the Fourier transform is~$\ell^1(\Z^n)$. Abouelaz proved uniqueness under the same assumption for the $d$-plane Radon transform in~\cite{A11}. A more general view and references on the Radon transform and the geodesic X-ray transform are given in \cite{S94,H99,PSU14,IM18}.

\subsection{Inversion formulas and Tikhonov regularization}
\label{sec:results}

{
We state here our main theorems regarding the X-ray transform on~$\T^2$. We write the X-ray transform on~$\T^2$ as~$\xrt $ and the formula
\begin{equation}
\xrt f(x,v)=\xrt_vf(x) = \int_0^1 f(x+tv)\dt,\quad x\in\T^2, v \in \Z^2,
\end{equation}
holds true for sufficiently regular functions $f\colon \T^2 \to \C$, including $L^1(\T^2)$. In our proofs, we subsequently apply the Fourier slice theorem of the X-ray transform on~$\T^2$, stated in formula~\eqref{eq:recform}. This fundamental, but simple, property is proved in~\cite{I15}. The exact definition of the X-ray transform of the periodic distributions on $\T^2$ is given in section~\ref{sec:torusCTall}.}

Our first theorem gives new inversion formulas for the X-ray transform. We give two proofs of theorem~\ref{thm:invfor} in section~\ref{sec:rec}. The first one does not rely on the inversion formula of~\cite{I15} whereas the second simpler proof does.

\begin{theorem}
\label{thm:invfor}
Suppose that $f \in L^1(\T^2)$. Let $k \in \Z^2$. If $k,v \neq 0$ and $v \bot k$, then
\begin{equation}
\label{eq:rec1}
\hat{f}(k) = \begin{cases} \int_0^1 \xrt_vf(0,y)\exp(-2 \pi ik_2y)dy,\quad k_2 \neq 0 \\
\int_0^1 \xrt_vf(x,0)\exp(- 2 \pi i k_1x)dx,\quad k_1 \neq 0.\end{cases}\end{equation} If $k = 0$, then \begin{equation}\label{eq:rec0}\hat{f}(k) = \int_0^1 \xrt_{(1,0)}f(0,y)dy = \int_0^1 \xrt_{(0,1)}f(x,0)dx.
\end{equation}
The function~$f$ can be reconstructed by Fourier series~\eqref{eq:Fseries} and formulas~\eqref{eq:rec1} and~\eqref{eq:rec0}. 
\end{theorem}

Let~$Q$ denote the set of all integer directions; {a more detailed description will be given in section \ref{sec:adj}}. We consider a Tikhonov minimization problem: given some data $g\in H^r(\T^2\times Q)$, find
\begin{equation}
\label{eq:reg-min}
\argmin_{f \in H^r(\T^2)}
\left(\aabs{\xrt f-g}_{H^r(\T^2\times Q)}^2 + \alpha\aabs{f}^2_{H^s(\T^2)}\right)
.\end{equation}
Let us define the post-processing operator~$P^s_\alpha$ to be the Fourier multiplier $(1+\alpha\vev{k}^{2s})^{-1}$ and denote by~$\xrt^*$ the adjoint of~$\xrt$. We have the following theorems on regularization. The proofs are given in sections~\ref{sec:adj} and~\ref{sec:regstrat} respectively.

\begin{theorem}
\label{thm:regularization}
Let $r\in\R$, $s\geq r$, and $\alpha>0$.
Suppose $g\in H^r(\T^2\times Q)$.
The unique minimizer~$f$ of the minimization problem~\eqref{eq:reg-min} corresponding to Tikhonov regularization is $f=P^{s-r}_\alpha\xrt^*g\in H^{2s-r}(\T^2) \subset H^r(\T^2)$.
\end{theorem}

\begin{theorem}
\label{thm:strategy}
Suppose $r,t,s,\delta\in\R$ are such that $2s+t\geq r$, $\delta\geq0$, and $s>0$.
We assume that $f\in H^{r+\delta}(\T^2)$ and $g\in H^t(\T^2\times Q)$.

Then our regularized reconstruction operator~$P^s_\alpha\xrt^*$ gives a regularization strategy in the sense that
\begin{equation}
\lim_{\eps\to0}
\sup_{\aabs{g}_{H^t(\T^2\times Q)}\leq\eps}
\aabs{P^s_{\alpha(\eps)}\xrt^*(\xrt f+g)-f}_{H^r(\T^2)}
=
0,
\end{equation}
where $\alpha(\eps)=\sqrt{\eps}$.

Moreover, if $\aabs{g}_{H^t(\T^2\times Q)}\leq\eps$, $0<\delta<2s$ and $0<\alpha\leq 2s/\delta-1$, we have
\begin{equation}
\label{eq:strategy-estimate}
\aabs{P^s_\alpha\xrt^*(\xrt f+g)-f}_{H^r(\T^2)}
\leq
\alpha^{\delta/2s}
C(\delta/2s)
\aabs{f}_{H^{r+\delta}(\T^2)}
+
\frac{\eps}{\alpha}
,
\end{equation}
where $C(x)=x(x^{-1}-1)^{1-x}$.
\end{theorem}

\begin{remark}
If we choose the regularization parameter as $\alpha=\eps^\gamma$, the optimal asymptotic rate of convergence is obtained when $\gamma=(1+\delta/2s)^{-1}$.
Then the terms~$\alpha^{\delta/2s}$ and~$\eps/\alpha$ are of equal order.
\end{remark}

We have also studied mapping properties, the adjoint and the normal operator of~$\xrt$ in propositions~\ref{prop:xrt-continuous} and~\ref{prop:adjoint}; these results are stated in section~\ref{sec:torusCTall}.
For example, it turns out that~$\xrt|_{H^s(\T^2)}$ is unitary to its range for any $s \in \R$ (see proposition~\ref{prop:adjoint}).

\subsection*{Acknowledgements}
J.I. was supported by the Academy of Finland
(decision 295853).
O.K. was supported by Jane and Aatos Erkko Foundation, and Jenny and Antti Wihuri Foundation.
J.R. was supported by the Academy of Finland (CoE in Inverse Problems Research in 2017 \& CoE in Inverse Modelling and Imaging in 2018--2019). The authors are grateful to Sampsa Pursiainen for providing access to a computer hardware used in simulations, Mikko Salo for helpful discussions, and Martin Bright for a helpful MathOverflow discussion related to proposition~\ref{prop:Schanuel}.
O.K. wishes to thank Jari Hyttinen for providing a good working environment for completing this research.
J.R. wishes to thank Matti Lassas and Samuli Siltanen for providing a good working environment at the University of Helsinki during his one-year visit in 2018. We thank the anonymous referees for their helpful comments.

\section{Torus CT method}
\label{sec:torusCTall}

In this section we will lay out the theory of the Torus CT method.
The reconstruction method is based on the Fourier series and properties of the geodesic X-ray transform on~$\T^2$.
There is a natural projection operator from the X-ray transform data of a compactly supported function on the plane to the X-ray transform data on~$\T^2$.
This so called torus-projection operator plays the role of the back-projection operator.
For more details on the geodesic X-ray transform on tori see~\cite{I15} and \cite[Chapter 3]{I17}.


\subsection{The geodesic X-ray transform on $\T^2$}
\label{sec:geodXray}

We define the \textit{flat torus} as the quotient $\T^2 := \R^2 / \Z^2$ and denote the \textit{quotient mapping} $[\cdot]\colon \R^2 \to \T^2$.  A function $f\colon \T^2 \to \C$ can be equivalently thought as a periodic function on~$\R^2$ via the quotient mapping~$[\cdot]$.
We may thus consider a function $f\colon \T^2 \to \C$ as a periodic function  on the whole~$\R^2$.

On closed Riemannian manifolds one defines the geodesic X-ray transform as a collection of line integrals of a function over periodic geodesics.
All geodesics of~$\T^2$ are given by the parametrizations $\gamma_{x,v}(t) := [x+tv], (x,v) \in [0,1]^2 \times \R^2\setminus0$.
The geodesic~$\gamma_{x,v}$ is periodic with period~$1$ (with respect to the parameter~$t$) if and only if  $(x,v) \in [0,1]^2 \times (\Z^2 \setminus 0)$ (see e.g. \cite[Exercise 23]{I17}).
In general, a geodesic is periodic on~$\T^2$ if and only if its direction vector is a multiple of a rational vector.

We denote the space of test functions by $\mathcal{T} := C^\infty(\T^2)$ and the set of all mappings $X \to Y$ by~$Y^X$. We define the \textit{(geodesic) X-ray transform on~$\T^2$} as an operator $\xrt\colon \mathcal{T} \to \R^{\T^2 \times (\Z^2 \setminus 0)}$ by
\begin{equation}
\xrt f(x,v)
:=
\int_0^1 f(\gamma_{x,v}(t))\dt,
\quad
f \in \mathcal{T},
\quad
x \in [0,1]^2 
\quad
v \in \Z^2 \setminus 0.
\end{equation}

A simple calculation shows that that $f \mapsto \xrt f(\cdot,v)$ is a formally self-adjoint operator on~$\mathcal{T}$ for any fixed $v \in \Z^2 \setminus 0$.
We denote the dual space of~$\mathcal{T}$ by~$\mathcal{T}'$, i.e. the space of \textit{distributions}. 
By formal self-adjointness of~$\xrt$, we may define the \textit{X-ray transform on distributions $f \in \mathcal{T}'$} by
\begin{equation}
[\xrt f(\cdot,v)](\eta) := (f,\xrt \eta(\cdot,v)), \quad \eta \in \mathcal{T}
\end{equation}
where $(\cdot,\cdot)$ is the duality pairing.

If $f \in \mathcal{T}'$, then we denote the Fourier coefficients of~$f$ as $\hat{f}(k) := f(e^{-2\pi i k \cdot x}), k \in \Z^2$, and the Fourier series
\begin{equation}
\label{eq:Fseries}
f(x) = \sum_{k\in\Z^2}\hat{f}(k)e^{2\pi ik\cdot x}
\end{equation}
converges in the sense of distributions. We are now ready to recall the inversion formula in~\cite{I15}:

\begin{theorem}[Fourier slice theorem on $\T^2$, Eq. (9) in~\cite{I15}]
\label{thm:ilmavirta}
If $f \in \mathcal{T}'$, then
\begin{equation}
\label{eq:recform} \widehat{\xrt f}(k,v) = \begin{cases}
\hat{f}(k) & k \cdot v = 0 \\
0 & k \cdot v \neq 0.
\end{cases}
\end{equation}
\end{theorem}

Theorem~\ref{thm:ilmavirta} gives a constructive formula~\eqref{eq:recform} for the inverse of the X-ray transform on~$\T^2$.

\subsection{Inversion formula for integrable functions}
\label{sec:rec}

In this section we simplify formula~\eqref{eq:recform} for functions in~$L^1(\T^2)$. It turns out that the dimension of the integral defining $\widehat{\xrt f}(k,v)$ can be decreased by one using a change of coordinates, which enables a computationally faster implementation.

Recall that $f \in \mathcal{T}'$ is in~$L^1(\T^2)$ if there exists a function $\tilde{f} \in L^1(\T^2)$ such that \begin{equation}(f,\phi) = \int_{\T^2} \tilde{f}\phi dm,\quad \forall \phi \in C^\infty(\T^2).\end{equation} It holds that $L^1(\T^2) \subset \mathcal{T}'$. If $\xrt f(\cdot,v) \in L^1(\T^2)$ for some $f \in \mathcal{T}'$, then we simply denote that $\xrt_vf =\xrt f(\cdot,v)$.

We define a family of coordinates which will be used repeatedly in this subsection. Suppose that $v \in \Z^2 \setminus 0$ and $v_1, v_2 \neq 0$. Let $w_m := \frac{m}{\abs{v_1}}v, m\in \Z$, and define the coordinates $\phi_{v,m}$ on~$\T^2$ as
\begin{equation}
\label{eq:coords} 
\phi_{v,m}(a,b) = a\frac{v}{\abs{v_2}} + (0,b) + w_m, \quad a \in [0,\abs{\frac{v_2}{v_1}}), b \in [0,1).
\end{equation}
Notice that $a\frac{v}{\abs{v_2}} + (0,b) = (a\frac{v_1}{\abs{v_2}},\frac{v_2}{\abs{v_2}}(a+b))$ and $w_m = m(\frac{v_1}{\abs{v_1}},\frac{v_2}{\abs{v_1}})$ in the Cartesian coordinates. It easily follows that the Lebesgue measure on~$\T^2$ transforms as $dm = \abs{\frac{v_1}{v_2}} d(a,b)$ where $d(a,b)$ denotes the Lebesgue measure on $X := [0,\abs{\frac{v_2}{v_1}}) \times [0,1)$. 

\begin{remark}
The coordinates~$\phi_{v,m}$ parametrize~$\T^2$ as parallelograms which are located on~$\R^2$.
Moreover, the parallelograms associated with $\phi_{v,m},{m\in\Z}$ are disjoint for a fixed $v \in \Z^2 \setminus 0$ when looked on~$\R^2$.
An example is given on Figure~\ref{fig:coords}.
\end{remark}

\begin{figure}
  \includegraphics[width=0.3\textwidth]{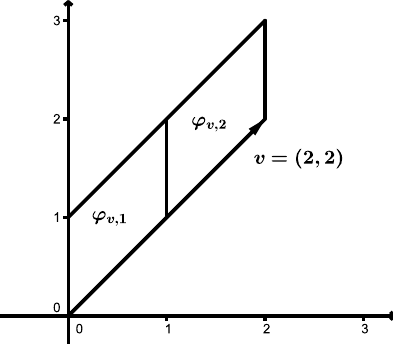}
  \caption{Parallelograms associated to the coordinates $\phi_{v,m}$ when $v = (2,2)$ and $m = 0,1$.}
  \label{fig:coords}
\end{figure}

The next lemma states that the geodesic X-ray transform of an ~$L^1(\T^2)$ function can be defined geodesic-wise for almost every closed geodesic. Furthermore, the X-ray data for any fixed direction is also~$L^1(\T^2)$, and this definition agrees with the distributional definition.

\begin{lemma}
\label{lem:L1range}
Suppose that $v \in \Z^2 \setminus 0$. Then the X-ray transform $\xrt_v\colon L^1(\T^2) \to L^1(\T^2)$ can be defined by the formula
\begin{equation}
\label{eq:L1formula}
\xrt_vf(p):= \int_0^1 f(p+tv) dt\quad \text{for a.e. $p \in\T^2$}.
\end{equation}
Moreover, we have:
\begin{enumerate}
\item This definition coincides with the distributional definition; for every $f \in L^1(\T^2)$ and $g \in L^\infty(\T^2)$ it holds that $(\xrt_v f,g) = (f,\xrt_v g)$.
\item $\xrt_v\colon L^1(\T^2) \to L^1(\T^2)$ is Lipschitz continuous with Lipschitz constant~$1$.
\item \label{itm:prop3} For almost every $p \in \T^2$ and every $v \in \Z^2 \setminus 0$ and  $t \in \R$ it holds that $\xrt_vf(p) = \xrt_vf(p+tv)$.
\end{enumerate}
\end{lemma}

\begin{proof}
This follows from Fubini's theorem and straightforward calculations using the coordinates~$\phi_{v,m}$. We omit the details.
\end{proof}

We will give two proofs for theorem~\ref{thm:invfor}.
The first proof is based on the assumption that $f \in L^1(\T^2)$ and straightforward computation of the Fourier coefficients.
The first proof proves the injectivity of the X-ray transform on~$\T^2$ for~$L^1(\T^2)$ functions independently of~\cite{I15}.
The second proof is based on formula~\eqref{eq:recform} and the assumption that $\xrt f(\cdot,v) \in L^1(\T^2)$.
Both of the proofs involve the coordinates~$\phi_{v,k}$.

\begin{proof}[First proof of theorem~\ref{thm:invfor}]
Recall that
\begin{equation}
\label{eq:fourcoef}
\hat{f}(k) = \int_0^1 \int_0^1 f(x,y)\exp(-2\pi i k \cdot (x,y))dxdy.
\end{equation}
If $k_1 = 0$ or $k_2 = 0$, then formulas~\eqref{eq:rec1} and~\eqref{eq:rec0} follow trivially from~\eqref{eq:fourcoef}. 

\textit{The case $k_1,k_2 \neq 0$.} We can use the coordinates $\phi_{v,m}, m \in \Z$, defined by formula~\eqref{eq:coords}. Using these coordinates we can calculate
\begin{equation}\hat{f}(k) = \int_0^1\int_0^{\abs{v_2/v_1}} f(\phi_{v,m}(a,b))\exp\left(-2 \pi i k \cdot \left(\frac{a}{\abs{v_2}}v+(0,b)+w_m\right)\right)\abs{\frac{v_1}{v_2}}dadb.\end{equation} Notice that $k \cdot \left(\frac{a}{\abs{v_2}}v+(0,b)+w_m\right) = k_2b$ since $v \cdot k = w_m \cdot k = 0$.

Hence, we have
\begin{equation}
\label{eq:fcoefs}
\hat{f}(k) = \abs{\frac{v_1}{v_2}}\int_0^1 \int_0^{\abs{v_2/v_1}} f\left(a\frac{v}{\abs{v_2}} + (0,b) +w_m\right)da \exp(- 2 \pi ik_2b)db.
\end{equation}
We sum formula~\eqref{eq:fcoefs} for the values $m = 0,\dots,\abs{v_1}-1$, which gives
\begin{equation}\abs{v_1} \hat{f}(k) = \abs{v_1}\int_0^1 \xrt_vf(0,y) \exp(-2\pi ik_2y)dy.
\end{equation}
This completes the proof.
\end{proof}

We will next prove a more general version of theorem~\ref{thm:invfor}.

\begin{theorem}
\label{thm:invforL1data}
Suppose that $f \in \mathcal{T}'$ and $\xrt f(\cdot,v) \in L^1(\T^2)$ for any $v\in \Z^2 \setminus 0$. Then formulas~\eqref{eq:rec1} and~\eqref{eq:rec0} are true.
\end{theorem}

\begin{proof}
We only show how to argue if $k_1, k_2 \neq 0$ since the other special cases are trivial. Recall that the inversion formula~\eqref{eq:recform} states that $\widehat{\xrt_vf}(k) = \hat{f}(k)$ for any $v \in \Z^2 \setminus 0$ such that $k \bot v$. We apply the coordinates~$\phi_{v,0}$. 

Using Fubini's theorem and calculations similar to the first proof of theorem~\ref{thm:invfor}, we get
\begin{equation}\widehat{\xrt_vf}(k) = \abs{\frac{v_1}{v_2}}\int_0^1\int_0^{\abs{v_2/v_1}} \xrt_v f(\phi_{v,0}(a,b))da\exp(-2 \pi i k_2 b)db.\end{equation} Now, formula~\eqref{eq:rec1} follows from property~\eqref{itm:prop3} of lemma~\ref{lem:L1range}. This completes the proof.
\end{proof}

\begin{proof}[Second proof of theorem~\ref{thm:invfor}] Lemma~\ref{lem:L1range} implies that $\xrt_vf \in L^1(\T^2)$ if $f \in L^1(\T^2)$. Hence, theorem~\ref{thm:invforL1data} implies the inversion formulas.
\end{proof}

\subsection{The torus-projection operator}
\label{sec:red}

We denote the X-ray transform of $f\colon \R^2 \to \C$ by
$\mathcal{R}_vf(p)$ for any $(p,v) \in \R^2 \times S^1$.
We parametrize the lines of the plane so that
\begin{equation}
\mathcal{R}_vf(p)
=
\int_\R f(p+tv)dt.
\end{equation}
Suppose that~$f$ is a compactly supported function on~$\R^2$.
We may then consider~$f$ as a function defined on~$\T^2$ after rescaling and periodizing.
Let us denote the periodic extension of~$f$ into~$\T^2$ by the same symbol~$f$.

Suppose further that $f \in C(\T^2)$.
As described in \cite[Lemma 3.1]{I17}, for any $p \in \T^2$ and $v \in \Z^2 \setminus 0$ one can write $\xrt_vf(p)$ as a finite sum of terms $\mathcal{R}_{v/\abs{v}}f(p_i)$, $i =1,\dots,m$, {and rescaling data by $\abs{v}^{-1}$.}
One simply has to write any periodic geodesic~$\gamma$ of~$\T^2$ as a finite disjoint union of line segments that are supported in $[0,1) \times [0,1)$ and travel from the boundary to the boundary in the fundamental domain of~$\T^2$.
However, such unions are tedious to write down rigorously.
This procedure defines the torus-projection operator $\mathcal{R}f \mapsto \xrt f$ for compactly supported continuous functions $f\colon\R^2\to\C$. {We give example on Figure~\ref{fig:torusproj} which illustrates this procedure.}
For further details, see \cite[Chapter 3]{I17}.
See also the description of our numerical implementation in section~\ref{sec:t2forward}. {We have not studied mapping properties of the torus-projection operator, but we would like to see it studied thoroughly in the future. This would give new and important information about the connection between Euclidean and periodic X-ray transforms.}

\begin{figure}
\includegraphics[width=0.4\textwidth]{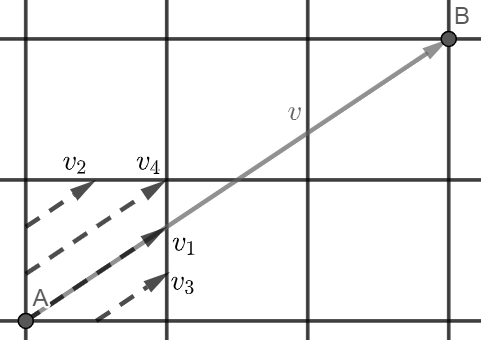}
\caption{{The closed geodesic on $\T^2$ from $A$ to $B$ with $[A] = [B]$, determined by the vector $v = (3,2)$ and starting from $A$, can be presented as the disjoint union of the geodesic segments $v_1$, $v_2$, $v_3$ and $v_4$. If one knows the X-ray transform over these four line segments on $\R^2$, then one can calculate the X-ray transform over the corresponding geodesic on $\T^2$.}}
\label{fig:torusproj}
\end{figure}

\subsection{Sobolev spaces, adjoint, normal operator and regularization}
\label{sec:adj}

Let $Q\subset\Z^2$ be such that every nonzero $v\in\Z^2$ is an integer multiple of a unique element in~$Q$.
We can simply take~$Q$ to be the set of those vectors $(a,b)$ for which~$a$ and~$b$ are coprime with $a>0$ and $b\neq0$ and the vectors $(0,1)$ and $(1,0)$.
The set~$Q$ is the set of all periodic directions on the torus, with all multiple counts removed.
This set can be naturally identified with the projective space~$\mathbb P^1$ defined later.

The X-ray transform we study takes a function on~$\T^2$ to a function on $\T^2\times Q$.
To set things up properly, we need to define function spaces and norms on both sides.
On~$\T^2$, we use the standard Sobolev scale of spaces~$H^s(\T^2)$ with the norms
\begin{equation}
\aabs{f}_{H^s(\T^2)}^2
=
\sum_{k\in\Z^2}
\vev{k}^{2s}\abs{\hat f(k)}^2,
\end{equation}
where $\vev{k}=(1+\abs{k}^2)^{1/2}$ as usual.
On $\T^2\times Q$, we define the spaces $H^{s}(\T^2\times Q)$ to be the set of functions $g\colon\T^2\times Q\to\C$ for which
\begin{enumerate}[(i)]
\item  $g(\cdot,v)\in H^s(\T^2)$ for every $v\in Q$;
\item {the average of every $g(\cdot,v)$ over $\T^2$ is the same; and}
\item the norm
\begin{equation}
\aabs{g}_{H^{s}(\T^2\times Q)}^2
=
\abs{\hat g(0,0)}^2+
\sum_{k\in\Z^2\setminus0}
\sum_{v\in Q}
\vev{k}^{2s}\abs{\hat g(k,v)}^2
\end{equation}
is finite.
We set $v=0$ for the Fourier term $k=0$ to emphasize that it is the same for every $v\in Q$.
We remind the reader that $0\notin Q$.
\end{enumerate}

We emphasize that the regularity parameter~$s$ can be any real number in the theory presented in this section.
By setting $s=0$ one obtains a theory in $H^0=L^2$.
We point out that the spaces considered here are different from~\cite{I15}.

\begin{remark}
The norm of $g\in H^s(\T^2\times Q)$ is essentially an~$\ell^2(Q)$ norm of the~$H^s(\T^2)$ norms of the functions $g(\cdot,v)$.
This~$\ell^2$ can be replaced with any~$\ell^p$ without much effect to the theory, as the different functions in the family indexed by~$Q$ have disjointly supported Fourier series apart from the origin.
The case $p=\infty$ is particularly convenient because then special considerations are not needed at $k=0$.
We choose $p=2$ to stay in a Hilbert space setting.
\end{remark}

We denote $v^\perp=(-v_2,v_1)$ for any $v=(v_1,v_2)\in\Z^2$.
For $v\in\Z^2\setminus0$, we denote by~$\hat v$ the unique point in~$Q$ that is parallel to~$v$.
We can define~$\hat 0$ to be any point in~$Q$; this choice will not matter.
To keep notation neater, we will write~$\hat v^\perp$ instead of~$\widehat{v^\perp}$.

\begin{proposition}
\label{prop:xrt-continuous}
The X-ray transform is continuous $H^s(\T^2)\to H^s(\T^2\times Q)$ for any $s\in\R$.
\end{proposition}

\begin{proof}
For any $v\in Q$, the Fourier transform of function $\xrt f(\cdot,v)$ is supported on the line~$v^\perp\Z$ by theorem~\ref{thm:ilmavirta}.
In fact, it is the restriction of~$\hat f$ to this line.
It then follows easily from the definition of the Sobolev norm on the Fourier side that $\xrt f(\cdot,v)\in H^s(\T^2)$ whenever $f\in H^s(\T^2)$.

It follows from the same theorem that $\widehat{\xrt f}(0,v)=\hat f(0)$ for all $v\in Q$, and so all the averages agree as required.

Since~$\Z^2$ is a disjoint union of the origin and the punctured lines $v\Z\setminus0$ with $v\in Q$, one can easily verify that $\aabs{\xrt f}_{H^s(\T^2\times Q)}=\aabs{f}_{H^s(\T^2)}$.
\end{proof}

\begin{proposition}
\label{prop:adjoint}
Fix any $s\in\R$.
The adjoint of $\xrt\colon H^s(\T^2)\to H^s(\T^2\times Q)$ is $\xrt^*\colon H^s(\T^2\times Q)\to H^s(\T^2)$ given by
\begin{equation}
\widehat{\xrt^*g}(k)
=
\hat g(k,\hat k^\perp).
\end{equation}
The normal operator $\xrt^*\xrt\colon H^s(\T^2)\to H^s(\T^2)$ is the identity, so that~$\xrt$ is unitary to its range.
\end{proposition}

\begin{remark}
We emphasize that there is a striking difference with the usual Euclidean X-ray transform, where the normal operator is a convolution.
In our setup the X-ray transform is directly inverted by its adjoint operator without any filtering or post-processing.
\end{remark}

\begin{proof}[Proof of proposition~\ref{prop:adjoint}]
Let us take any two functions $f\in H^s(\T^2)$ and $g\in H^s(\T^2\times Q)$. We denote the complex conjugate of $z \in \C$ as~$z^*$.
Theorem~\ref{thm:ilmavirta} shows that $\widehat{\xrt f}(k,v)=\hat f(k)\delta_{0,k\cdot v}$, and so the~$H^s$ inner products satisfy
\begin{equation}
\begin{split}
\iip{\xrt f}{g}
&=
\widehat{\xrt f}(0,0)^*\hat g(0,0)
+
\sum_{k\in\Z^2\setminus0}
\sum_{v\in Q}
\vev{k}^{2s}
\widehat{\xrt f}(k,v)^*\hat g(k,v)
\\&=
\sum_{k\in\Z^2}
\vev{k}^{2s}
\hat f(k)^*\hat g(k,\hat k^\perp)
\\&=
\iip{f}{\xrt^*g}.
\end{split}
\end{equation}
Therefore, the operator~$\xrt^*$ defined above is the adjoint of~$\xrt$.

It follows directly from the formula of theorem~\ref{thm:ilmavirta} that~$\xrt^*$ is a left inverse of~$\xrt$.
\end{proof}

\begin{remark}
The X-ray transform or its normal operator have no effect on regularity.
In the usual formulation, the normal operator does increase the smoothness index~$s$, but when everything is set up on~$\T^2$ the operators leave the regularity level intact.
\end{remark}

We now turn to regularized inversion, and solve the Tikhonov minimization problem~\eqref{eq:reg-min}. We will make use of the post-processing operator~$P^s_\alpha$, which is the Fourier multiplier $(1+\alpha\vev{k}^{2s})^{-1}$.
It is evident that~$P^s_\alpha$ maps continuously $H^r(\T^2)\to H^{r+2s}(\T^2)$ for any $s,r\in\R$.

\begin{proof}[Proof of theorem~\ref{thm:regularization}]
We begin with expanding the norms along lines given by~$Q$ on the Fourier side.
We have
\begin{equation}
\aabs{f}^2_{H^s(\T^2)}
=
\abs{\hat f(0)}^2
+
\sum_{v\in Q}\sum_{p\in\Z\setminus0}
\vev{pv}^{2s}\abs{\hat f(p v^\perp)}^2
\end{equation}
and
\begin{equation}
\aabs{\xrt f-g}_{H^r(\T^2\times Q)}^2
=
\abs{\widehat{\xrt f}(0,0)-\hat g(0,0)}^2
+
\sum_{v\in Q}A_v,
\end{equation}
where
\begin{equation}
\begin{split}
A_v
&=
\sum_{p\in\Z\setminus0}
\vev{pv}^{2r}\abs{\widehat{\xrt f}(p v^\perp,v)-\hat g(pv^\perp,v)}^2
\\&\quad+
\sum_{w\in\Z^2\setminus v\Z}
\vev{w}^{2r}\abs{\widehat{\xrt f}(w,v)-\hat g(w,v)}^2.
\end{split}
\end{equation}
Each $\widehat{\xrt f}(w,v)$ vanishes in the last sum by theorem~\ref{thm:ilmavirta}.
Therefore, the second sum of~$A_v$ is independent of~$f$ and can be left out of the minimization problem.
Furthermore, $\widehat{\xrt f}(p v^\perp,v)=\hat f(pv^\perp)$.

Thus, we are left with minimizing
\begin{equation}
\begin{split}
&\abs{\widehat{\xrt f}(0,0)-\hat g(0,0)}^2
+
\alpha\abs{\hat f(0)}^2
\\&+
\sum_{v\in Q}\sum_{p\in\Z\setminus0}
\left(
\vev{pv}^{2r}\abs{\hat f(p v^\perp)-\hat g(pv^\perp,v)}^2
+
\alpha\vev{pv}^{2s}\abs{\hat f(p v^\perp)}^2
\right).
\end{split}
\end{equation}
The notation introduced above allows us to rewrite the minimized quantity as
\begin{equation}
\label{eq:reg-sum}
\sum_{k\in\Z^2}
\vev{k}^{2r}
\left(
\abs{\hat f(k)-\hat g(k,\hat k^\perp)}^2
+
\alpha\vev{k}^{2(s-r)}\abs{\hat f(k)}^2
\right),
\end{equation}
and this can be minimized explicitly.

It suffices to choose each~$\hat f(k)$ so that the term in the parentheses of~\eqref{eq:reg-sum} is minimized.
A straightforward computation shows that the minimal~$\hat f(k)$ is
\begin{equation}
\hat f(k)
=
(1+\alpha\vev{k}^{2(s-r)})^{-1}
\hat g(k,\hat k^\perp).
\end{equation}
That is, the minimizer we sought is $f=P_\alpha^{s-r}\xrt^*g$.
Finally, by the mapping properties of~$P_\alpha^{s-r}$ and~$\xrt^*$ we have $f\in H^{2s-r}(\T^2)$.
This implies that~$f$ is in the correct space~$H^r(\T^2)$ since we assumed $s\geq r$.
\end{proof}

\begin{remark}
Choosing $r=0$ and $s=1$, we reconstruct a function in~$L^2$ with an~$H^1$ penalty term.
If we want the penalty to be the~$L^2$ norm of the gradient without the~$L^2$ norm of the function, the Fourier multiplier in the penalty term is changed from~$\vev{k}^2$ to~$\abs{k}^2$.
This corresponds to changing the Sobolev norm to a homogeneous Sobolev norm.
Such changes lead to similar results but with slightly different postprocessing operator.
\end{remark}

\subsection{Regularization strategy}
\label{sec:regstrat}

We define the concept of a regularization strategy according to \cite{EHN96,K11}. Let~$X$ and~$Y$ be subsets of Banach spaces and $F\colon X \to Y$ a continuous mapping. A family of continuous maps $\mathcal{R}_\alpha\colon Y \to X$ with $\alpha \in (0,\alpha_0]$ is called a \textit{regularization strategy} if $\lim_{\alpha \to 0} \mathcal{R}_\alpha (F(x)) = x$ for every $x \in X$. A choice of regularization parameter $\alpha(\epsilon)$ with $\lim_{\epsilon\to0}\alpha(\epsilon) = 0$ is called \textit{admissible} if 
\begin{equation}
\lim_{\epsilon \to 0}
\sup_{y \in Y} \left\{ \norm{\mathcal{R}_{\alpha(\epsilon)}y -x}_X \,;\, \norm{y-F(x)}_Y \leq \epsilon\,\right\}
= 0
\end{equation}
holds for every $x \in X$. Regularization strategies have been found for other inverse problems including, for example, electrical impedance tomography (EIT)~\cite{KLMS09} and inverse problem for the $1+1$ dimensional wave equation \cite{KLO16, KLO18}.

We will next prove that the regularized inversion operator~$P_\alpha^s\xrt^*$ obtained in theorem~\ref{thm:regularization} actually provides an admissible regularization strategy with a quantitative stability estimate.

\begin{proof}[Proof of theorem~\ref{thm:strategy}]
Using proposition~\ref{prop:adjoint}, we write
\begin{equation}
P^s_\alpha\xrt^*(\xrt f+g)-f
=
(P^s_\alpha-\id)f+P^s_\alpha\xrt^* g
\end{equation}
and aim to estimate these two terms.
In this proof, we denote the norm of~$H^r(\T^2)$ simply by~$\aabs{\cdot}_r$.

Since $\aabs{g}_{H^t(\T^2\times Q)}\leq\eps$ and $\aabs{\xrt^*}=\aabs{\xrt}=1$, we have $\aabs{\xrt^*g}_t\leq\eps$.
Applying the definitions of the norms and the operator~$P^s_\alpha$, we find
\begin{equation}
\aabs{P^s_\alpha \xrt^* g}_r^2
\leq
\eps^2
\sup_{k\in\Z^2}
(1+\alpha\vev{k}^{2s})^{-2}\vev{k}^{2r-2t}.
\end{equation}
Estimating $1+\alpha\vev{k}^{2s}\geq\alpha\vev{k}^{2s}$ and using $-4s+2r-2t\leq0$ shows that the supremum is at most~$\alpha^{-2}$.
Therefore
\begin{equation}
\aabs{P^s_\alpha \xrt^* g}_r
\leq
\alpha^{-1}\eps,
\label{eq:regestimate}
\end{equation}
which converges to zero as $\eps\to0$ with $\alpha=\sqrt{\eps}$.

A calculation shows that $P_\alpha^s -\id = -\frac{\alpha\vev{k}^{2s}}{1+\alpha\vev{k}^{2s}}$ as a Fourier multiplier. Unfortunately, this implies that
\begin{equation}
\label{eq:rr-norm}
\aabs{P^s_\alpha-\id}_{H^r(\T^2)\to H^r(\T^2)}
= \sup_{k\in\Z^2}
\frac{\alpha\vev{k}^{2s}}{1+\alpha\vev{k}^{2s}}
=
1
\end{equation}
whenever $s>0$ and $\alpha>0$.
Therefore, a uniform estimate is impossible when $\delta=0$, but it follows from the dominated convergence theorem that $\aabs{(P^s_\alpha-\id)f}_r^2\to0$ as $\alpha\to0$ when $f\in H^r(\T^2)$.
The first claim of the theorem follows.

If $\delta>0$, the additional regularity of~$f$ can be used to our advantage.
It follows from the definitions of the norms that \begin{equation}
\norm{(P_\alpha^s - \id)f}_r^2 \leq \left(\sup_{k \in \Z^2} \left(\frac{\alpha\vev{k}^{2s}}{1+\alpha\vev{k}^{2s}}\right)^2\vev{k}^{-2\delta} \right)\norm{f}_{r+\delta}^2,
\end{equation}
and thus
\begin{equation}
\label{eq:ur-norm}
\aabs{P^s_\alpha-\id}_{H^{r+\delta}(\T^2)\to H^r(\T^2)}
=
\sup_{k\in\Z^2}
\frac{\alpha\vev{k}^{2s-\delta}}{1+\alpha\vev{k}^{2s}}
.
\end{equation}
Estimating this norm is crucial for the proof.

The supremum of~\eqref{eq:ur-norm} can be studied using the function $F\colon(0,\infty)\to(0,\infty)$ given by
\begin{equation}
F(x)
=
\frac{\alpha x^{2s-\delta}}{1+\alpha x^{2s}}.
\end{equation}
Simple calculus shows that if $2s>\delta$, then the maximum is attained at $x^{2s}=\alpha^{-1}(2s/\delta-1)$ and the maximal value on $(0,\infty)$ is
\begin{equation}
\alpha^{\delta/2s}
\frac{\delta}{2s}
\left(\frac{2s}{\delta}-1\right)^{1-\delta/2s}
.
\end{equation}
We are interested in the maximum of~$F$ on $[1,\infty)$.
If $2s/\delta-1<\alpha$, then the maximum is reached at $x\in(0,1)$, and so the maximum on the relevant interval is $F(1)=\alpha/(1+\alpha)$.
(One can also verify that the two maxima coincide when $2s/\delta-1=\alpha$, as they should.)
We assumed that $2s/\delta  > 1$, so $\alpha \in (0, 2s/\delta -1]$ for small enough~$\alpha$.

For $\alpha\leq 2s/\delta-1$, the maximum value of~$F$ is
\begin{equation}
\alpha^{\delta/2s}\frac{\delta}{2s}\left(\frac{2s}{\delta}-1\right)^{1-\delta/2s}
=
\alpha^{\delta/2s}C(\delta/2s).
\end{equation}
We conclude that
\begin{equation}
\aabs{P^s_\alpha-\id}_{H^{r+\delta}(\T^2)\to H^r(\T^2)}
\leq
\alpha^{\delta/2s}C(\delta/2s),
\end{equation}
and so
\begin{equation}
\aabs{(P^s_\alpha-\id)f}_r
\leq
\alpha^{\delta/2s}C(\delta/2s)\aabs{f}_{r+\delta}.
\label{eq:opnorm}
\end{equation}
Estimate~\eqref{eq:strategy-estimate} now follows easily from estimates~\eqref{eq:regestimate} and~\eqref{eq:opnorm}.
\end{proof}

If~$\alpha$ is bigger than assumed in the proof, then we may use the simpler estimate $F(x)\leq\alpha/(1+\alpha)\leq\alpha$ for all $x\geq1$, which would lead to replacing $\alpha^{\delta/2s}C(\delta/2s)$ in estimate~\eqref{eq:strategy-estimate} by simply~$\alpha$.
However, we are only interested in the limit of small~$\alpha$.

We point out that $C(\delta/2s)\to1$ and $\alpha^{\delta/2s}\to1$ when $\delta\to0$, matching the norm in the limiting case of~\eqref{eq:rr-norm}.

The noise~$g$ in theorem~\ref{thm:strategy} can be in any function space so that~$\xrt^*g$, the reconstruction from pure noise, is in a suitable Sobolev space.

\subsection{Numerical and asymptotic analysis for discretized problem}
\label{ssec:dft}

In this section, we consider questions arising from discrete practice. We analyze errors caused by a discretization of data in section~\ref{sec:discrete}. In section~\ref{sec:Schanuel}, we study how to choose a minimal set of X-ray directions in order to reconstruct all Fourier coefficients of a phantom in a given box.

Another source of errors in practice comes from the fact that we can only calculate finitely many coefficients of the Fourier series.
The error caused by the cutoff of the Fourier series can be estimated with knowledge of asymptotic behavior of the Fourier coefficients.
We do not consider this matter here further since it is a general question about convergence rates of Fourier series.

\subsubsection{On convergence rates for discretization}
\label{sec:discrete}

Let $\mathbf{f} \in \C^N$ be written as $\mathbf{f} = (\mathbf{f}_0,\dots,\mathbf{f}_{N-1})$. We define the discrete Fourier transform (DFT) of~$\mathbf{f}$ by \begin{equation}\mathrm{DFT}(\mathbf{f})_k := \frac{1}{N}\sum_{l=0}^{N-1} \mathbf{f}_l \exp(-2\pi i k l / N),\quad k = 0,\dots,N-1.\end{equation} The following corollary of theorem~\ref{thm:invforL1data} discretizes the inverse problem and reduces it to calculations of 1-dimensional DFTs. It is elementary and included here for completeness.

\begin{corollary}
\label{cor:dft}
Let $f \in \mathcal{T}'$, $\xrt_vf \in L^1(\T^2)$, $k \in \Z^2 \setminus 0$. Denote $g_v(y) := \xrt_vf(0,y)$ and $h_v(x) := \xrt_vf(x,0)$. \begin{enumerate}
\item \label{item:dft1} If $v \bot k$, then $\hat{f}(0,0) =  \hat{g}_{(1,0)}(0) = \hat{h}_{(0,1)}(0)$ and \[\hat{f}(k_1,k_2) = \begin{cases} \hat{g}_v(k_2),\quad k_2 \neq 0 \\
 \hat{h}_v(k_1),\quad k_1 \neq 0.\end{cases}\]
\item \label{item:dft2} (Left-point rule and DFT) Let $N \in \Z_+$. We denote $\textbf{g}_l = g_v(l/N)$ and $\textbf{h}_l = h_v(l/N)$ for $l = 0,\dots,N-1$. If~$\xrt_vf$ is Riemann integrable along vertical and horizontal lines, then
\[
\mathrm{DFT}(\textbf{g})_{k_2} \to \hat{g}_v(k_2) \text{ as } N \to \infty.
\]
Moreover, if $\xrt_vf \in C^1(\T^2)$, then $\abs{\hat{g}_v(k_2) - \mathrm{DFT}(\textbf{g})_{k_2}} \leq C_{f,k_2}/N$ where $C_{f,k_2} > 0$ does not depend on~$N$. Similar statements hold for~$h_v$ as well.

\item (Mid-point rule and DFT) Let $N \in \Z_+$. We denote $\textbf{g}_l = g_v(l/N+1/2N)$ and $\textbf{h}_l = h_v(l/N + 1/2N)$ for $l = 0,\dots,N-1$. If $\xrt_vf$ is Riemann integrable along vertical and horizontal lines, then
\[
\exp(-\pi i k_2/N)\mathrm{DFT}(\textbf{g})_{k_2} \to \hat{g}_v(k_2) \text{ as } N \to \infty.
\]
Moreover, if $\xrt_vf \in C^2(\T^2)$, then $\abs{\hat{g}_v(k_2) - \exp(-\pi i k_2/N)\mathrm{DFT}(\textbf{g})_{k_2}} \leq C_{f,k_2}/N^2$ where $C_{f,k_2} > 0$ does not depend on~$N$. Similar statements hold for~$h_v$ as well.
\end{enumerate}
\end{corollary}
\begin{proof}
The statement (1) is a rephrased version of theorem~\ref{thm:invforL1data}. We only prove the statement (3). The proof of the statement (2) is similar and thus omited. Let $N \in \Z_+$ be fixed. By the definition of the DFT
\begin{equation}
\begin{split}
&\exp(-\pi i k_2/N)\mathrm{DFT}(\mathbf{g})_{k_2} \\ &\,\,= \frac{1}{N} \sum_{l=0}^{N-1} \mathbf{g}_l \exp(-\pi i k_2/N)\exp(-2\pi i k_2 l/N) \\
&\,\,= \frac{1}{N} \sum_{l=0}^{N-1} g_v(l/N+1/2N) \exp(-2\pi i k_2 (l/N + 1/2N)).
\end{split}
\end{equation}
The statement follows since this the mid-point approximation of $\hat{g}_v(k_2)$. The convergence rate is just a standard result on the mid-point rule (see e.g. \cite{DR07}).
\end{proof}

\subsubsection{Choosing directions for X-ray data.}
\label{sec:Schanuel}

Let us define the set
\begin{equation}
A_N
:=
\{ v \in \Z^2 \setminus 0 \,;\,  v \in k^\bot\text{ for some }k \in \mathcal{Z}_N\}
\end{equation}
where $\mathcal{Z}_N = [-N,N]^2 \cap \Z^2$. {Based on theorem~\ref{thm:ilmavirta}, the data $(\xrt f(\cdot,v))_{v \in A_N}$ determines $(\hat{f}(k))_{k \in \mathcal{Z}_N}$. Thus, we define}
\begin{equation}\phi(N)
:=
\min \{\abs{B} \,;\, B\subset A_N, (\xrt f(\cdot,v))_{v\in B} \text{ determines $(\hat{f}(k))_{k \in \mathcal{Z}_N}$}\,\}.
\end{equation}
Define the set $V_N := X_+ \cup X_- \cup \{(1,0),(0,1)\}$ where \begin{equation}\begin{split}X_+ &= \{\, (v_1,v_2) \in \mathcal{Z}_N\setminus 0 \,;\, \gcd(v_1,v_2) = 1, v_1,v_2 \geq 1\,\}, \\ 
X_- &= \{\,(-v_1,v_2) \,;\, v \in X_+\,\}.\end{split}\end{equation} Now, it is an elementary observation that the data $(\xrt f(\cdot,v))_{v\in V_N}$ determines $(\hat{f}(k))_{k \in \mathcal{Z}_N}$ and $\abs{V_N} = \phi(N)$.

We then turn to studying the asymptotic behavior of~$\phi(N)$. We denote by $\mathbb{P}^1 := \mathbb{P}^1(\mathbb{Q})$ the collection of equivalence classes $(a:b)$, $(a, b) \in \Z^2 \setminus 0$, such that $(x,y) \in (a:b)$ if and only if $c(x,y) = (a,b)$ for some $c \neq 0$ and $(x,y) \in \Z^2 \setminus 0$. The \textit{height} is defined as $H(a:b) := \max\{\abs{a},\abs{b}\}$ using the unique representative (up to a sign) of $(a:b)$ with $\gcd(a,b) = 1$. One of the simplest special cases of Schanuel's theorem \cite[Theorem 1]{S64} states that \begin{equation}\abs{\{\,(a:b) \in \mathbb{P}^1\,;\, H(a:b) \leq N\,\}} = \frac{2}{\zeta(2)}N^2 + O(N\log N)\end{equation} as $N \to \infty$. More detailed exposition is given in the book of Serre \cite[Chapter 2.5]{S97}. 

We conclude with the following proposition.

\begin{proposition}
\label{prop:Schanuel}
It holds that $\phi(N) = \frac{2}{\zeta(2)}N^2+ O(N\log N)$.
\end{proposition}
\begin{proof}
If we want to reconstruct~$\hat{f}(k)$, then we need at least one $v \in k^\bot$ by theorem~\ref{thm:ilmavirta} and, on the other hand, just one $v \in k^\bot$ is enough. It follows from the definition of height that \begin{equation}\phi(N) = \abs{\{\,(a:b) \in \mathbb{P}^1\,;\, H(a:b) \leq N\,\}}.\end{equation} The estimate follows now from Schanuel's theorem.
\end{proof}

\begin{remark}
\label{remark30percent}
The trivial estimate for directions needed in reconstruction of the Fourier coefficients $(\hat{f}(k))_{k\in\mathcal{Z}_N}$ would be $\phi(N) \leq (2N+1)^2$. In comparison, proposition~\ref{prop:Schanuel} implies that one needs to use asymptotically about $3/\pi^2 \approx 30$ \% of the data $(\xrt f(\cdot,v))_{v\in\mathcal{Z}_N}$.
\end{remark}

\section{Computational forward and inverse models}
\label{sec:implementation}

We have implemented two forward models for the X-ray transform on~$\T^2$. The first forward model is based on direct integration over periodic geodesics on~$\T^2$ (two different numerical integration schemes are implemented), and the second forward model on the usual Radon transform and the torus-projection operator. The regularized inverse model is based on theorems~\ref{thm:invfor} and~\ref{thm:regularization}.

\subsection{Computational forward models}
\subsubsection{Forward models on the torus}
\label{sssec:forwardontorus}

We have two different numerical integration schemes for the forward integration. The first one is analytical integration of a phantom which is discretized into square pixels of equal size. In this case, the forward operator, denoted by~$\mathcal{A}_1$, is
\begin{equation}
\label{eq:forwardmodel1int}
     \mathcal{A}_1f(x,v) 
    := \frac{1}{\abs{v}}\sum_{i=1}^{N} d_i f_i \approx 
     \int_0^1 f(x+tv) dt =  \xrt f(x,v)
\end{equation}
where~$d_i$ is the length of the geodesic~$\gamma_{x,v}$ and~$f_i$ is the value of the discretized phantom in the $i$'th pixel, and~$N$ is the size of the grid. The lengths~$d_i$ are calculated by solving the intersection points of the line $\{\,x+tv\,;\, t \in [0,1]\,\}$ and the edges of the pixels when the pixels are periodically extended to~$\R^2$.

In the second one, the integral is based on the use of global adaptive quadrature~\cite{S08} which is implemented into Matlab's \texttt{integral} function. In this case, a phantom is given in an analytical form. We denote this forward model by~$\mathcal{A}_2$.

\subsubsection{Forward model using the torus-projection and Radon data}
\label{sec:t2forward}

This forward model corresponds to converting conventional X-ray data sets on~$\R^2$ into X-ray data sets on~$\T^2$. The forward model has two steps. The first step is to calculate Radon transform data using Matlab's \texttt{radon} function. The second step is to calculate the torus-projection (see Section~\ref{sec:red}) of the Radon data. The directions for the Radon transform are chosen so that they contain all directions generated by integer vectors (see Section~\ref{sec:Schanuel}). 

The X-ray beams on the \texttt{radon} function are parametrized by the distance between the line of a X-ray beam and the center of a domain~$O$, and the angle of a X-ray beam measured from the $y$-axis into the counterclockwise direction. We denote simply that $Rf(v) = \text{\texttt{radon}}(f,\alpha_{v},M)$ where~$\alpha_{v}$ is the angle defined above and~$M$ is the number of X-rays taken into direction~$v$. We index the rays as $k = 1,\dots, M$. Further, denote the distances of rays to~$O$ by~$c_{k,v}$ and the projection values with the respective rays by~$Rf(v)_k$. 

We split each geodesic~$\gamma_{x,v}$ into segments in which it travels from the boundary to the boundary when looked at the fundamental domain $[0,1] \times [0,1]$ of~$\T^2$. Let~$d_{x,v,i}$ be the distance of the $i$'th segment of the geodesic~$\gamma_{x,v}$ and~$O$, and~$N$ the number of distinct segments. Finally, we can define the forward model~$\mathcal{A}_{\T^2}$ as
\begin{equation}
\label{eq:T2forward}
\mathcal{A}_{\T^2}f(x,v) = \frac{1}{\abs{v}}\sum_{i=1}^N \left(w_{1, i}Rf(v)_{k_{1,i}} + w_{2,i}Rf(v)_{k_{2,i}}\right)
\end{equation}
where
\[
\begin{split}
k_{1,i} = \argmin_{k\in \{1,\dots,M \}} |c_{k,v}-d_{x,v,i}|&,\quad
k_{2,i} = \argmin_{ k\in \{1,\dots,M\}\setminus \{k_{1,i}\}} |c_{k,v}-d_{x,v,i}|, \\
w_{1,i} = \abs{\frac{c_{k_{2,i},v} - d_{x,v,i}}{c_{k_{1,i},v}-c_{k_{2,i},v}}}&,\quad
w_{2,i} = \abs{\frac{c_{k_{1,i},v} - d_{x,v,i}}{c_{k_{1,i},v}-c_{k_{2,i},v}}},
\end{split}
\]
if $\abs{c_{k_{1,i},v} - d_{x,v,i}} + \abs{c_{k_{2,i},v} - d_{x,v,i}} < \abs{c_{k_{1,i},v}-c_{k_{2,i},v}}$, and $w_{1,i} = w_{2,i} = 0$ otherwise.

The last condition ensures that the rays, corresponding to the data in interpolation, are on the different sides of the geodesic segment. Vice versa, if the condition does not hold, the geodesic segment is outside the projection width. In other words, this condition is the zero extension of the data near boundaries of the domain. In short,~$\mathcal{A}_{\T^2}$ is the sum of weighted averages of two closest projection values with respect to their distances to the corresponding geodesic segments.


\subsection{Computational inverse model}
\label{ssec:Computationalinversemodel}

In the inverse model, we calculate the Fourier series coefficients of a phantom and reconstruct its Fourier series up to a finite radius $r > 0$. The Fourier coefficients are calculated using the inversion formulas~\eqref{eq:rec1} and~\eqref{eq:rec0} of theorem~\ref{thm:invfor}. Furthermore, we Tikhonov regularize reconstructions using the filter~$P_\alpha^s$ on the Fourier side according to theorems~\ref{thm:regularization} and~\ref{thm:strategy}.

Let us write $B_r = B(0,r) \cap \Z^2$. The inverse model is 
\begin{equation}\begin{split}
f_\mathrm{rec}^{\alpha,s}(x) &= \sum_{k\in B_r} P_\alpha^s \widehat{f}_\mathrm{rec}(k)\exp(2\pi i k \cdot x) \\ 
&= \sum_{k\in B_r} (1+\alpha\vev{k}^{2s})^{-1} \widehat{f}_\mathrm{rec}(k)\exp(2\pi i k \cdot x)
\end{split}
\label{eq:reconstruction}
\end{equation}
where $\hat{f}_\mathrm{rec}(k)$ is calculated from data using the left-point rule and the DFT according to~\eqref{item:dft1} and~\eqref{item:dft2} of corollary~\ref{cor:dft}. {We remark that the inverse model is meshless and
its output is a trigonometric polynomial.}

\section{Numerical experiments}
\label{sec:numerics}


\subsection{Phantoms, convergence rates of Fourier series and discretization}
\label{ssec:numexp.pre}

\subsubsection{Phantoms}
We have used two phantoms in the numerical experiments, the Shepp--Logan phantom based on Matlab's function \texttt{phantom} and the Flag phantom which is a piece-wise constant function representing a Nordic flag.
The Flag phantom $f_F\colon[0,1] \times [0,1] \to \R$ was defined as 
\begin{equation}
f_{F}(x,y)
=
\begin{cases}
g_F(x,y), & x \in (0.14,0.86) \text{ and } y \in (0.28,0.72) \\
0, & \text{otherwise}
\end{cases}
\label{eq:flag1}
\end{equation}
where
\begin{equation}
g_{F}(x,y)
=
\begin{cases}
0.3, & x \in (0.34,0.46) \text{ or } y \in (0.44,0.56) \\
0.9, & \text{otherwise.}
\end{cases}
\label{eq:flag2}
\end{equation}
That is,~$f_F$ describes the outer boundaries of the flag, and~$g_F$ returns the background unless~$x$ or~$y$ is on the horizontal or vertical stripe, respectively.

\subsubsection{Cutoff errors of Fourier series of phantoms}

We analyzed the cutoff errors of Fourier series of the phantoms in order to determine a good, practical value of $r>0$ for the reconstructions. The squared cutoff error of Fourier series can be calculated via the formula
\begin{equation}
\label{eq.effectofn}
\epsilon_r = \norm{f}_{L^2(\T^2)}^2 - \sum_{k \in B_r} \hat{f}(k)^2
\end{equation}
using Parseval's identity.

We computed~$\epsilon_r$ for the Shepp--Logan phantom, the Flag phantom and the Flag phantom with a $45^\circ$ rotation. All the three phantoms were studied without noise and with salt-and-pepper (S\&P) type noise applied to the phantoms using Matlab's \texttt{imnoise} function with $0.02$ noise density. The phantoms were discretized into $4000\times 4000$ pixel grid and the Fourier coefficients~$\hat{f}(k)$ were computed using Matlab's \texttt{fft2} and \texttt{fftshift} functions.

The squared cutoff errors~$\epsilon_r$ are shown in Figure~\ref{fig.effectofn}. The squared cutoff errors saturate at around $r=50$, though some improvement might be gained up to $r=200$. In our forward and inverse simulations, we have mainly used $r=50$ as it practically seems to be a sufficiently good choice.

\begin{figure}[ht]
\includegraphics[width=0.7\textwidth]{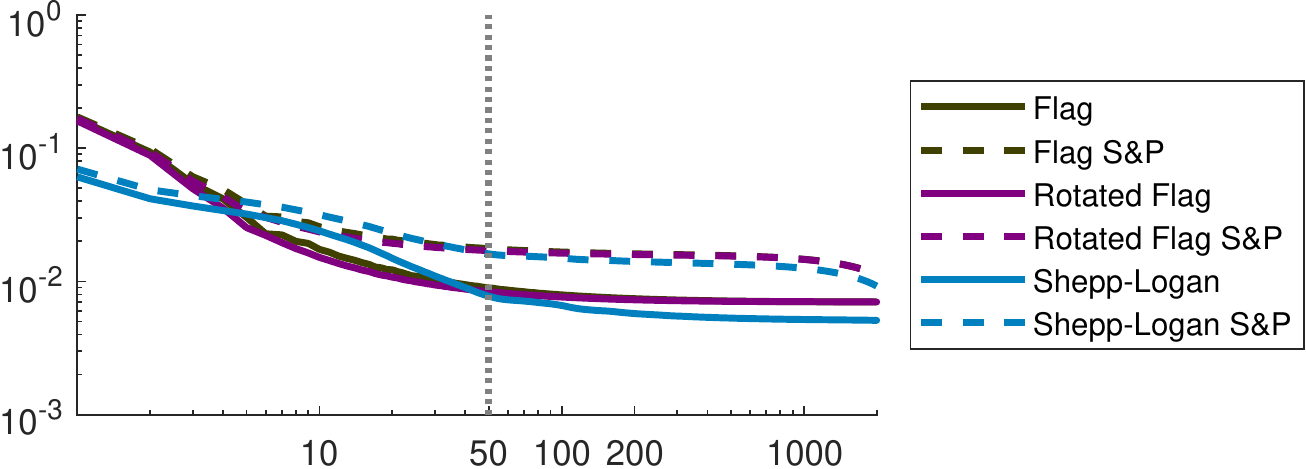}
\caption{The graphs $(r,\epsilon_r)$ in a logarithmic scale for the different phantoms. Vertical, dashed line marks $r=50$ where~$\epsilon_r$ saturates.}
\label{fig.effectofn}
\end{figure}

\subsubsection{Discretizations of phantoms and geodesics}
\label{ssec:discretization}

The starting points of the used geodesics were chosen to be the equispaced points 
$\{ (0,0), (1/n_d,0), (2/n_d,0), \dots, (1-1/n_d,0) \}$ on the $x$-axis, 
except for geodesics in direction $v=(1,0)$ where the sampling was 
$\{ (0,0), (0,1/n_d), (0,2/n_d), \dots, (0,1-1/n_d) \}$ on the $y$-axis.
In our experiments, we set $n_d=128$ when the cutoff radius of the Fourier series was $r\in\{50,100\}$, and $n_d=r$ when $r\in\{150,200\}$.

The phantoms were discretized with $512\times 512$ pixel grid when used for forward simulations with the forward models~$\mathcal{A}_1$ and~$\mathcal{A}_{\T^2}$. When we used the forward model~$\mathcal{A}_2$, the Flag phantom was not discretized. The values of reconstructions were evaluated at equispaced points in $256\times 256$ pixel grid, and when compared to the ground truth, the Shepp--Logan and the Flag phantoms were discretized for the same grid.

\subsection{Numerical analysis of forward models~$\mathcal{A}_1$, $\mathcal{A}_2$ and~$\mathcal{A}_{\T^2}$}
\label{ssect:numexp.forward}

\subsubsection{Forward models~$\mathcal{A}_1$ and~$\mathcal{A}_2$ on the torus}

We tested Torus CT using the Shepp--Logan phantom with simulated data 
\begin{equation}
\label{eq:noisemodel}
    y = \mathcal{A}_1f + \mathcal{E}, \quad \mathcal{E} \sim \mathcal{N}(0,\sigma^2), \quad \sigma = \frac{2}{100}. 
\end{equation}
We made reconstructions with cutoff radii $r\in\{50,100,150,200\}$ of the Fourier series.

In the case of $r=50$, we experimented with Tikhonov regularization. The reconstruction errors with different regularization parameters are shown in Figure~\ref{fig.shepploganerrsurf}. We have calculated the (relative) reconstruction errors using the formula
\begin{equation}
\epsilon_p^{\alpha,s} = \frac{ \norm{f-f_\mathrm{rec}^{\alpha,s}}_{L^p(\T^2)}}{\norm{f}_{L^p(\T^2)}}.
\label{eq.errorequation}
\end{equation}
The optimal regularization parameter values yielding the smallest error are given in Table~\ref{table.errsurftable}. The plotted errors Figure~\ref{fig.shepploganerrsurf} share some similarities in shape and the resulting regularization parameter values are close to each other.

\begin{table}[hb]
\caption{The regularization parameters $(\alpha,s)$ that give the best reconstructions with respect to the $L^p$-norms with $p = 1,2, \infty$; respective error~$\epsilon_p^{\alpha,s}$; and error~$\epsilon_p^{0,0}$ of non-regularized reconstruction.}
\label{table.errsurftable}
\begin{tabular}{l | 
S[table-format=1.3] 
S[table-format=1.2] 
S[table-format=2.0,table-space-text-post = \%]
S[table-format=3.0,table-space-text-post = \%] | 
S[table-format=1.3] 
S[table-format=1.2] 
S[table-format=2.0,table-space-text-post = \%]
S[table-format=3.0,table-space-text-post = \%] }
norm & \multicolumn{4}{c|}{Shepp--Logan} & \multicolumn{4}{c}{Flag} \\
$p$ & \multicolumn{1}{c}{$\alpha$} & \multicolumn{1}{c}{$s$} & \multicolumn{1}{c}{$\epsilon_p^{\alpha,s}$} & \multicolumn{1}{c|}{$\epsilon_p^{0,0}$} & \multicolumn{1}{c}{$\alpha$} & \multicolumn{1}{c}{$s$} & \multicolumn{1}{c}{$\epsilon_p^{\alpha,s}$} & \multicolumn{1}{c}{$\epsilon_p^{0,0}$} \\
1 & 0.050 & 0.69 & 62\% & 112\% & 0.025 & 0.71 & 41\% & 69\% \\
2 & 0.025 & 0.61 & 48\% & 70\% & 0.025 & 0.68 & 29\% & 45\% \\
$\infty$ & 0.025 & 0.56 & 75\% & 112\% & 0.025 & 0.78 & 73\% & 106\% \\
\end{tabular}
\end{table}

The Shepp--Logan phantom is shown in Figure~\ref{fig.shepploganphnatom} and its non-regularized solution in Figure~\ref{fig.shepplogannoreg}. The regularized solutions with $p=2$ and $p=\infty$ based regularization parameter values (Figures~\ref{fig.shepplogan2} and~\ref{fig.shepploganinf}) are similar, and $p=1$ based values yield slightly smoother reconstruction (Figure~\ref{fig.shepplogan1}).

We tested the effect of increasing the Fourier coefficient by computing the forward data required for reconstruction of the Fourier coefficients up to radii $r=100$, $r=150$ and $r=200$, and reconstructions are shown in Figures~\ref{fig.shepploganN100}, \ref{fig.shepploganN150}, and~\ref{fig.shepploganN200} respectively. The constant regions in the phantom become a bit more smoother, but overall dynamical range is increased and the impact of noise in reconstructions remains relatively high.

Similar analysis was also performed with the Flag phantom. We simulated noisy data using the model $y = \mathcal{A}_2f + \mathcal{E}$ with the noise model of~\eqref{eq:noisemodel}.
The case $r=50$ was used to test regularization. The reconstruction errors~$\epsilon_p^{\alpha,s}$ are shown in Figure~\ref{fig.flagerrsurf} and the regularization values yielding the minimum error are given in Table~\ref{table.errsurftable}. The regions close to the minimum of~$\epsilon_p^{\alpha,s}$ are more distinct than in the case of the Shepp--Logan phantom, but similar shape is seen.

The Flag phantom is shown in Figure~\ref{fig.flag} and the non-regularized reconstruction in Figure~\ref{fig.flagnonreg}. The regularized reconstructions with the optimal regularization parameters yielding the minimum errors with $p=1$, $p=2$ and $p=\infty$ are shown in Figures~\ref{fig.flag1}, \ref{fig.flag2} and~\ref{fig.flaginf}, respectively. The regularization parameter values yielding the minimum were close to each other, and with the Flag phantom, no significant difference is seen in the regularized reconstructions. 

Increasing the radius of the Fourier coefficients again increases the dynamical range, plotted in Figures~\ref{fig.flagN100}, \ref{fig.flagN150} and~\ref{fig.flagN200} for $r\in\{100,150,200\}$, respectively. However, unlike with the Shepp--Logan phantom, the details become more distinct, especially the details of the corners in the Flag phantom.

\begin{figure}
\begin{center}
\begin{subfigure}{0.6\textwidth}
\includegraphics[width=\textwidth]{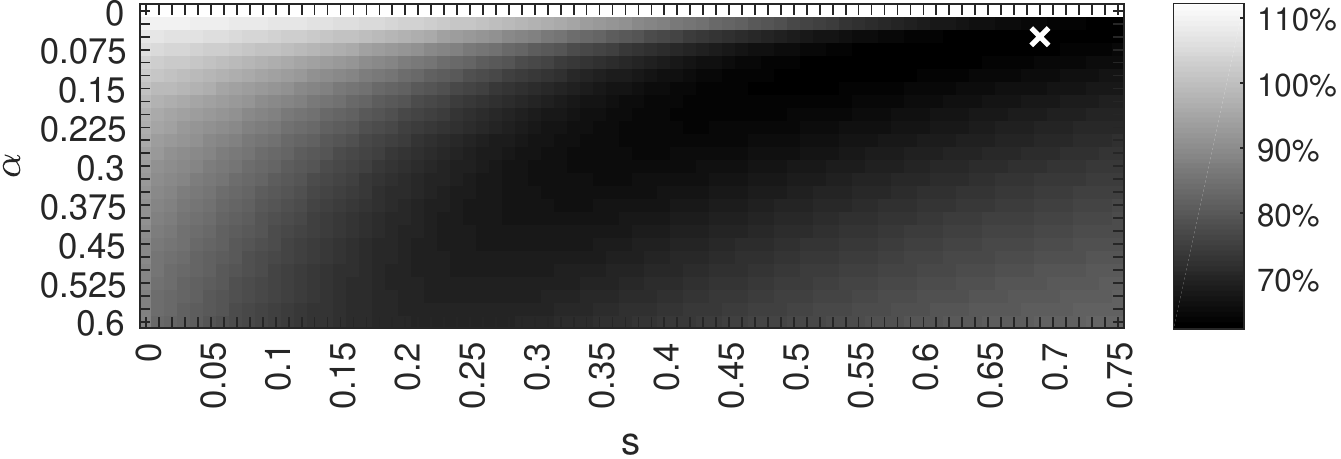}
\caption{} 
\end{subfigure}

\begin{subfigure}{0.6\textwidth}
\includegraphics[width=\textwidth]{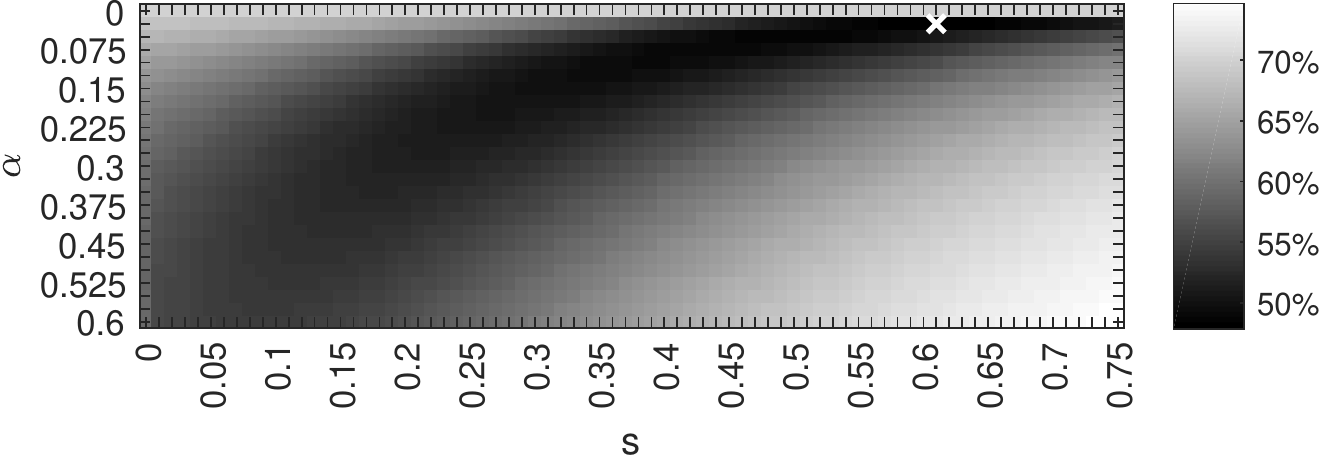}
\caption{}
\end{subfigure}

\begin{subfigure}{0.6\textwidth}
\includegraphics[width=\textwidth]{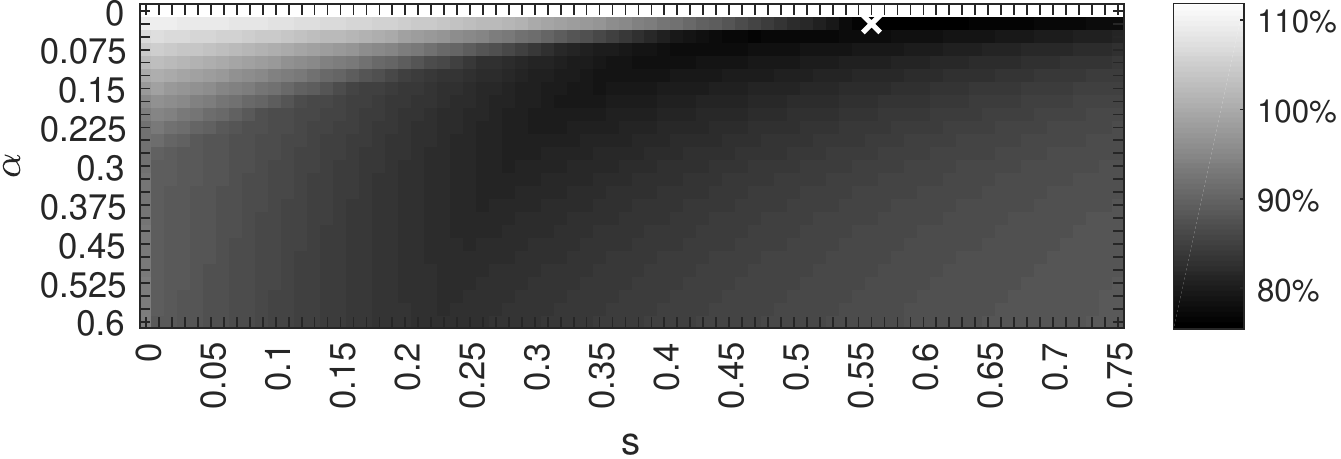}
\caption{}
\end{subfigure}
\end{center}
\caption{Error surfaces from Shepp--Logan phantom data using (A) $L^1$-norm, (B) $L^2$-norm and (C) $L^\infty$-norm. The values of the regularization parameters are $\alpha \in \{0, 0.025, 0.050, \dots, 0.600\}$ and $s \in \{0,0.01,0.02,\dots,0.75\}$.}
\label{fig.shepploganerrsurf}
\end{figure}

\begin{figure}
\begin{center}
\begin{subfigure}{0.25\textwidth}
\includegraphics[width=\textwidth]{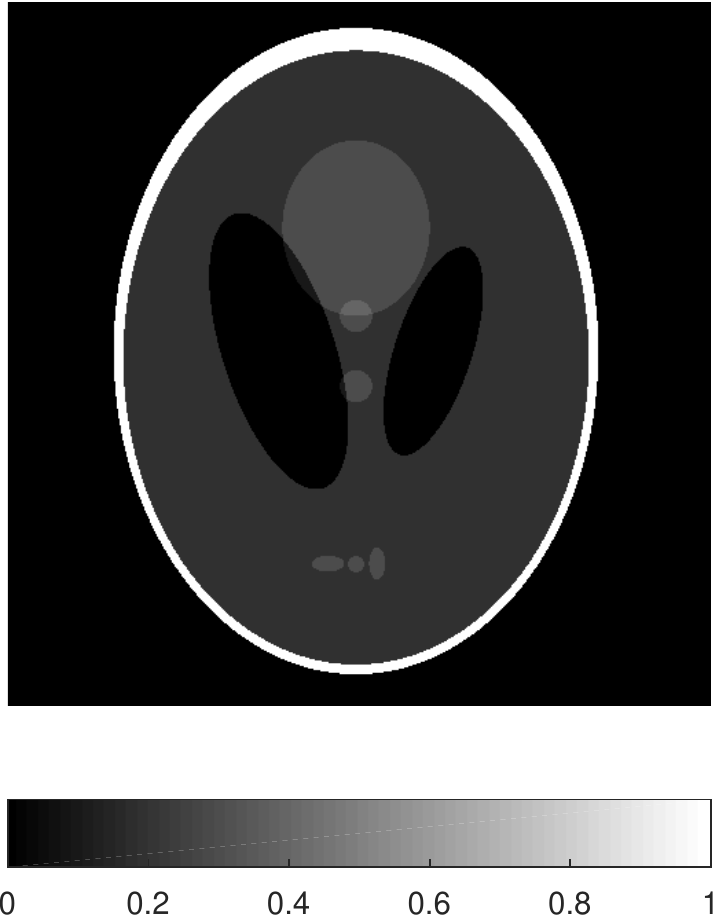}
\caption{} \label{fig.shepploganphnatom}
\end{subfigure}
~
\begin{subfigure}{0.25\textwidth}
\includegraphics[width=\textwidth]{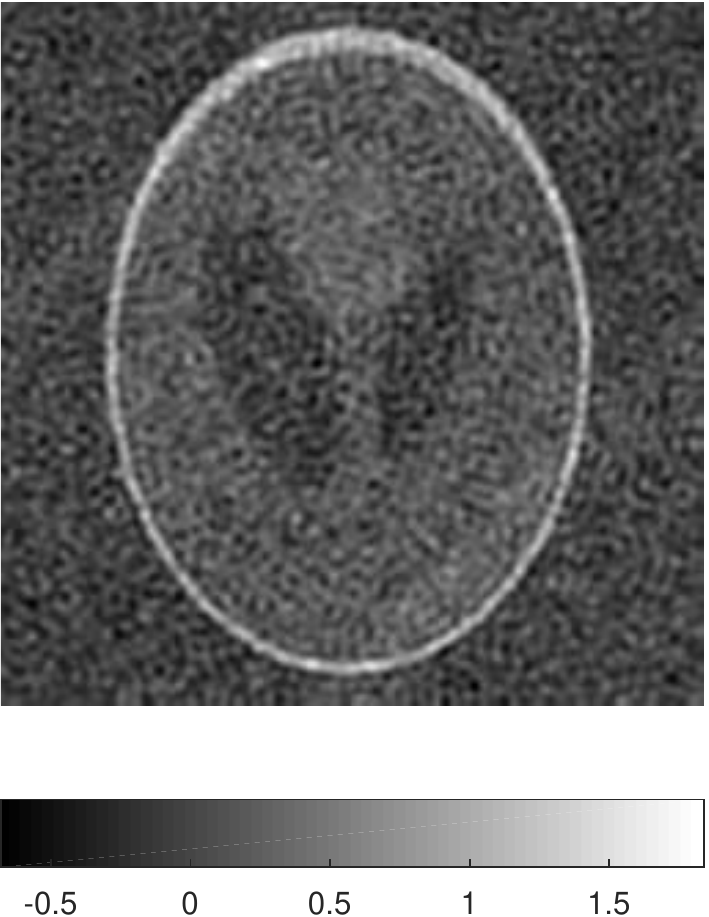}
\caption{} \label{fig.shepplogannoreg}
\end{subfigure}

\begin{subfigure}{0.25\textwidth}
\includegraphics[width=\textwidth]{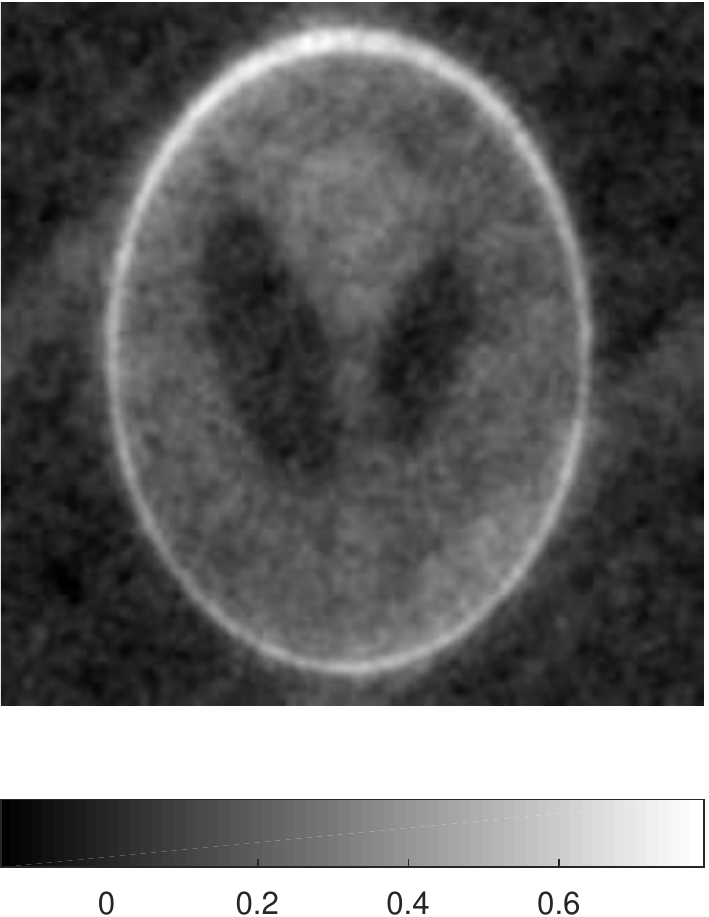}
\caption{} \label{fig.shepplogan1}
\end{subfigure}
\begin{subfigure}{0.25\textwidth}
\includegraphics[width=\textwidth]{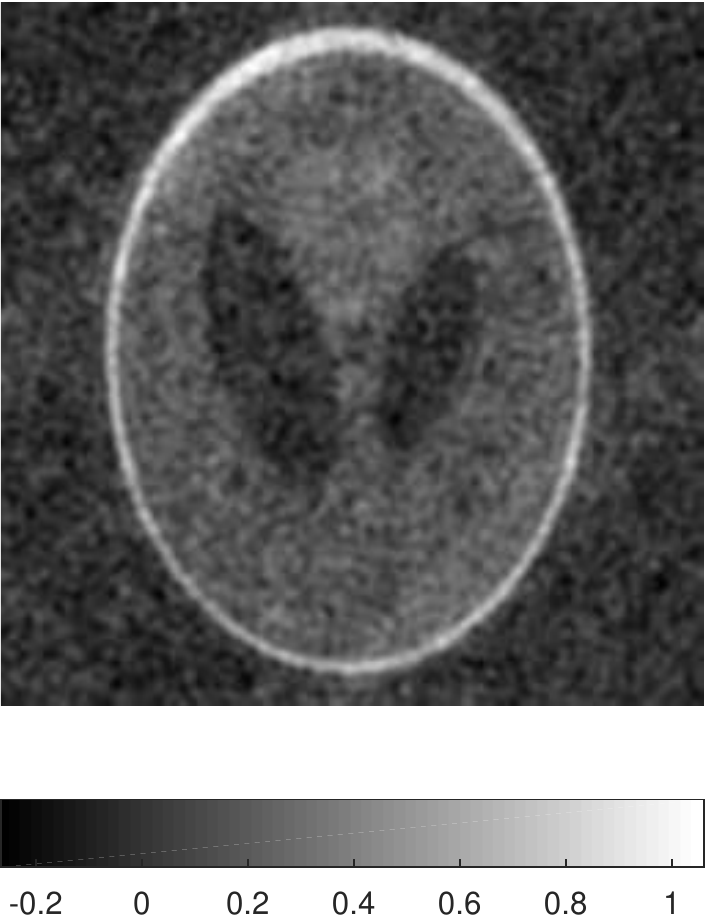}
\caption{} \label{fig.shepplogan2}
\end{subfigure}
\begin{subfigure}{0.25\textwidth}
\includegraphics[width=\textwidth]{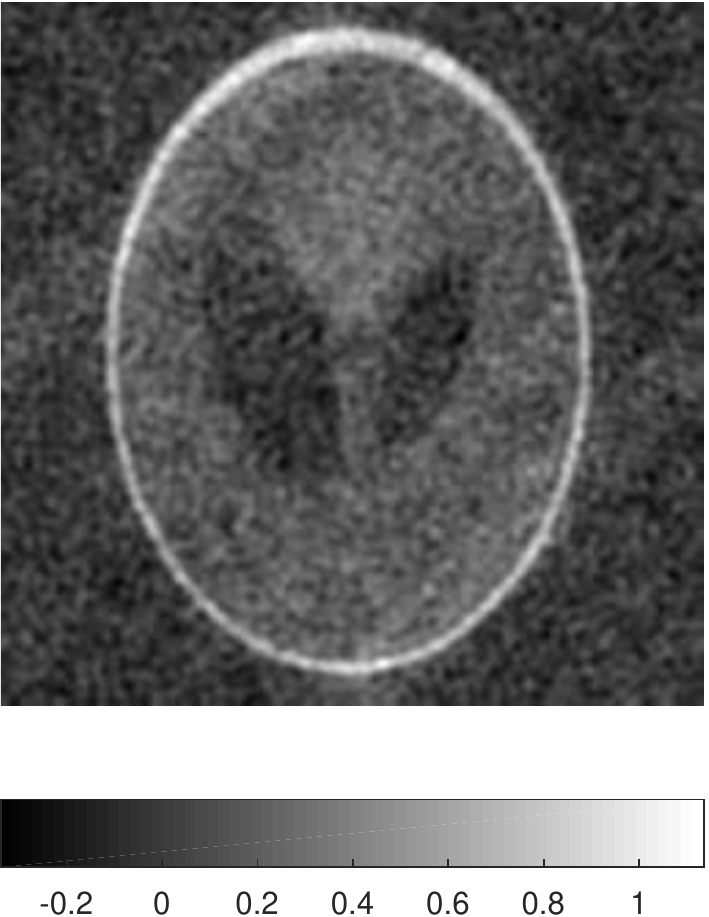}
\caption{} \label{fig.shepploganinf}
\end{subfigure}

\begin{subfigure}{0.25\textwidth}
\includegraphics[width=\textwidth]{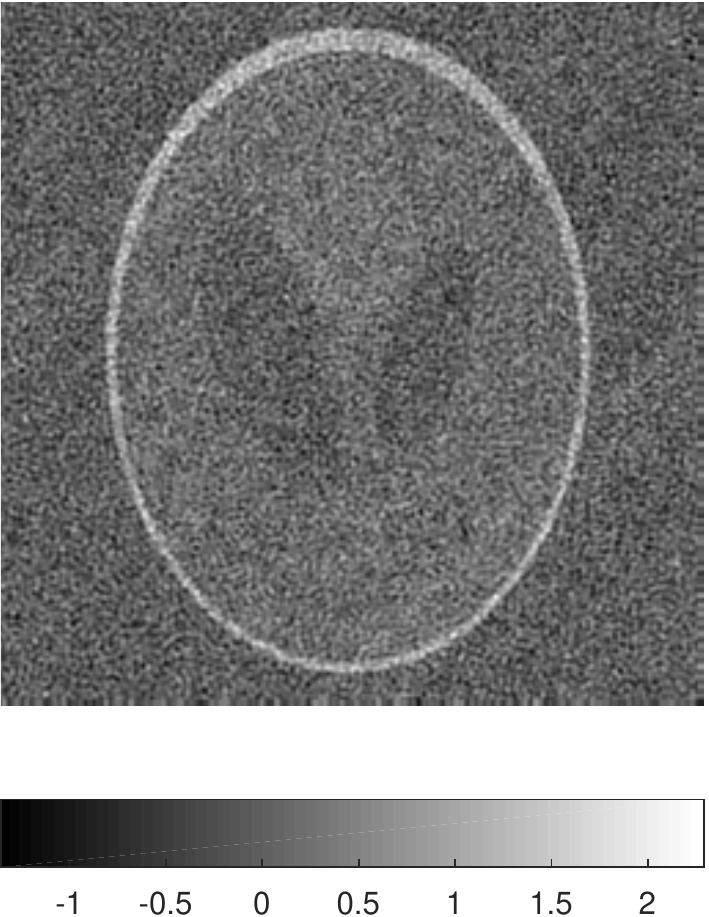}
\caption{} \label{fig.shepploganN100}
\end{subfigure}
\begin{subfigure}{0.25\textwidth}
\includegraphics[width=\textwidth]{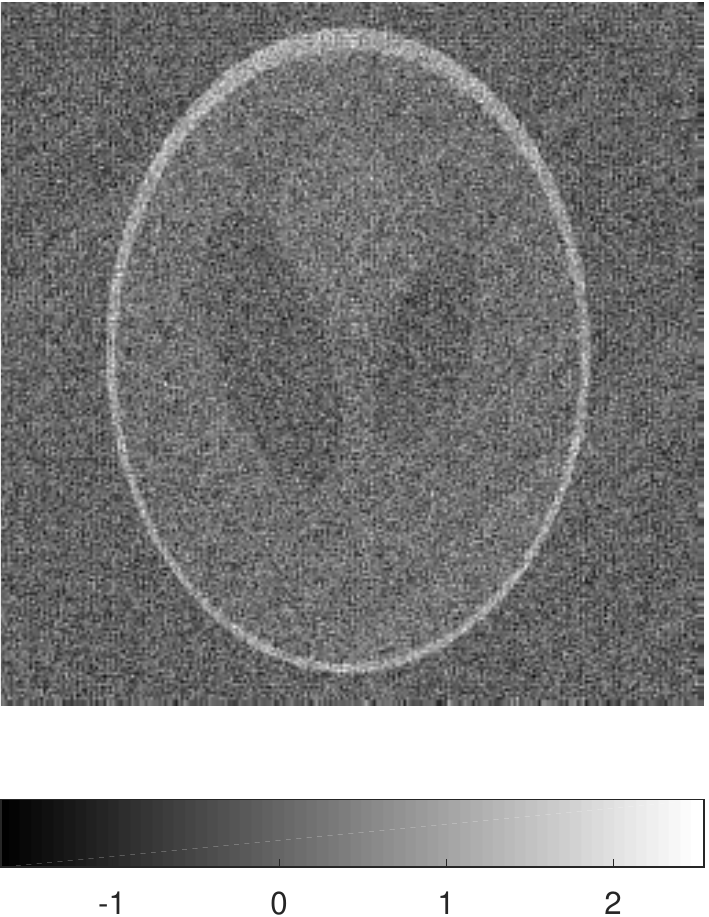}
\caption{} \label{fig.shepploganN150}
\end{subfigure}
\begin{subfigure}{0.25\textwidth}
\includegraphics[width=\textwidth]{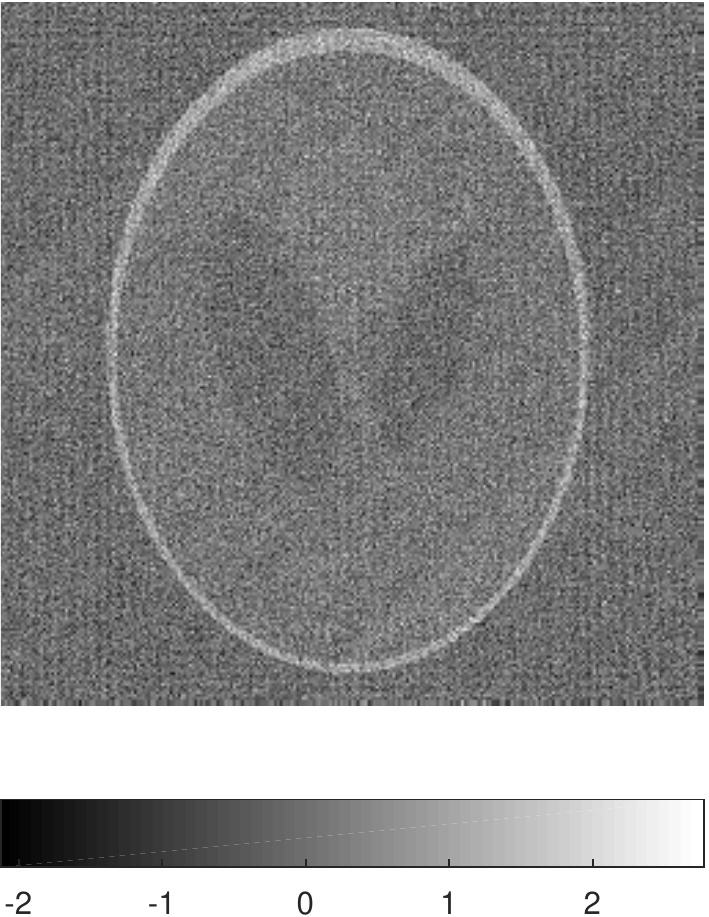}
\caption{} \label{fig.shepploganN200}
\end{subfigure}

\end{center}
\caption{(A) Shepp--Logan phantom, (B) non-regularized reconstruction and (C-E) regularized reconstructions respectively with $L^1$-, $L^2$- and $L^\infty$-norm based choice of reconstruction values. (F-H) Non-regularized reconstruction with increased cutoff radii of the Fourier series, $r=100, 150, 200$, respectively.}
\label{fig.shepploganreconstructions}
\end{figure}

\begin{figure}
\begin{center}
\begin{subfigure}{0.6\textwidth}
\includegraphics[width=\textwidth]{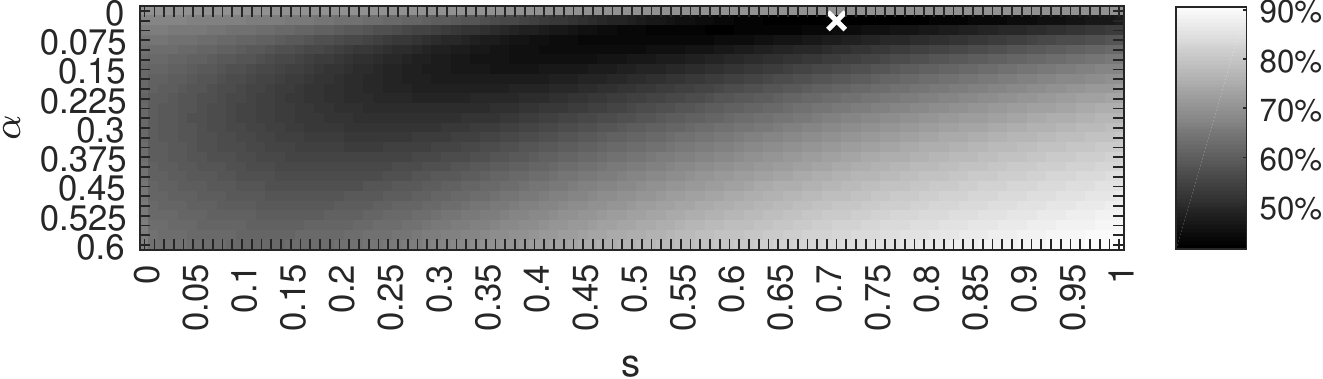}
\caption{}
\end{subfigure}

\begin{subfigure}{0.6\textwidth}
\includegraphics[width=\textwidth]{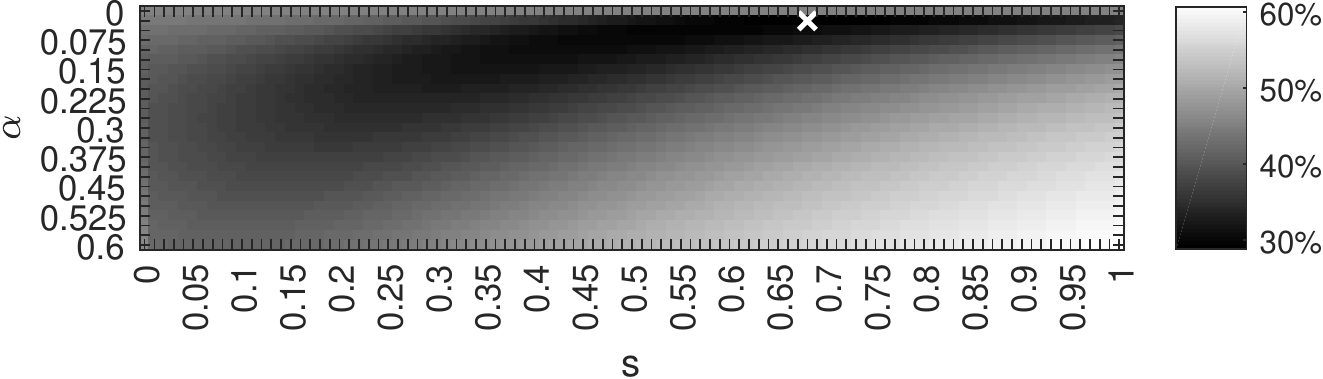}
\caption{}
\end{subfigure}

\begin{subfigure}{0.6\textwidth}
\includegraphics[width=\textwidth]{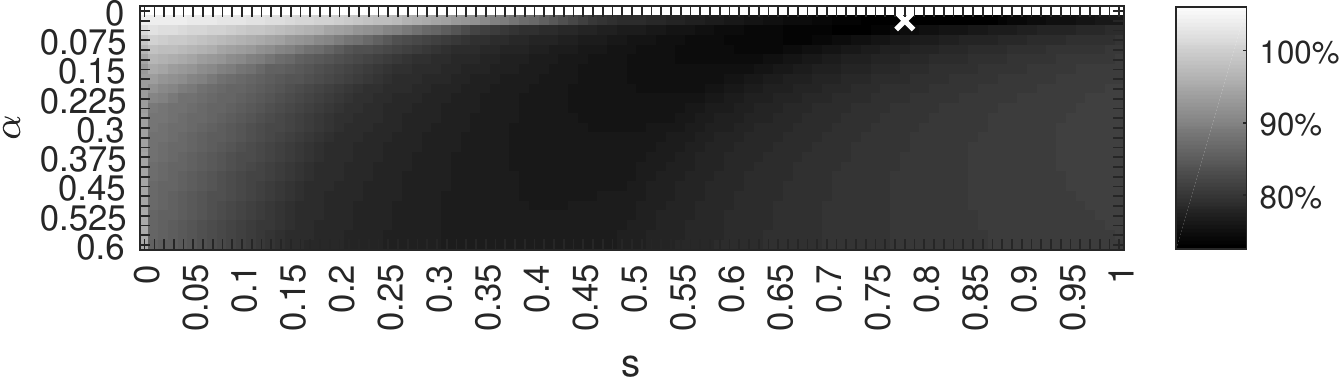}
\caption{}
\end{subfigure}
\end{center}
\caption{Error surfaces from Flag phantom data using (A) $L^1$-norm, (B) $L^2$-norm and (C) $L^\infty$-norm. Regularization parameters values are $\alpha \in \{0, 0.025, 0.050, \dots, 0.600\}$ and $s \in \{0,0.01,0.02,\dots,1.0\}$. }
\label{fig.flagerrsurf}
\end{figure}

\begin{figure}
\begin{center}
\begin{subfigure}{0.25\textwidth}
\includegraphics[width=\textwidth]{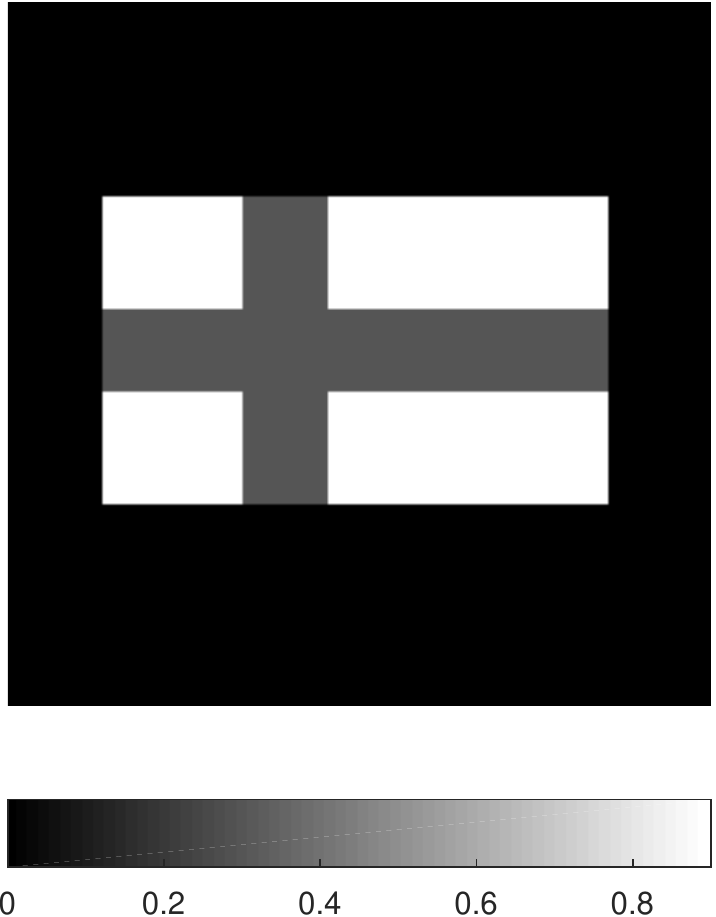}
\caption{} \label{fig.flag}
\end{subfigure} 
~
\begin{subfigure}{0.25\textwidth}
\includegraphics[width=\textwidth]{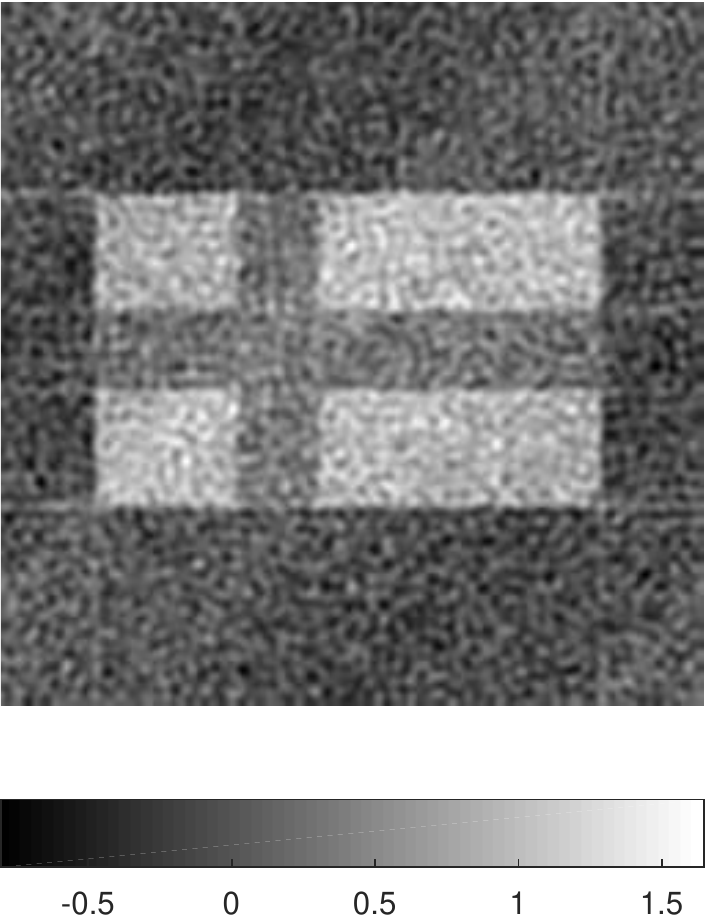}
\caption{} \label{fig.flagnonreg}
\end{subfigure}

\begin{subfigure}{0.25\textwidth}
\includegraphics[width=\textwidth]{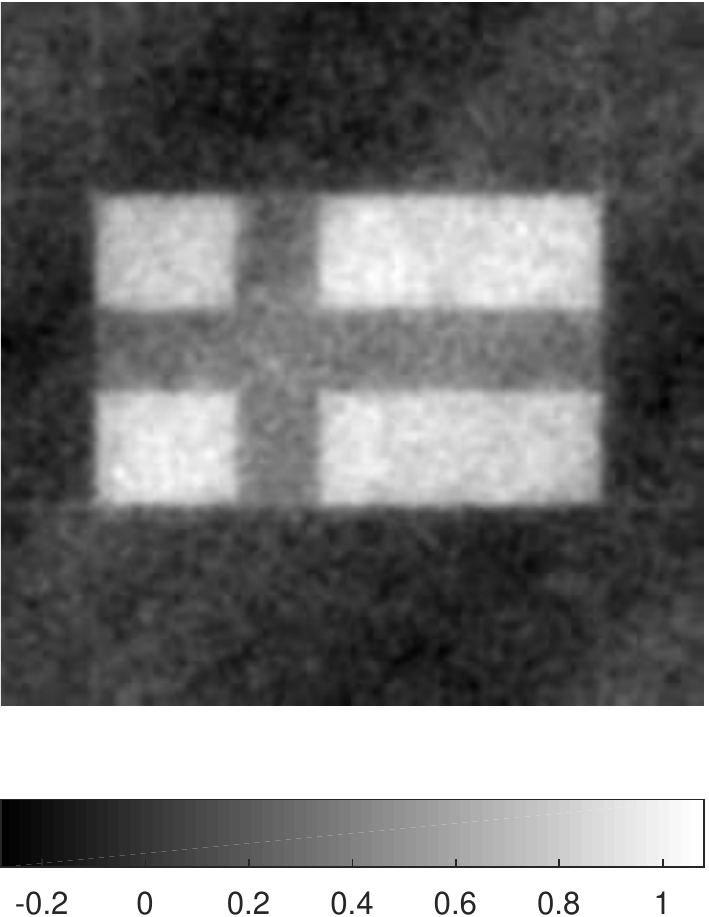}
\caption{} \label{fig.flag1}
\end{subfigure}
\begin{subfigure}{0.25\textwidth}
\includegraphics[width=\textwidth]{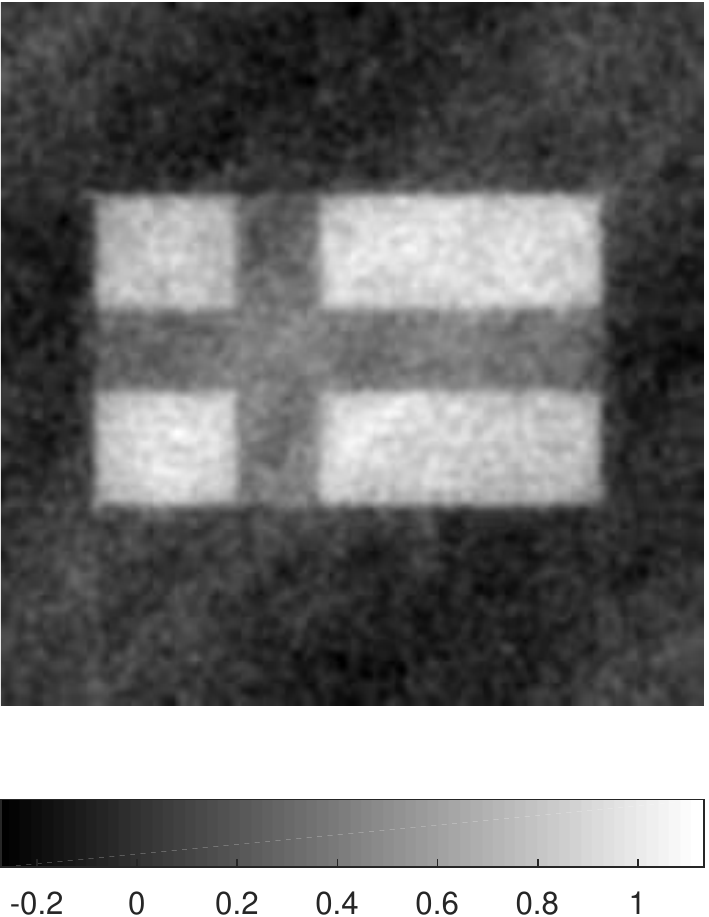}
\caption{} \label{fig.flag2}
\end{subfigure}
\begin{subfigure}{0.25\textwidth}
\includegraphics[width=\textwidth]{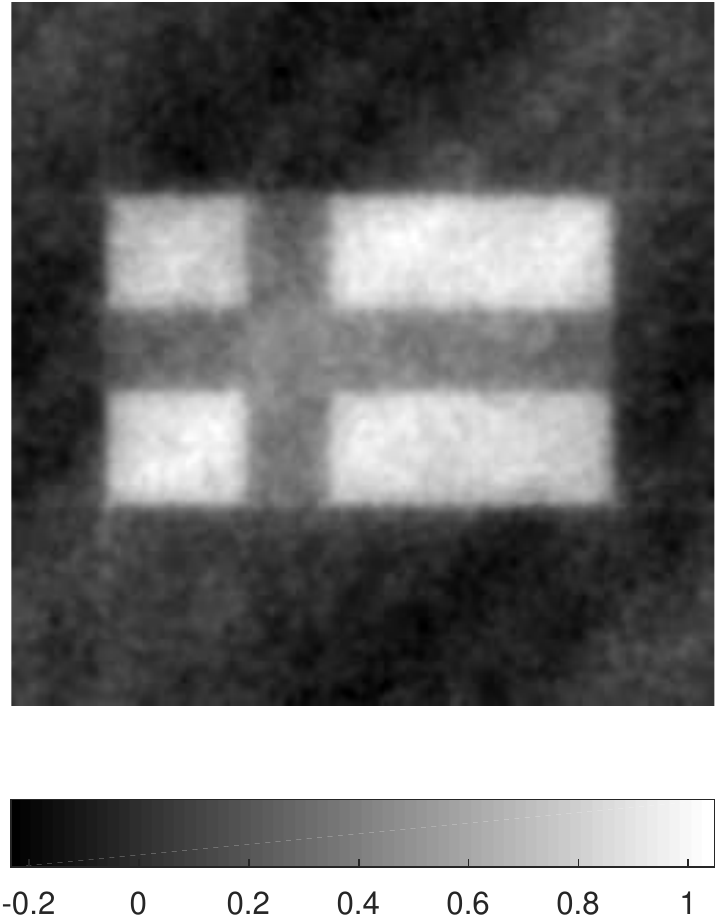}
\caption{} \label{fig.flaginf}
\end{subfigure}

\begin{subfigure}{0.25\textwidth}
\includegraphics[width=\textwidth]{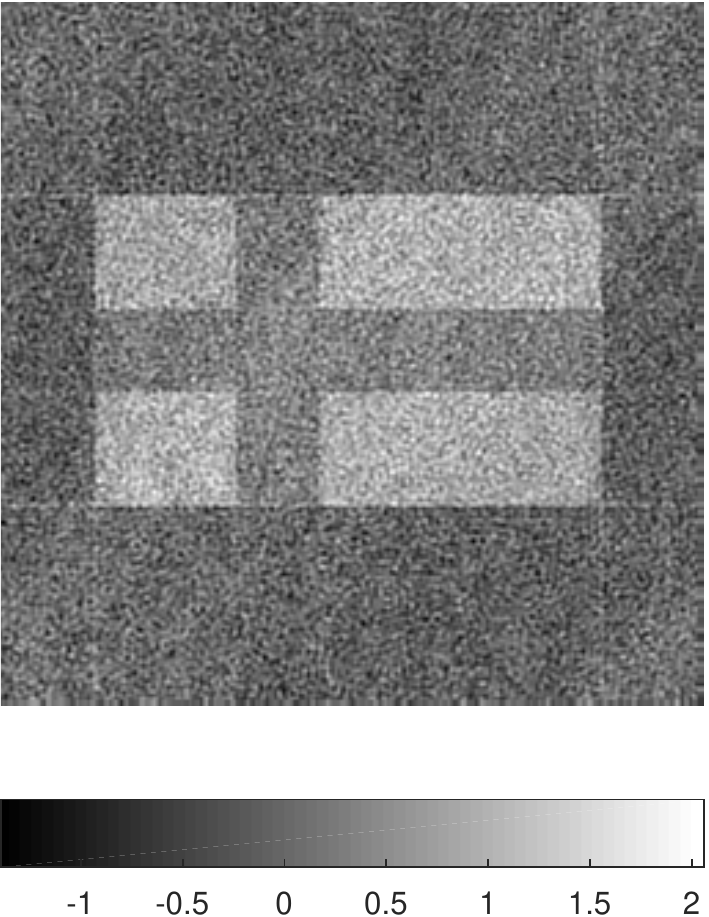}
\caption{} \label{fig.flagN100}
\end{subfigure}
\begin{subfigure}{0.25\textwidth}
\includegraphics[width=\textwidth]{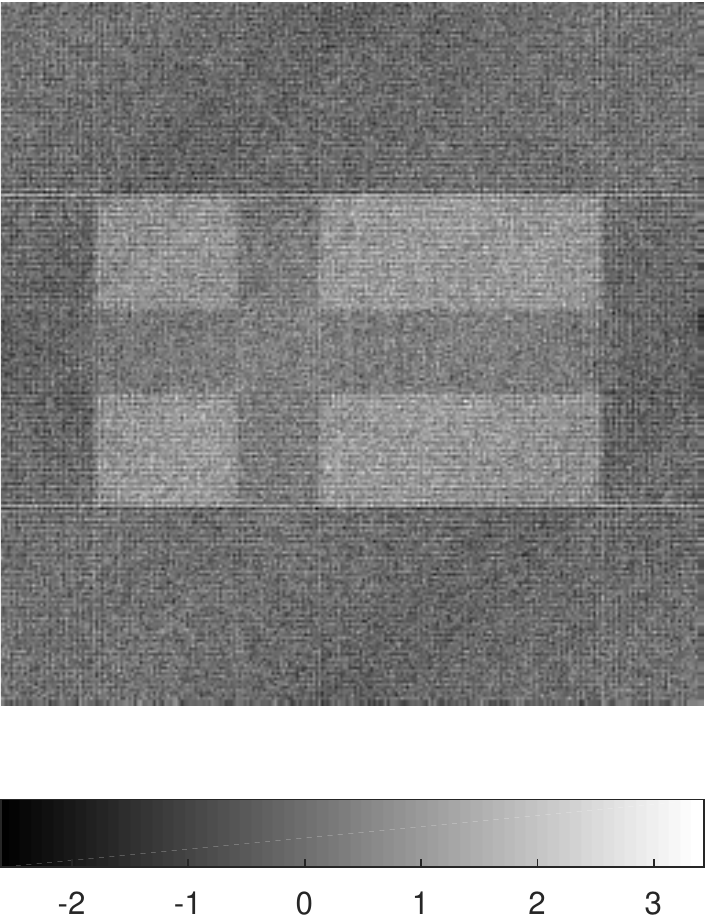}
\caption{} \label{fig.flagN150}
\end{subfigure}
\begin{subfigure}{0.25\textwidth}
\includegraphics[width=\textwidth]{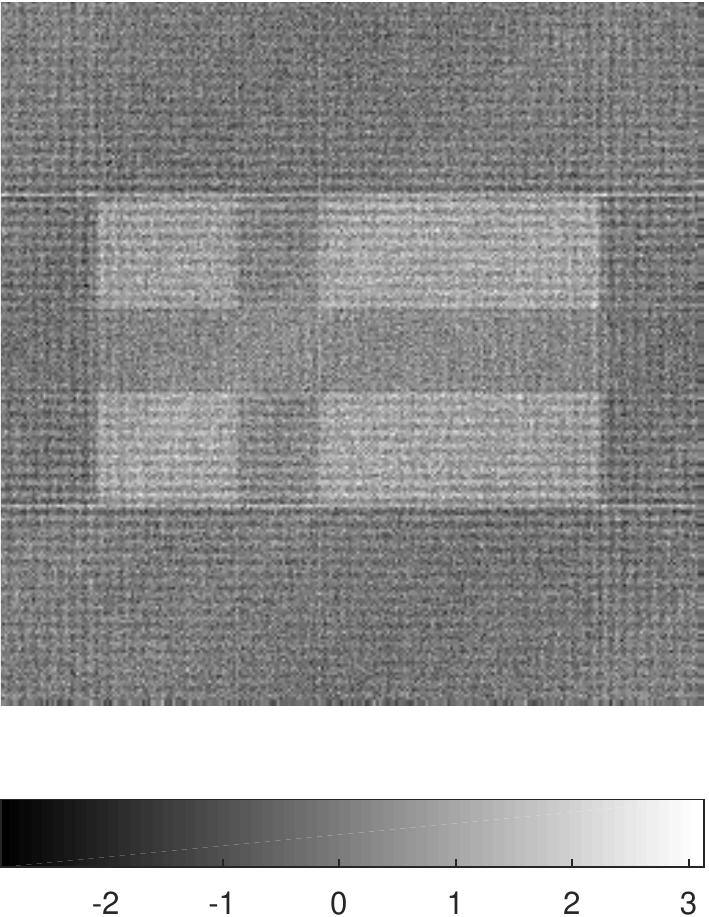}
\caption{} \label{fig.flagN200}
\end{subfigure}

\end{center}
\caption{(A) Flag phantom, (B) non-regularized reconstruction and (C-E) regularized reconstructions respectively with $L^1$-, $L^2$- and $L^\infty$-norm based choice of reconstruction values. (F-H) Non-regularized reconstruction with increased cutoff radii of the Fourier series, $r=100, 150, 200$, respectively. }
\label{fig.flagreconstructions}
\end{figure}

\subsubsection{Forward model~$\mathcal{A}_{\T^2}$ using the torus-projection and Radon data}

\begin{figure}
\begin{subfigure}{0.25\textwidth}
\includegraphics[width=\textwidth]{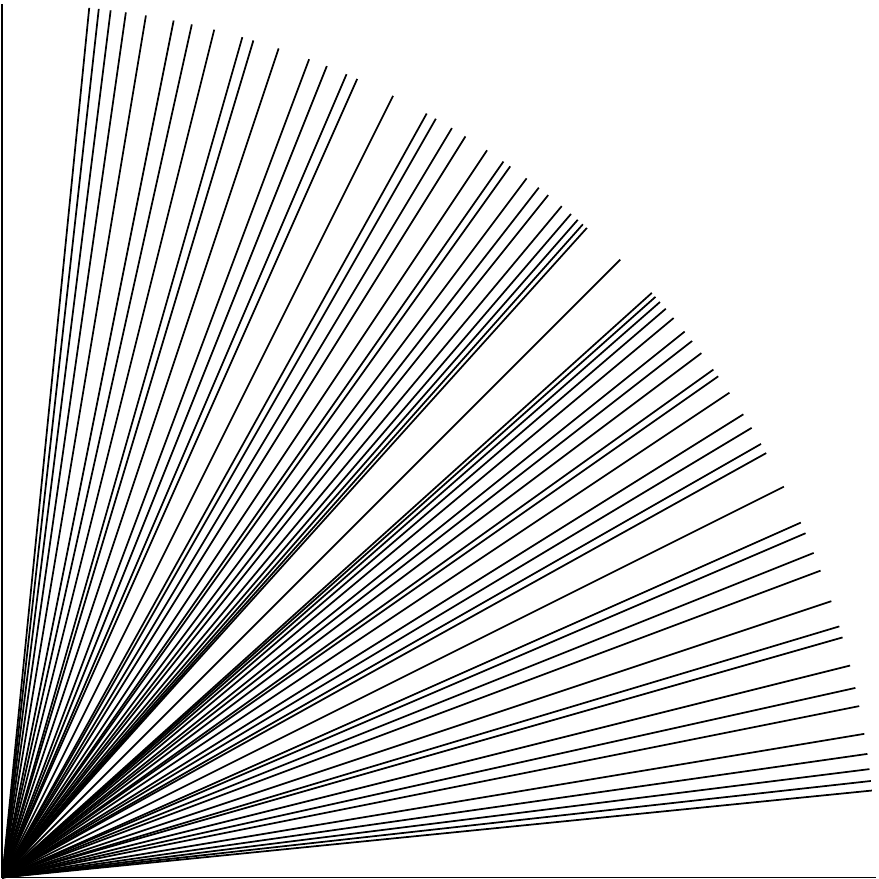}
\caption{}
\end{subfigure}
\begin{subfigure}{0.25\textwidth}
\includegraphics[width=\textwidth]{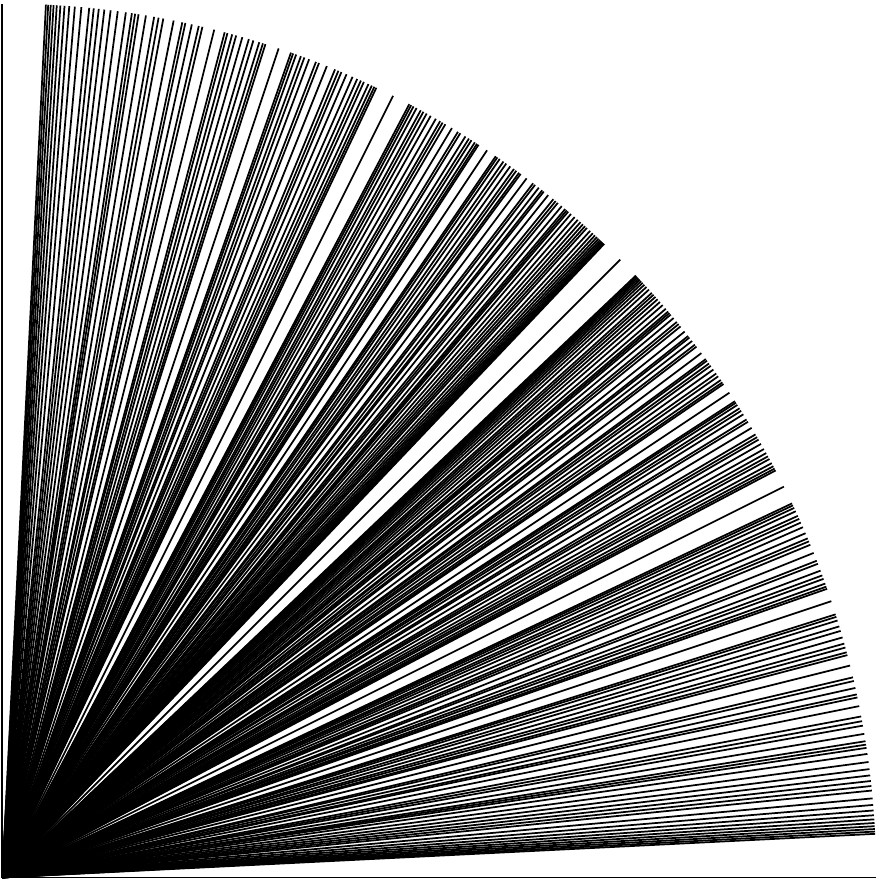}
\caption{}
\end{subfigure}
\begin{subfigure}{0.25\textwidth}
\includegraphics[width=\textwidth]{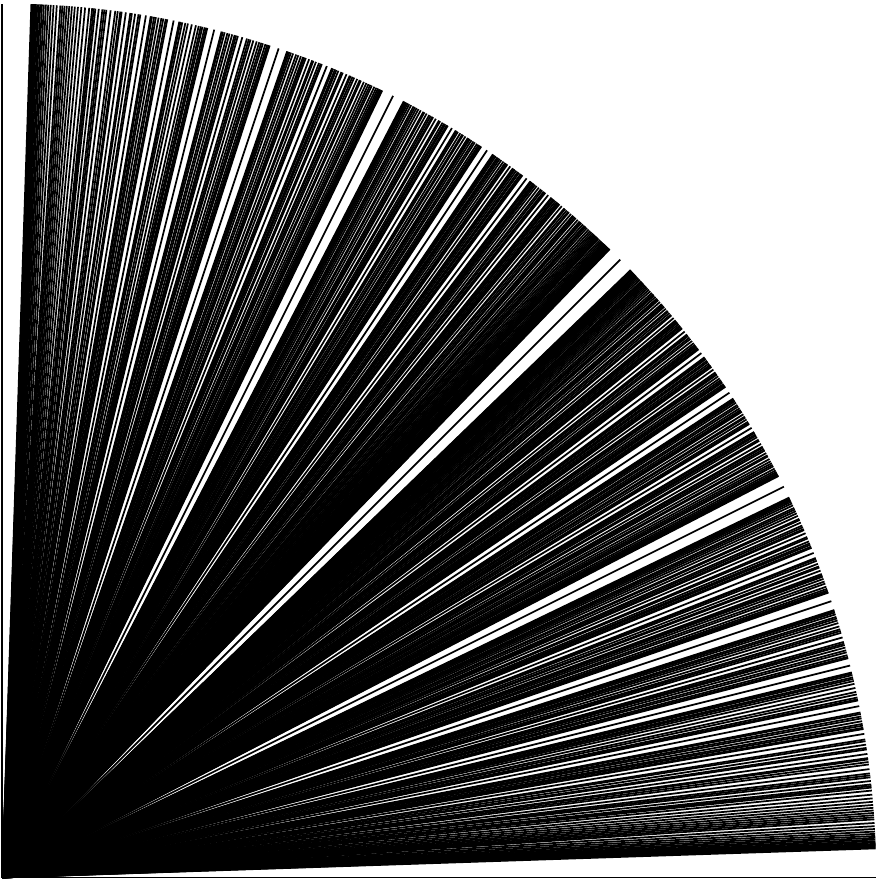}
\caption{}
\end{subfigure}
\caption{Visualization of the Radon projection angles from~$0^\circ$ to~$90^\circ$ that are required in the reconstruction of the Fourier series coefficients up to radii $r=10, 20, 30$, respectively. Each line represents a projection direction.}
\label{fig.torusradonangles}
\end{figure}

To test how a Torus CT algorithm would work with experimental data acquisition, we computed Radon transform of the phantoms and projected it to~$\T^2$ using the model~$\mathcal{A}_{\T^2}$ with noise on each data point on~$Rf(v)_k$. More precisely, we simulated data according to formula~\eqref{eq:T2forward} where each~$Rf(v)_k$ was replaced by noisy data $Rf(v)_k+\mathcal{E}$ where $\mathcal{E} \sim \mathcal{N}(0,\sigma^2)$ with $\sigma = 0.02$. 

The projection directions for Radon transform were computed such that they determined the Fourier coefficients up to radius $r=50$. An illustration of how the projection directions in $(0,90)^\circ$ are distributed is shown in Figure~\ref{fig.torusradonangles}, and the remaining projection directions are reflections of the projection directions in $(0,90)^\circ$ about the $y$-axis. In total, with $r=50$, there are 3097 unique projection directions. Two major concentrations of the directions are close to $45^\circ$, both above and below, but also smaller concentrations are found elsewhere, e.g., close to~$22.5^\circ$.

The reconstructions from data computed with~$\mathcal{A}_{\T^2}$ are shown in Figure~\ref{fig.radonrecs}. Shepp--Logan (Figures~\ref{fig:radonA} and~\ref{fig:radonD}) and $30^\circ$ rotated Flag (Figures~\ref{fig:radonC} and~\ref{fig:radonF}) are reconstructed well even with the noisy data. However, the non-rotated Flag phantom (Figures~\ref{fig:radonB} and~\ref{fig:radonE}) is rather poor.
{
As Figure~\ref{fig.torusradonangles} already showed, the torus-optimized data does not contain information uniformly of all directions.
Instability and thus blurriness is expected where the available directions shown in Figure~\ref{fig.torusradonangles} are sparse.
The most prominent example is the coordinate axis.
Hence, it was expected that recovery of singularities near coordinate axis directions is practically unstable \cite{Q93,Q17}. This unfavorable feature has motivated us to study a modification of the inversion method which is expected to succeed better in practice. This modified approach and the related simulations are presented in section \ref{ssec:numexp.additional}.
} With the Shepp--Logan phantom, the features are clearly detected, especially in the noise-free case (Figure~\ref{fig:radonA}) indicating potential in the technique. Regularized solutions are shown in Figures~\ref{fig:radonG}), \ref{fig:radonH} and~\ref{fig:radonI} from the Shepp--Logan, the non-rotated and the rotated Flag phantoms, repectively. The regularization smoothed the reconstructions, decreased their dynamic range and no additional features were revealed from the noise. The regularization parameter values were $\alpha=0.75$ and $s=0.5$, chosen with manual experimentation.

\begin{table}[hb]
\caption{Errors in reconstructions computed with $\mathcal{A}_{\T^2}$ and the FBP.}
\label{table.fbperrtable}
\begin{tabular}{l | 
S[table-format=3.0,table-space-text-post = \%]
S[table-format=3.0,table-space-text-post = \%]
S[table-format=3.0,table-space-text-post = \%] |
S[table-format=3.0,table-space-text-post = \%]
S[table-format=3.0,table-space-text-post = \%]
S[table-format=3.0,table-space-text-post = \%] }
& \multicolumn{1}{c}{Shepp--Logan} & \multicolumn{1}{c}{Flag} & \multicolumn{1}{c|}{Rotated Flag} & \multicolumn{1}{c}{Shepp--Logan} & \multicolumn{1}{c}{Flag} & \multicolumn{1}{c}{Rotated Flag} \\
\hline
& \multicolumn{3}{c|}{$\mathcal{A}_ {\T^2}$ with noiseless data ($\mathcal{E}=0$)} & \multicolumn{3}{c}{$\mathcal{A}_ {\T^2}$ with noisy data ($\mathcal{E} \sim \mathcal{N}(0,\sigma^2)$)} \\
$\epsilon_1^{0,0}$      & 313\% & 298\% & 302\% & 310\% & 300\% & 301\% \\
$\epsilon_2^{0,0}$      & 161\% & 171\% & 168\% & 162\% & 172\% & 167\% \\
$\epsilon_\infty^{0,0}$ & 75\% & 108\% & 121\% & 79\% & 125\% & 125\% \\
& \multicolumn{3}{c|}{} & \multicolumn{3}{c}{Regularized reconstruction from noisy data} \\
$\epsilon_1^{\alpha,s}$ & & & & 331\% & 305\% & 303\% \\
$\epsilon_2^{\alpha,s}$ & & & & 170\% & 174\% & 173\% \\
$\epsilon_\infty^{\alpha,s}$ & & & & 53\% & 99\% & 100\% \\
& \multicolumn{3}{c|}{FBP with torus optimized angles} & \multicolumn{3}{c}{FBP with evenly distributed angles} \\
$\epsilon_1$ & 73\% & 59\% & 67\% & 64\% & 55\% & 56\% \\
$\epsilon_2$ & 59\% & 41\% & 51\% & 54\% & 45\% & 45\% \\
$\epsilon_\infty$ & 155\% & 87\% & 127\% & 129\% & 93\% & 120\% \\
\end{tabular}
\end{table}

For comparison, we computed the respective FBP reconstructions (shown in Figure~\ref{fig.fbprecs}) with Matlab's \texttt{iradon} function using default settings. The projection data $R_\texttt{radon} f(v)_k + \mathcal{E}$ was down sampled by factor of $2$ with \texttt{imresize} to match reconstruction resolution $256\times 256$. It seems, that the uneven distribution of projection angles
creates errors in reconstruction, since similar artefacts in horizontal, vertical and diagonal directions are seen also in the FBP reconstruction as in the ones computed with the Torus CT method in Figure~\ref{fig.radonrecs}. From the FBP this was expected as it is prone to streaking. In general, the FBP reconstructions are of good quality, since there is a lot of data available. With the same number of projections, 3097, but evenly distributed as they normally are, the FBP reconstruction are better quality than any other presented in this paper. 

The error 
$\epsilon_p = \norm{f-f_\mathrm{rec}^\mathrm{FBP}}_{L^p(\R^2)} / \norm{f}_{L^p(\R^2)}$
between the FBP reconstruction~$f_\mathrm{rec}^\mathrm{FBP}$ and the phantom~$f$ is tabulated in Table~\ref{table.fbperrtable}. When compared with~$\mathcal{A}_1$ and~$\mathcal{A}_2$ and Shepp--Logan and Flag phantoms, with all values of~$p$, the errors~$\epsilon_p$ are higher than errors of regularized Torus CT reconstructions~$\epsilon_p^{\alpha,s}$ shown in Table~\ref{table.errsurftable}. When compared with the errors of non-regularized reconstructions~$\epsilon_1^{0,0}$ and~$\epsilon_2^{0,0}$, the FBP and Torus CT are relatively close, but the error~$\epsilon_\infty^{0,0}$ of Torus CT is lower with the Shepp--Logan phantom and higher with the Flag phantom. 

For use with practical data acquisition {with commonly used equispaced projection angles}, Torus CT requires more work to handle the increased additive noise in~$\mathcal{A}_{\T^2}$, since the reconstructions have more noise outside of the support of the phantom than in the FBP reconstructions as seen in Figure~\ref{fig.radonrecs}. The respective errors~$\epsilon_1^{\alpha,s}$ and~$\epsilon_2^{\alpha,s}$ are higher than those of the FBP, presented in Table~\ref{table.fbperrtable}. Nevertheless, in terms of reconstructing the correct dynamical range of the objects, measured with~$\epsilon_\infty^{\alpha,s}$ and~$\epsilon_\infty$, the Torus CT method is equivalent or better than with the FBP. 

\begin{figure}
\begin{subfigure}{0.25\textwidth}
\includegraphics[width=\textwidth]{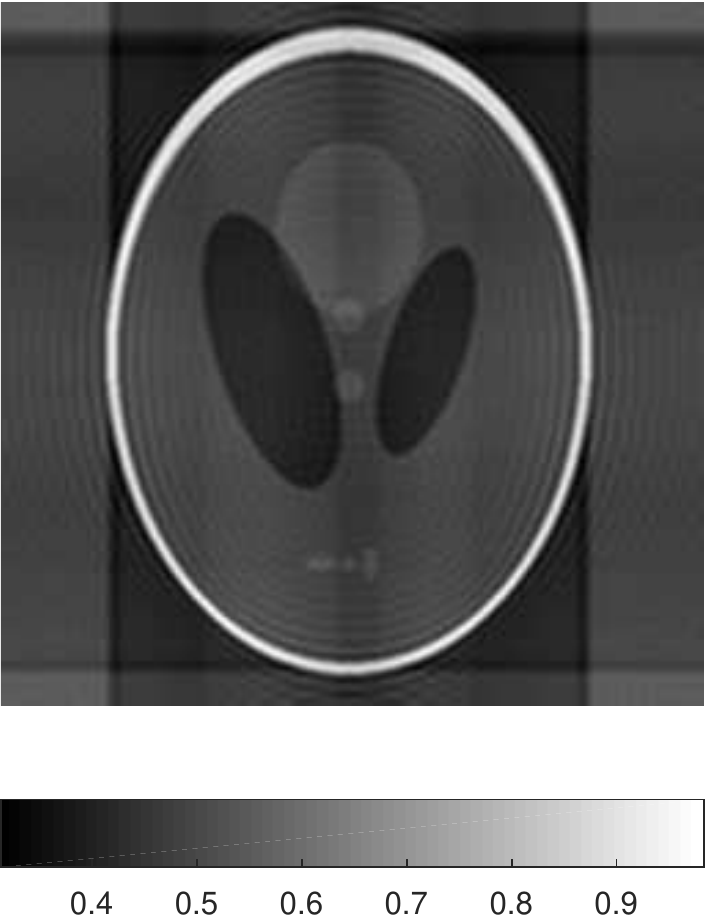}
\caption{} \label{fig:radonA}
\end{subfigure}
\begin{subfigure}{0.25\textwidth}
\includegraphics[width=\textwidth]{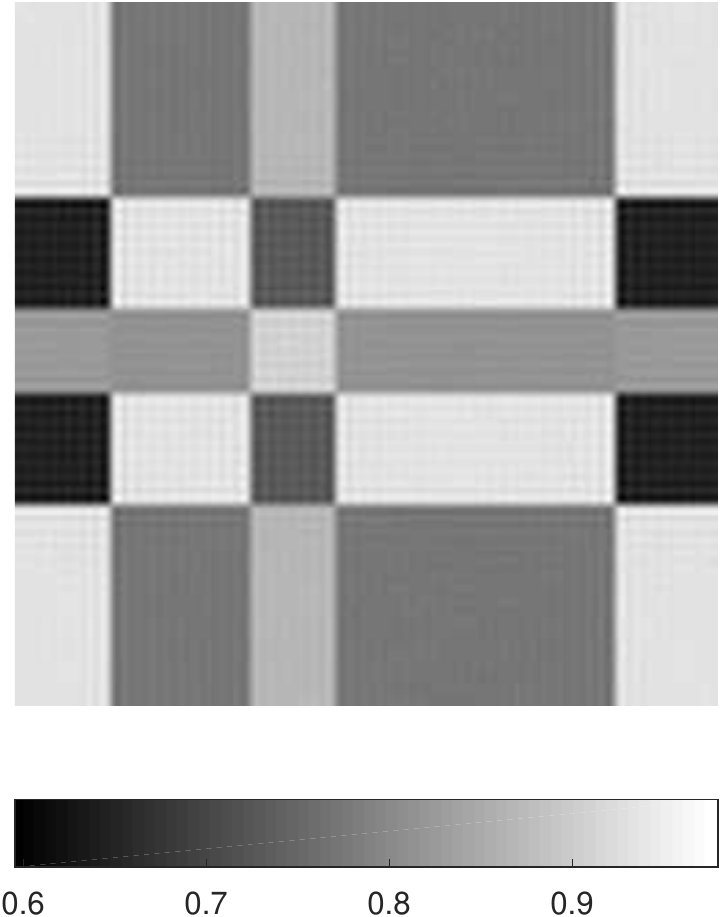}
\caption{} \label{fig:radonB}
\end{subfigure}
\begin{subfigure}{0.25\textwidth}
\includegraphics[width=\textwidth]{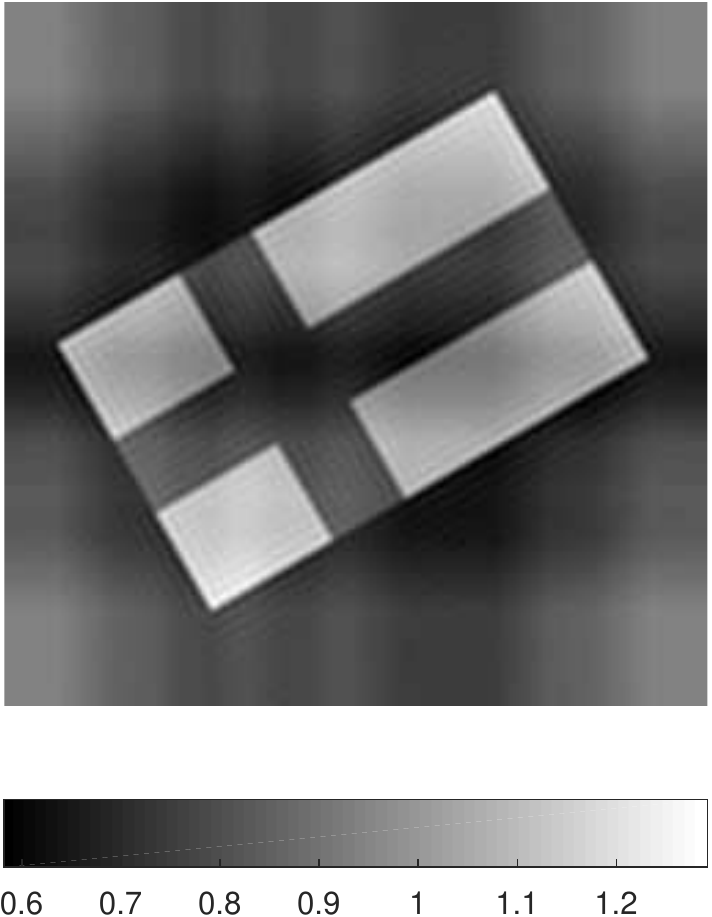}
\caption{} \label{fig:radonC}
\end{subfigure}

\begin{subfigure}{0.25\textwidth}
\includegraphics[width=\textwidth]{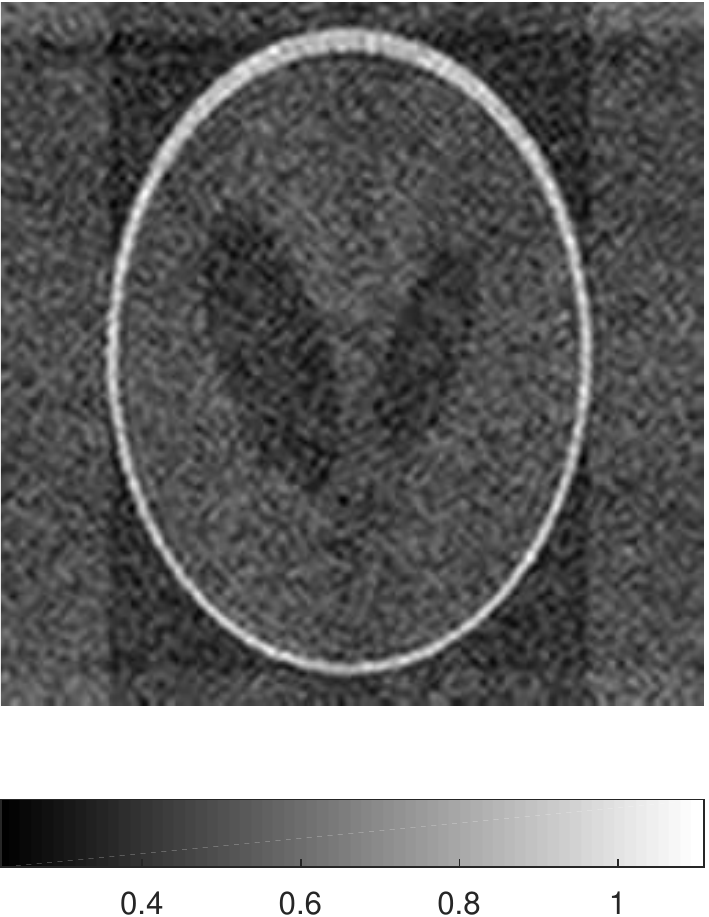}
\caption{} \label{fig:radonD}
\end{subfigure}
\begin{subfigure}{0.25\textwidth}
\includegraphics[width=\textwidth]{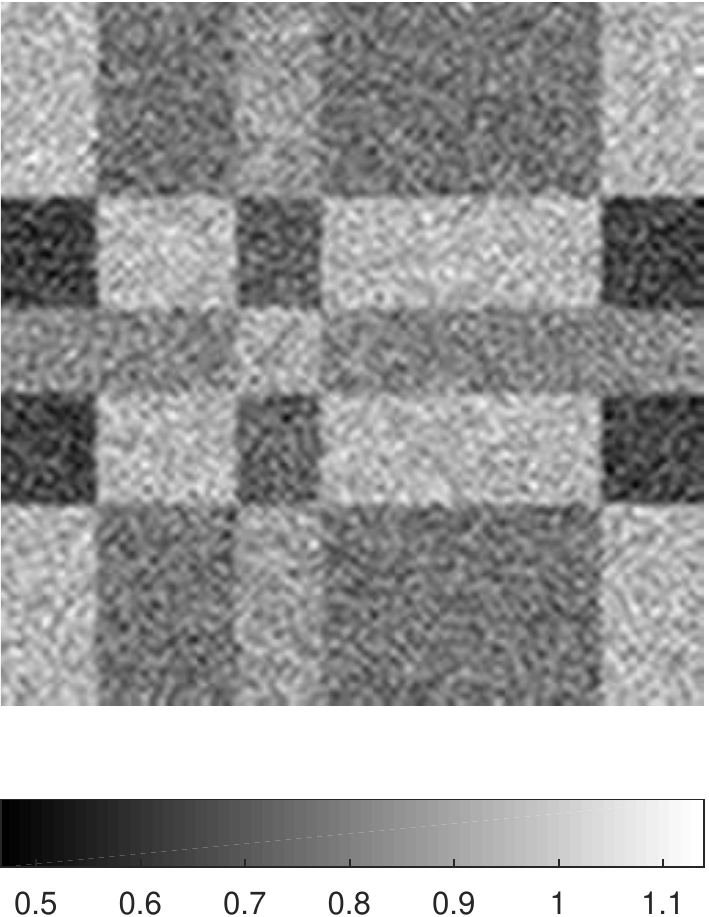}
\caption{} \label{fig:radonE}
\end{subfigure}
\begin{subfigure}{0.25\textwidth}
\includegraphics[width=\textwidth]{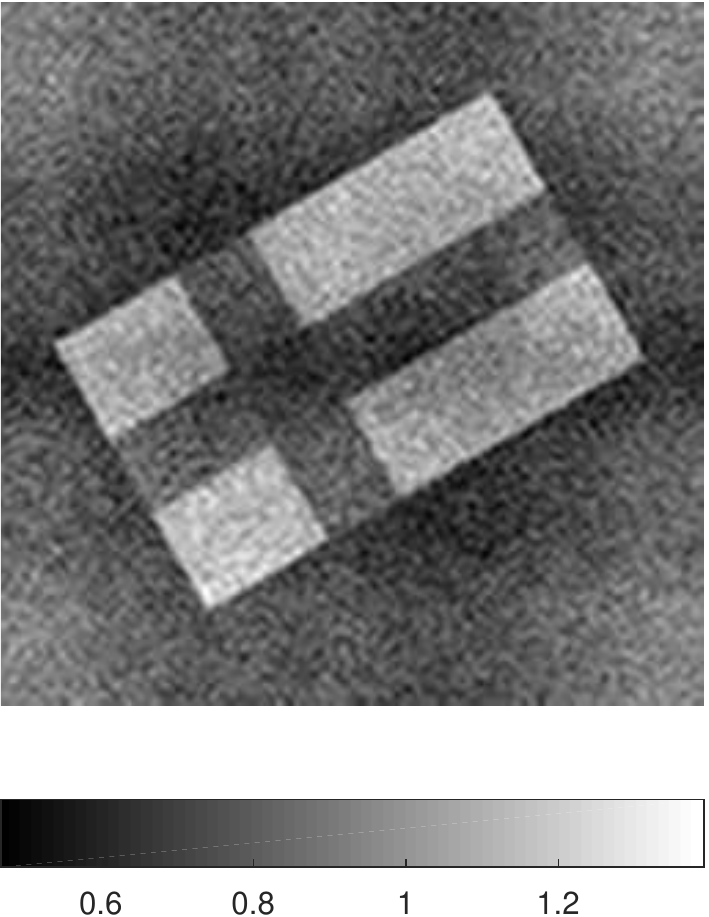}
\caption{} \label{fig:radonF}
\end{subfigure}

\begin{subfigure}{0.25\textwidth}
\includegraphics[width=\textwidth]{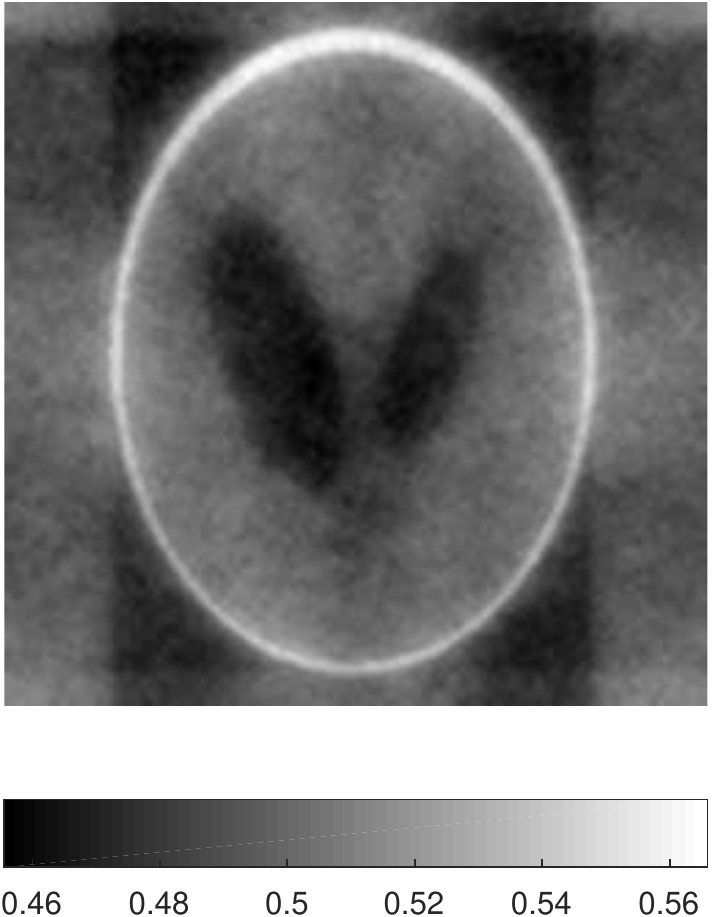}
\caption{} \label{fig:radonG}
\end{subfigure}
\begin{subfigure}{0.25\textwidth}
\includegraphics[width=\textwidth]{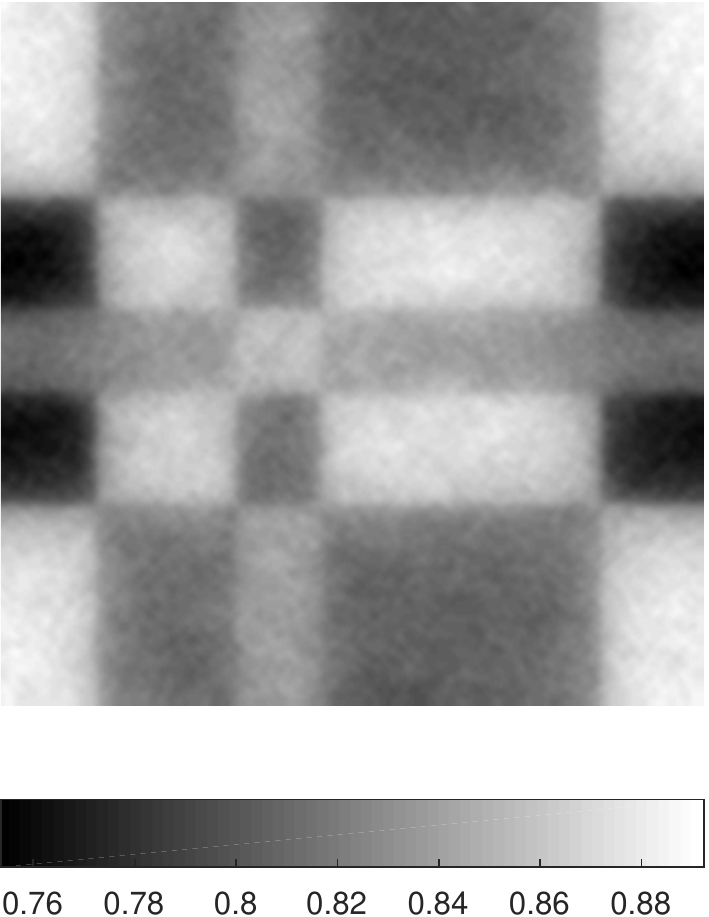}
\caption{} \label{fig:radonH}
\end{subfigure}
\begin{subfigure}{0.25\textwidth}
\includegraphics[width=\textwidth]{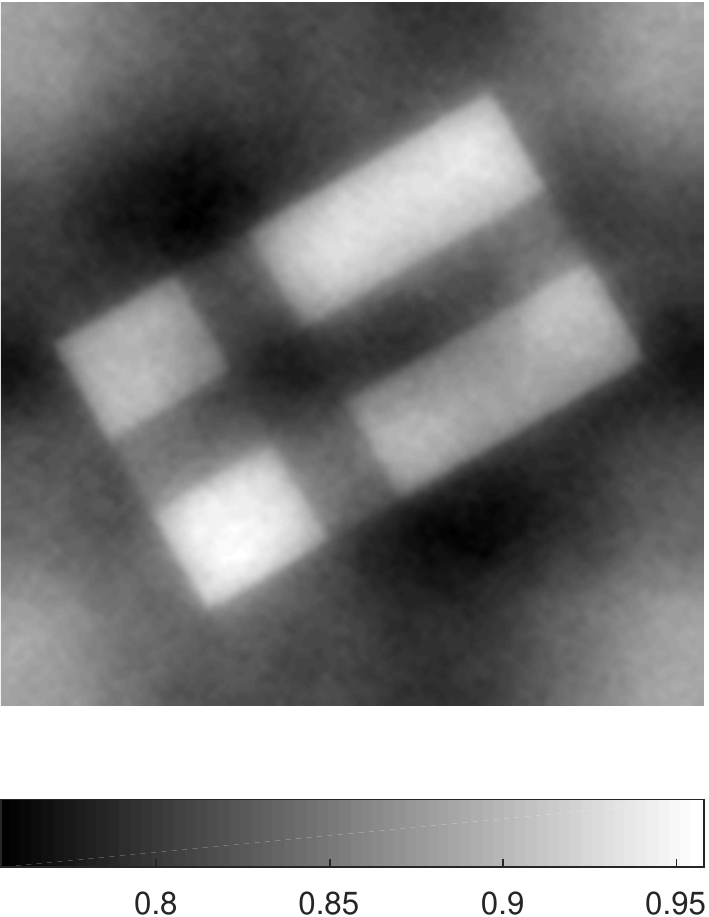}
\caption{} \label{fig:radonI}
\end{subfigure}
\caption{Reconstruction from (A)-(C) noise-less and (D)-(F) noisy Radon data mapped to torus of Shepp--Logan, Flag and Flag rotated $30^\circ$. (G)-(I) regularized reconstructions from respective noisy data.} \label{fig.radonrecs}
\end{figure}

\begin{figure}
\begin{subfigure}{0.25\textwidth}
\includegraphics[width=\textwidth]{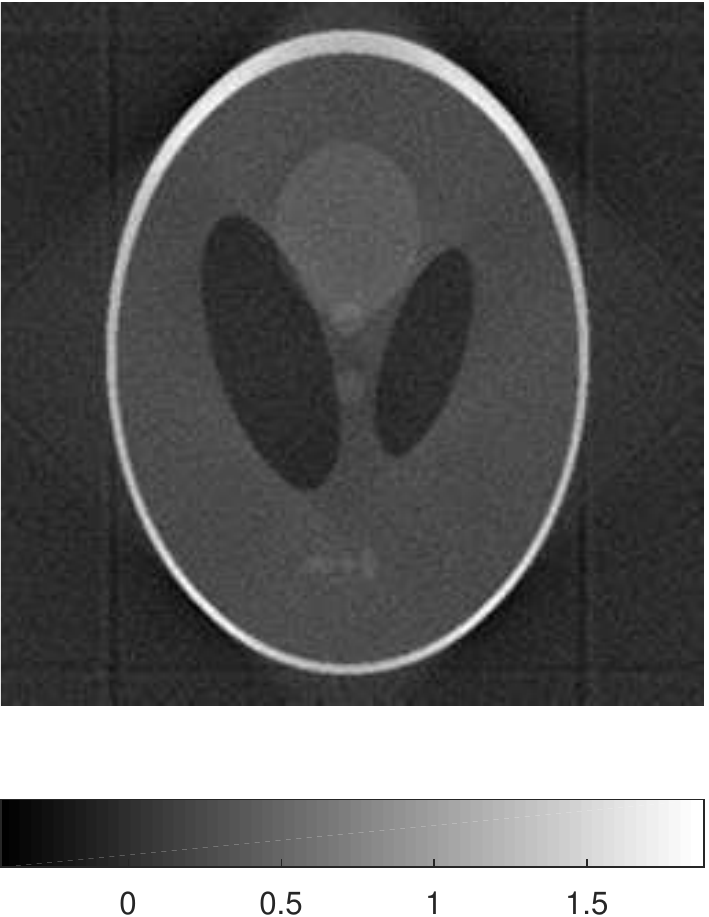}
\caption{}
\end{subfigure}
\begin{subfigure}{0.25\textwidth}
\includegraphics[width=\textwidth]{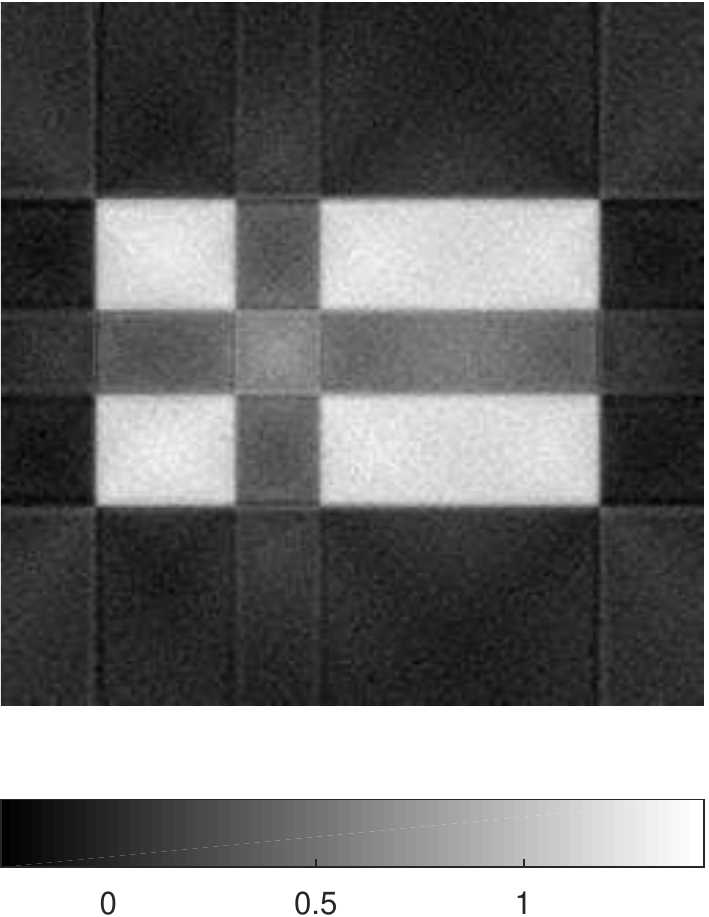}
\caption{}
\end{subfigure}
\begin{subfigure}{0.25\textwidth}
\includegraphics[width=\textwidth]{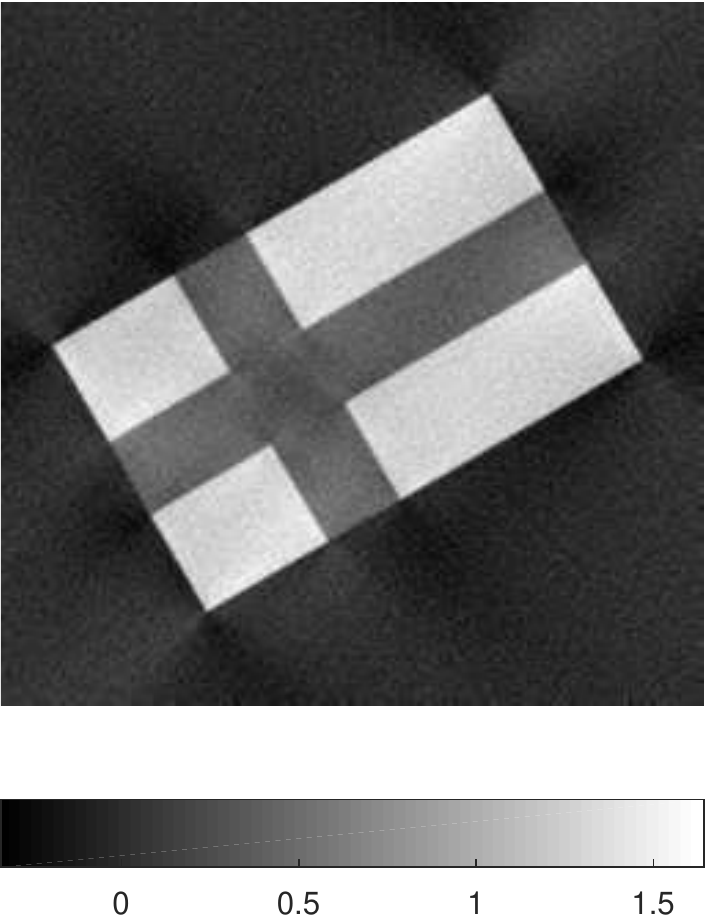}
\caption{}
\end{subfigure}

\begin{subfigure}{0.25\textwidth}
\includegraphics[width=\textwidth]{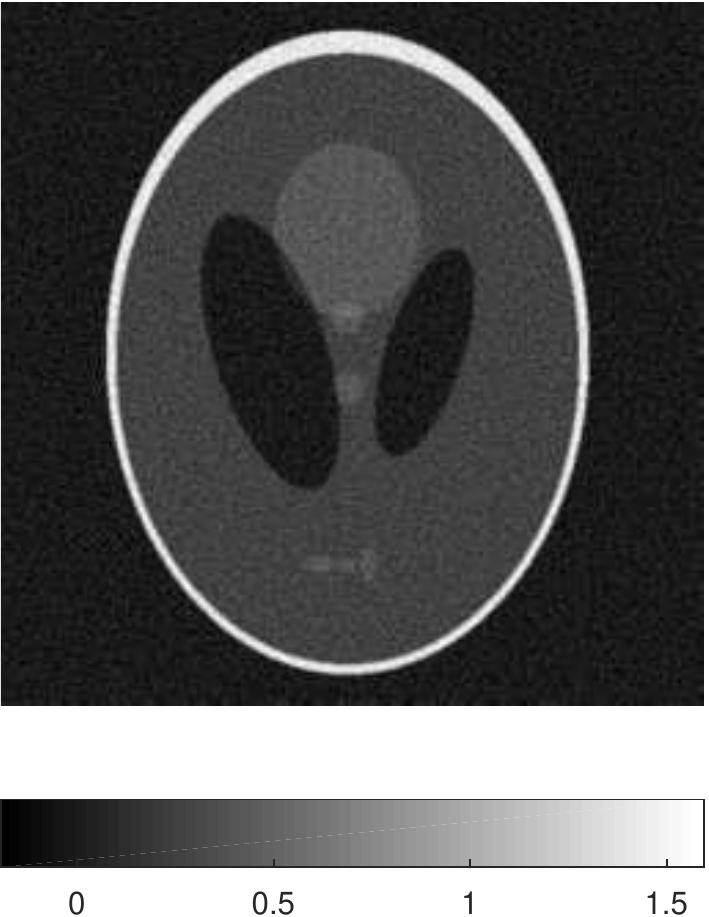}
\caption{}
\end{subfigure}
\begin{subfigure}{0.25\textwidth}
\includegraphics[width=\textwidth]{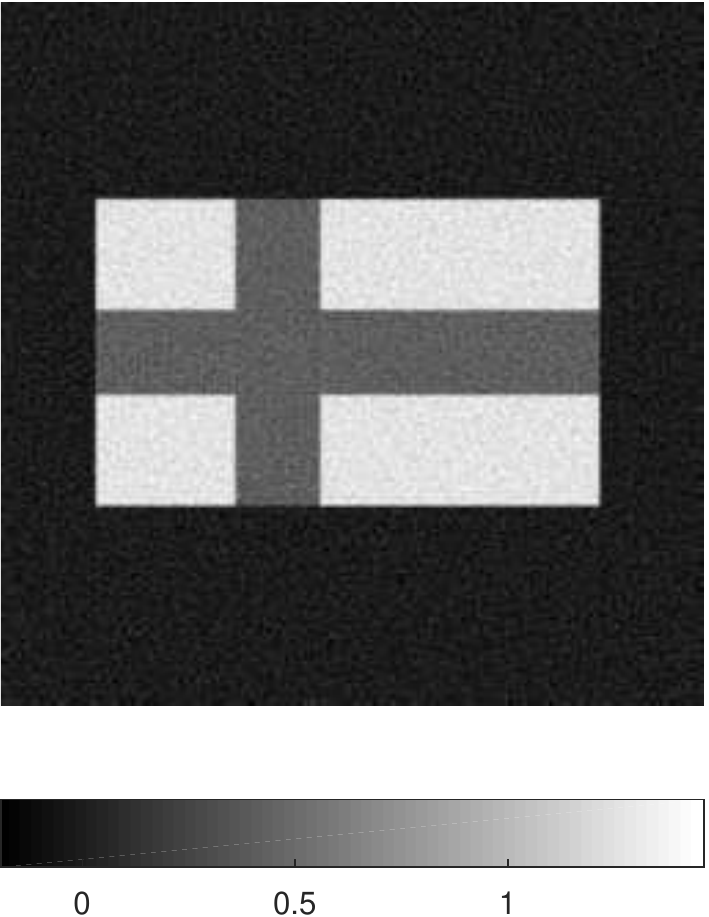}
\caption{}
\end{subfigure}
\begin{subfigure}{0.25\textwidth}
\includegraphics[width=\textwidth]{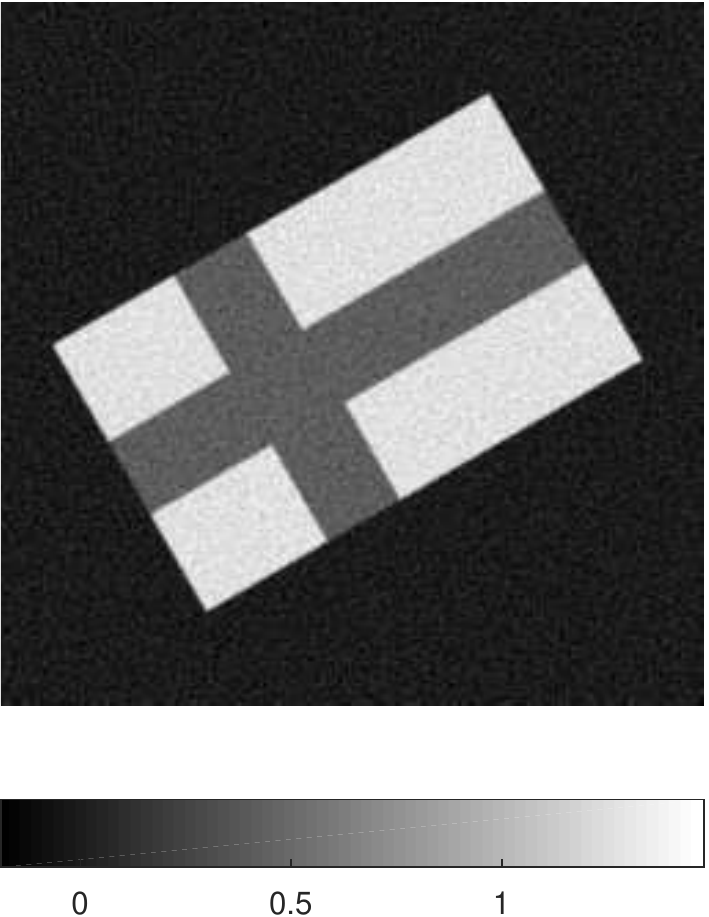}
\caption{}
\end{subfigure}
\caption{(A)-(C) Filtered backprojection reconstructions from Radon data with torus optimized angles. (D)-(F) Filtered backprojection reconstruction from equispaced projection angles with same amount of projections as in (A)-(C).}
\label{fig.fbprecs}
\end{figure}

\subsection{Rotating phantom on torus}
\label{ssec:numexp.additional}

\begin{figure}
\begin{subfigure}{0.25\textwidth}
\includegraphics[width=\textwidth]{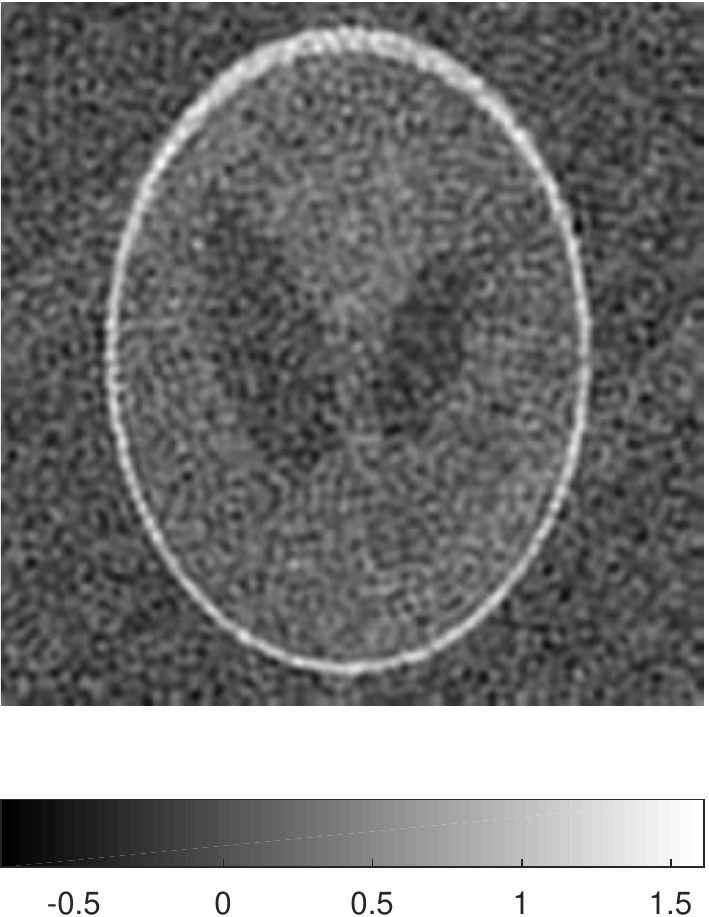}
\caption{$\Theta^{SL}_1$}
\end{subfigure}
\begin{subfigure}{0.25\textwidth}
\includegraphics[width=\textwidth]{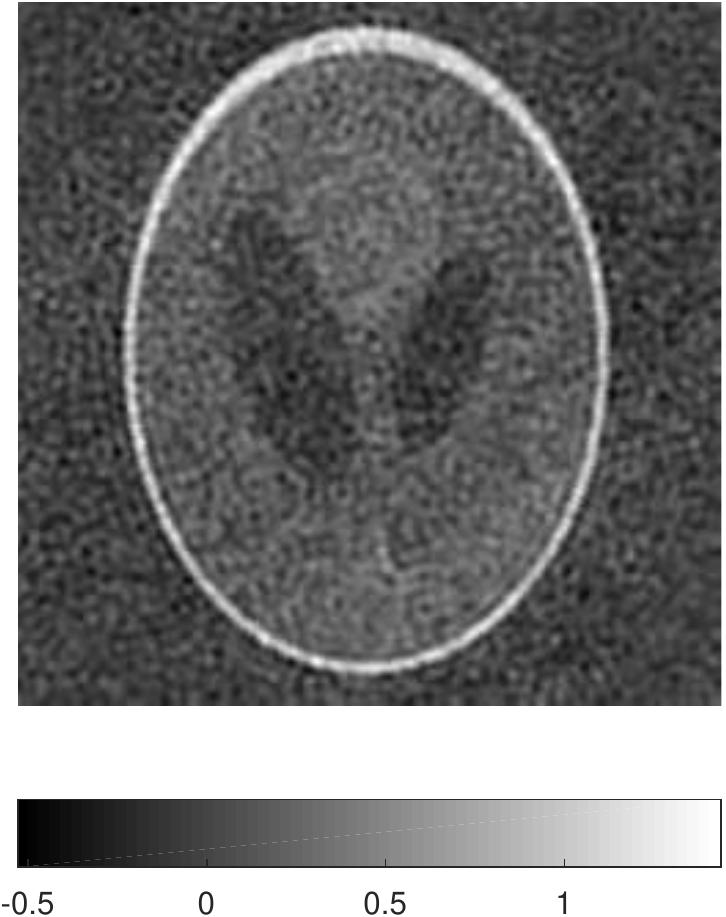}
\caption{$\Theta^{SL}_2$}
\end{subfigure}
\begin{subfigure}{0.25\textwidth}
\includegraphics[width=\textwidth]{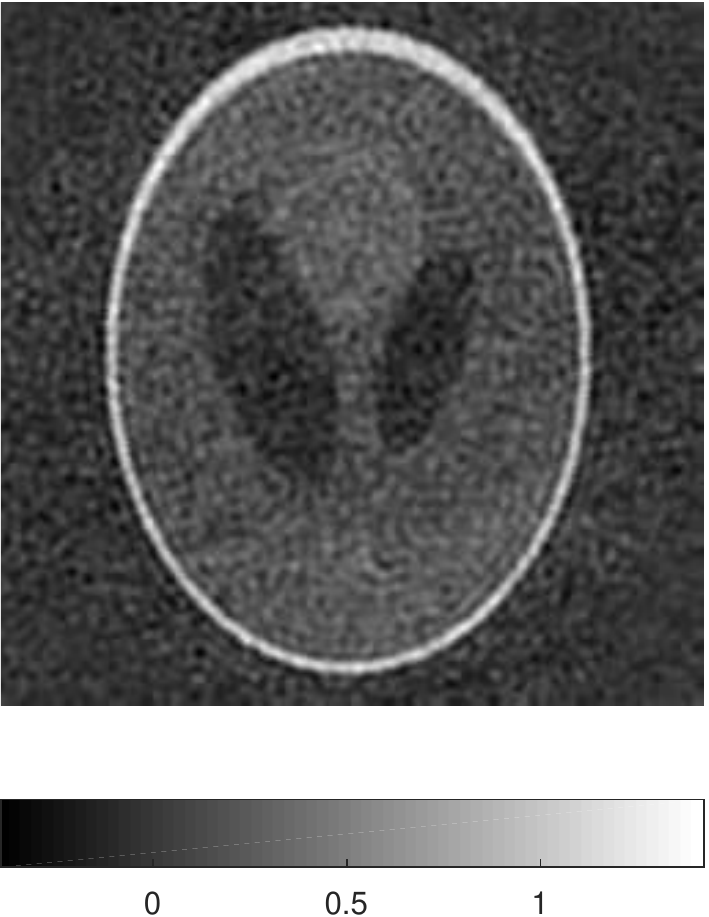}
\caption{$\Theta^{SL}_3$}
\end{subfigure}

\begin{subfigure}{0.25\textwidth}
\includegraphics[width=\textwidth]{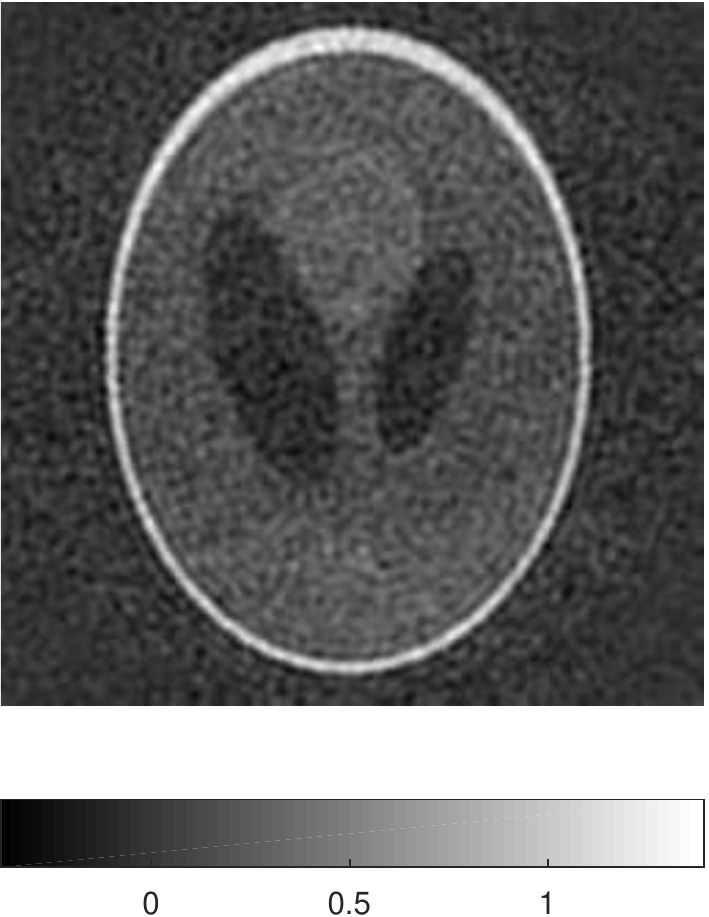}
\caption{$\Theta^{SL}_4$}
\end{subfigure}
\begin{subfigure}{0.25\textwidth}
\includegraphics[width=\textwidth]{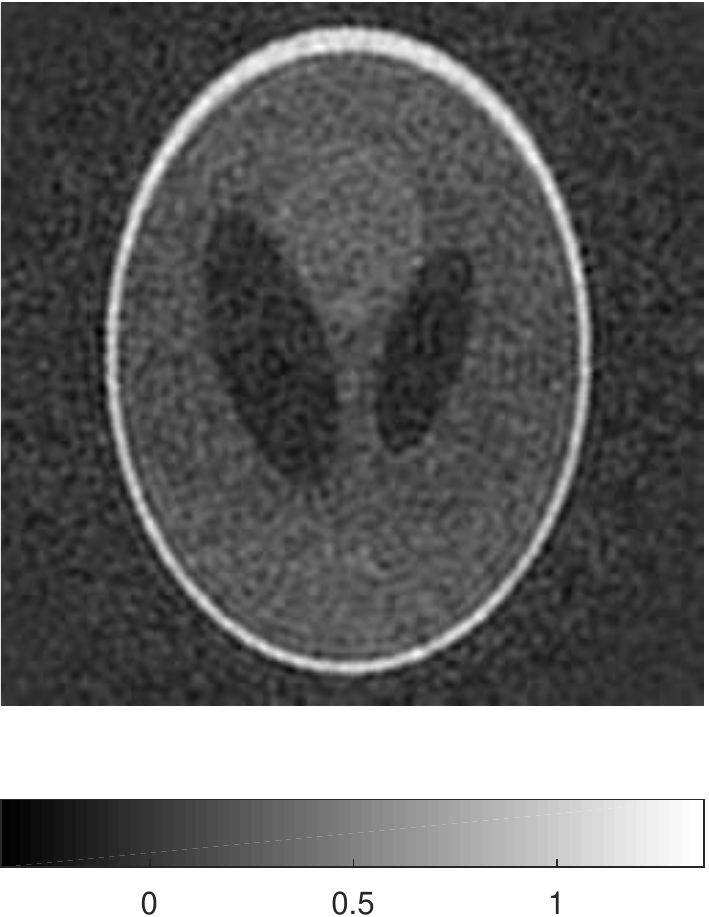}
\caption{$\Theta^{SL}_5$}
\end{subfigure}
\begin{subfigure}{0.25\textwidth}
\includegraphics[width=\textwidth]{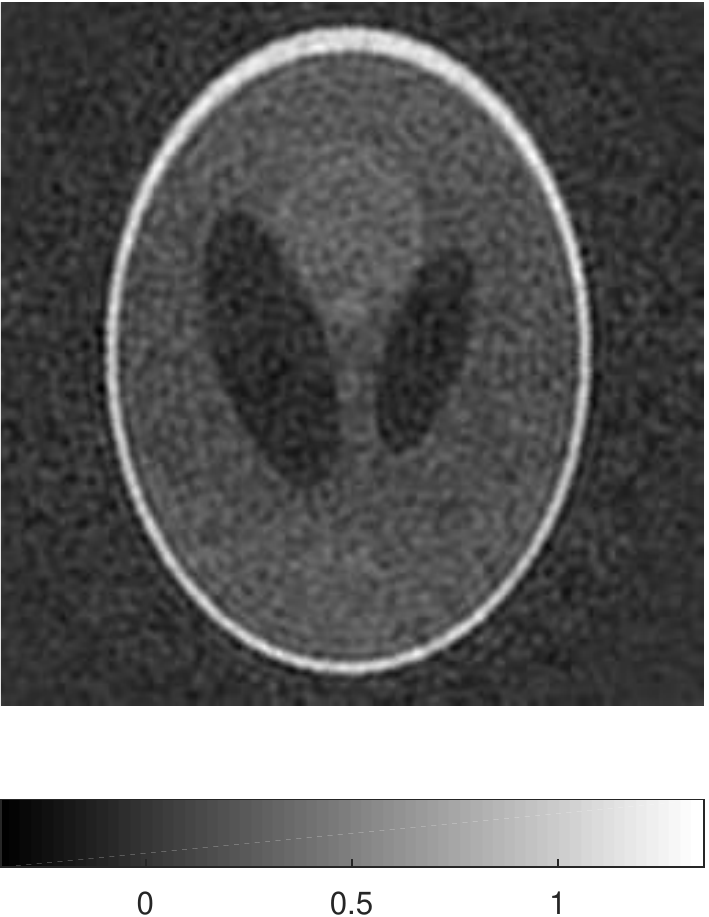}
\caption{$\Theta^{SL}_6$}
\end{subfigure}

\begin{subfigure}{0.25\textwidth}
\includegraphics[width=\textwidth]{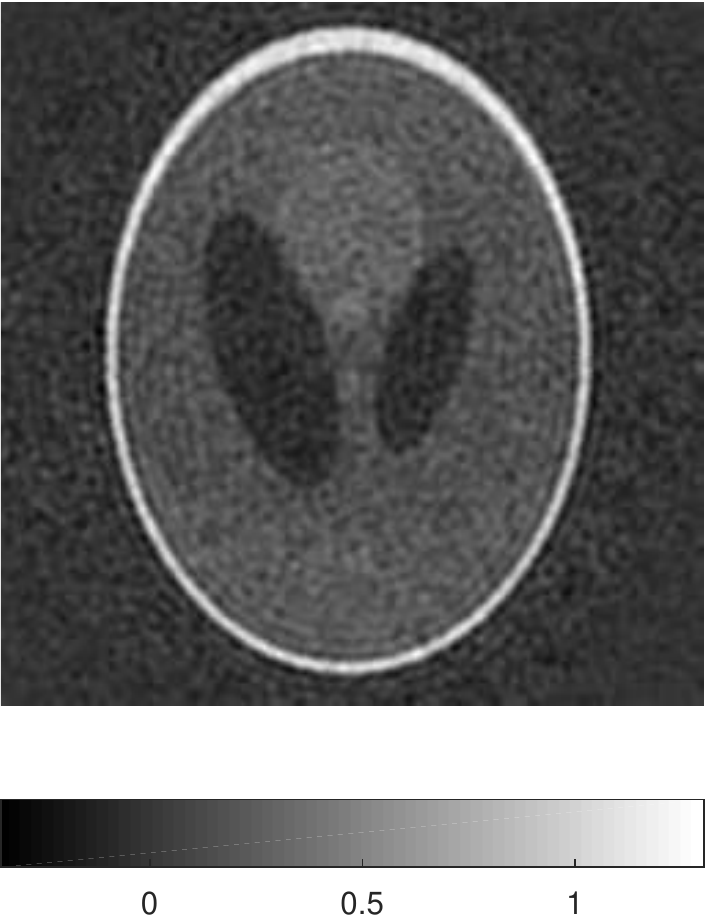}
\caption{$\Theta^{SL}_7$}
\end{subfigure}
\begin{subfigure}{0.25\textwidth}
\includegraphics[width=\textwidth]{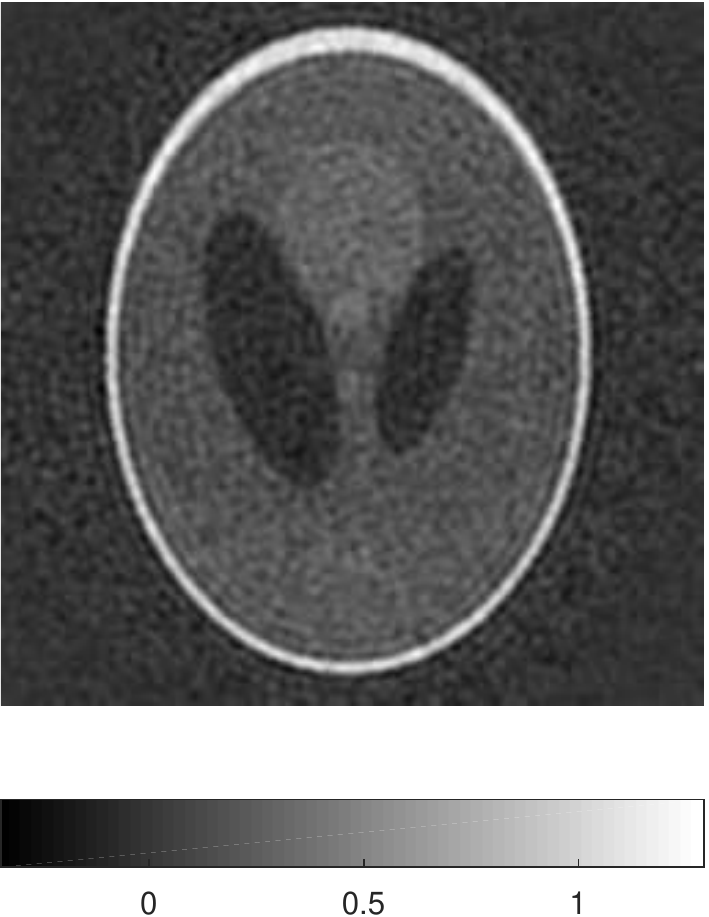}
\caption{$\Theta^{SL}_8$}
\end{subfigure}
\begin{subfigure}{0.25\textwidth}
\includegraphics[width=\textwidth]{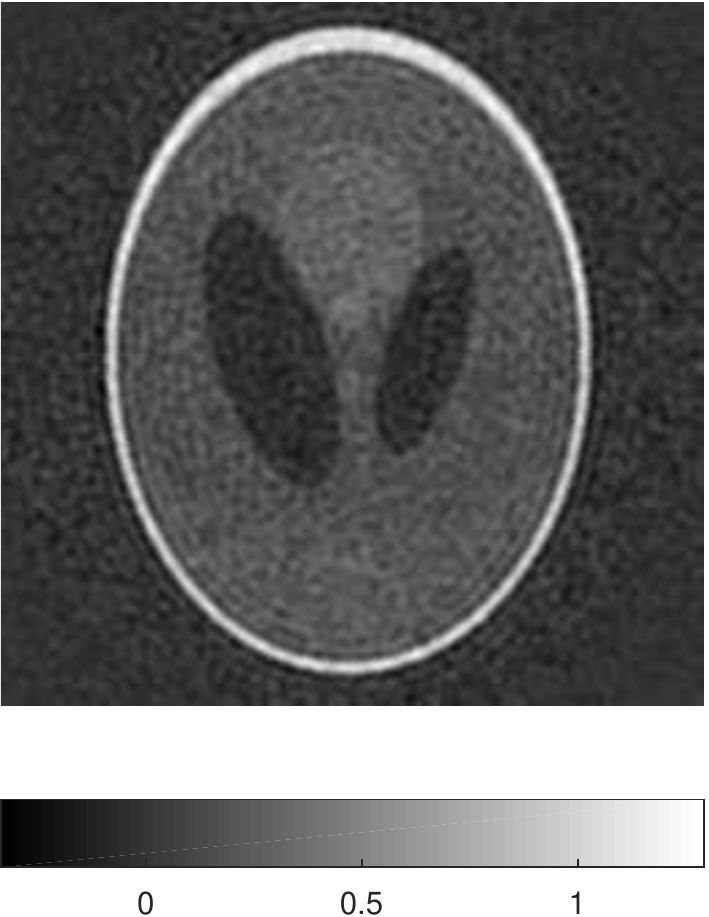}
\caption{$\Theta^{SL}_9$}
\end{subfigure}

\caption{Reconstructions based on data sets of Shepp--Logan phantom rotated with respective angles.}
\label{fig.rotationshepplogan}
\end{figure}

\begin{figure}
\begin{subfigure}{0.25\textwidth}
\includegraphics[width=\textwidth]{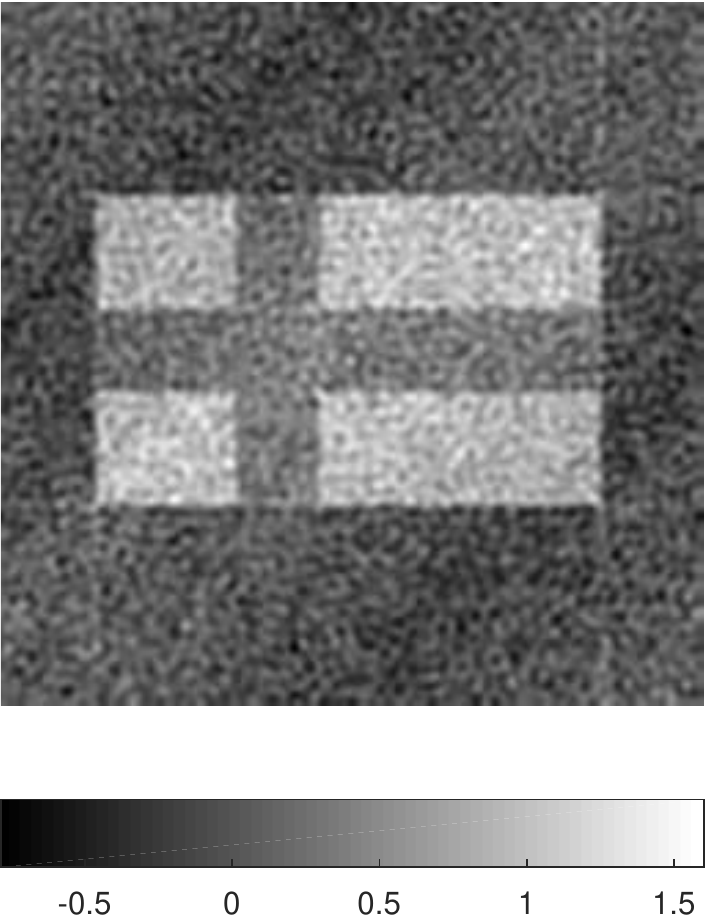}
\caption{$\Theta^F_1$}
\end{subfigure}
\begin{subfigure}{0.25\textwidth}
\includegraphics[width=\textwidth]{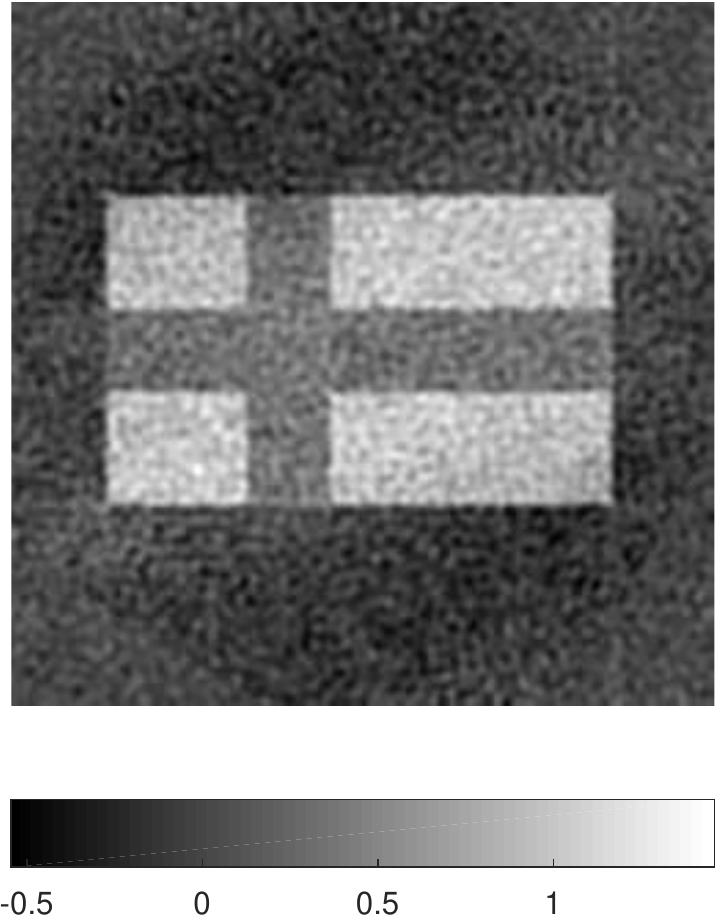}
\caption{$\Theta^F_2$}
\end{subfigure}
\begin{subfigure}{0.25\textwidth}
\includegraphics[width=\textwidth]{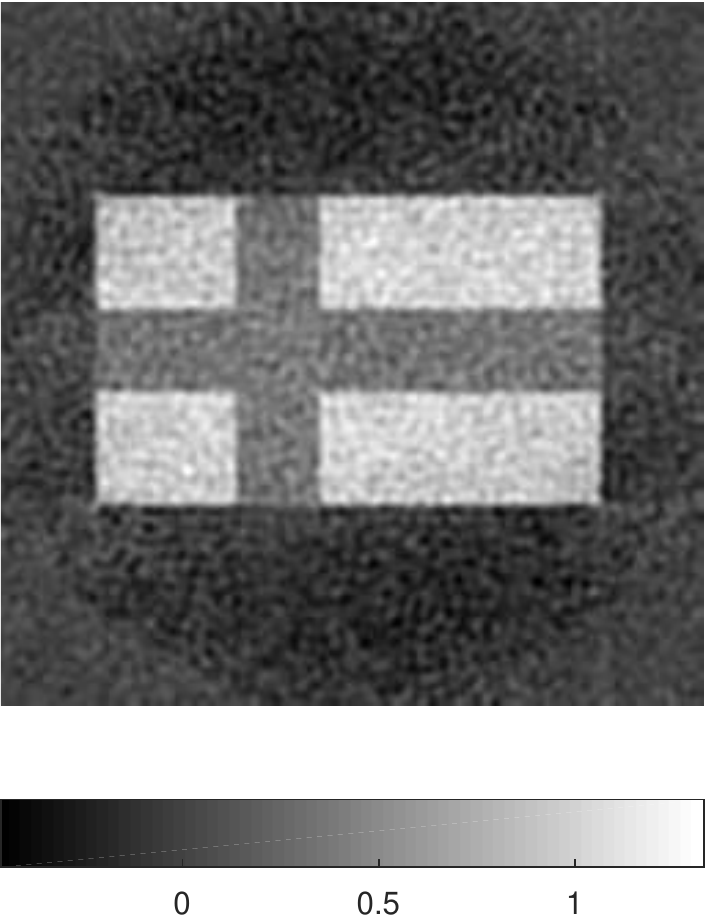}
\caption{$\Theta^F_3$}
\end{subfigure}

\begin{subfigure}{0.25\textwidth}
\includegraphics[width=\textwidth]{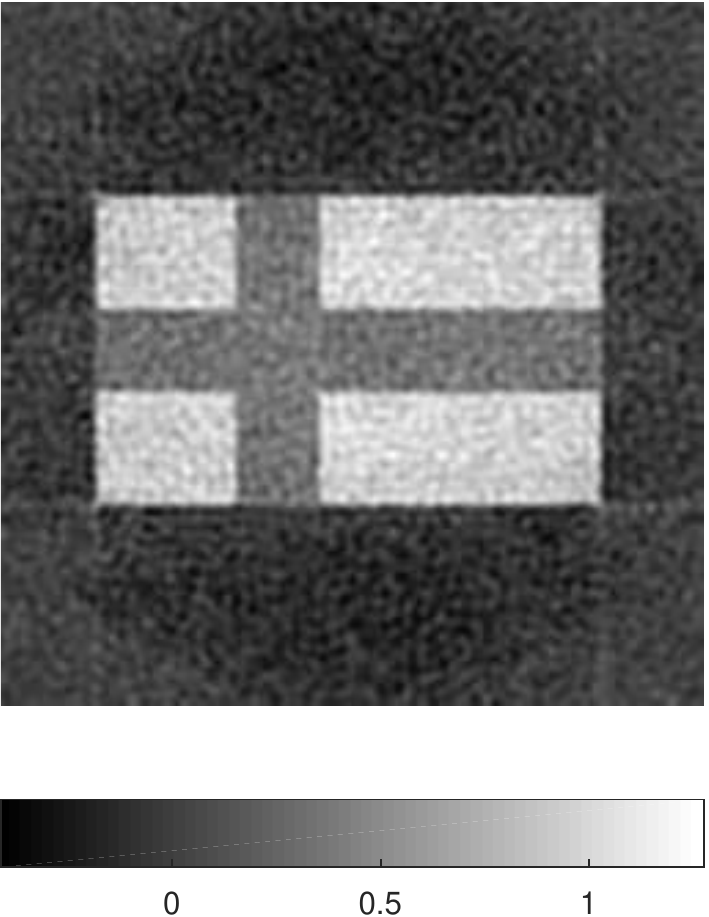}
\caption{$\Theta^F_4$}
\end{subfigure}
\begin{subfigure}{0.25\textwidth}
\includegraphics[width=\textwidth]{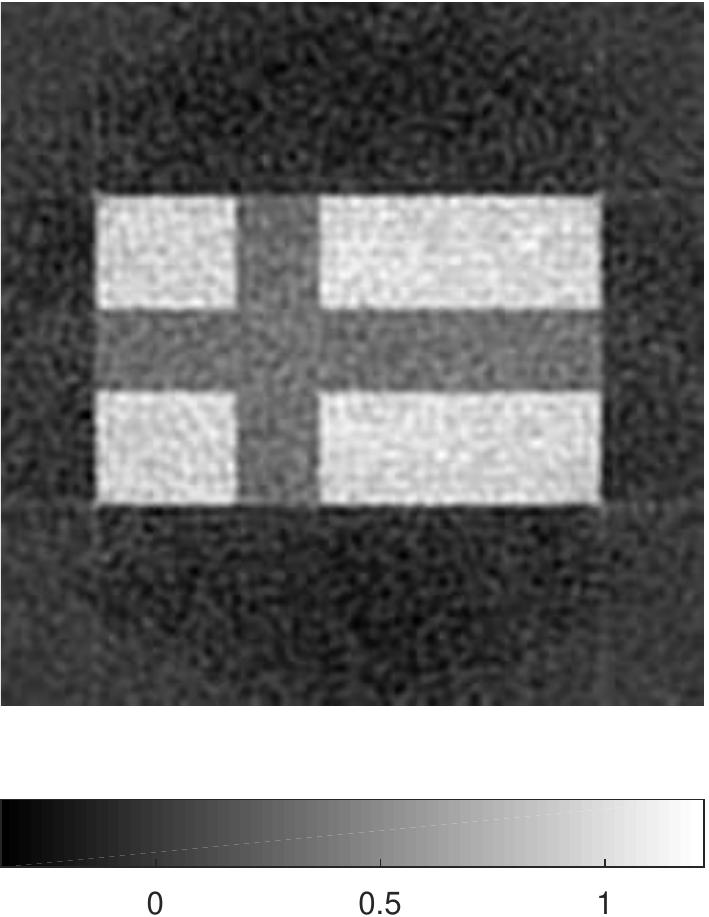}
\caption{$\Theta^F_5$}
\end{subfigure}
\begin{subfigure}{0.25\textwidth}
\includegraphics[width=\textwidth]{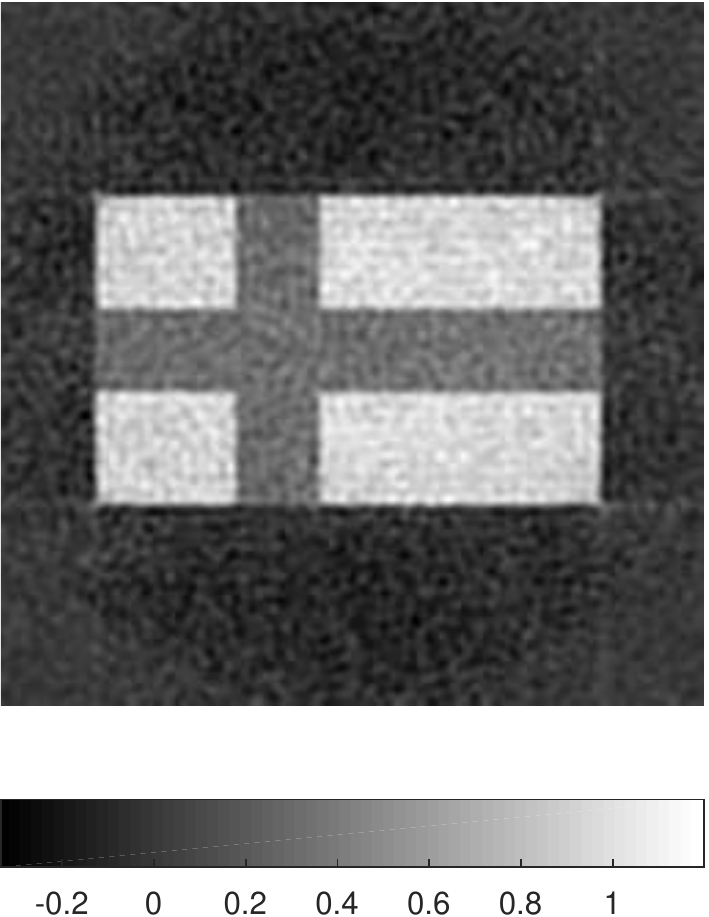}
\caption{$\Theta^F_6$}
\end{subfigure}

\caption{Reconstructions based on data sets of the Flag phantom rotated with respective angles.}
\label{fig.rotationflag}
\end{figure}

\begin{table}
\caption{Reconstruction errors~$\epsilon_p^{0,0}$ of (non-regularized) reconstructions from rotational data sets.}
\label{table.rotationerr}
\begin{tabular}{l 
S[table-format=3.0,table-space-text-post = \%] 
S[table-format=3.0,table-space-text-post = \%] 
S[table-format=3.0,table-space-text-post = \%] 
S[table-format=3.0,table-space-text-post = \%] 
S[table-format=3.0,table-space-text-post = \%] 
S[table-format=3.0,table-space-text-post = \%] 
S[table-format=3.0,table-space-text-post = \%] 
S[table-format=3.0,table-space-text-post = \%] 
S[table-format=3.0,table-space-text-post = \%] }
Shepp--Logan & 
\multicolumn{1}{c}{$\Theta^\mathrm{SL}_1$} & 
\multicolumn{1}{c}{$\Theta^\mathrm{SL}_2$} & 
\multicolumn{1}{c}{$\Theta^\mathrm{SL}_3$} & 
\multicolumn{1}{c}{$\Theta^\mathrm{SL}_4$} & 
\multicolumn{1}{c}{$\Theta^\mathrm{SL}_5$} & 
\multicolumn{1}{c}{$\Theta^\mathrm{SL}_6$} & 
\multicolumn{1}{c}{$\Theta^\mathrm{SL}_7$} & 
\multicolumn{1}{c}{$\Theta^\mathrm{SL}_8$} & 
\multicolumn{1}{c}{$\Theta^\mathrm{SL}_9$} \\
$p=1$ & 113\% &  82\% &  69\% &  61\% &  56\% &  52\% &  49\% &  46\% &  44\%  \\
$p=2$ & 72\% &  54\% &  46\% &  42\% &  40\% &  38\% &  36\% &  34\% &  33\%   \\
$p=\infty$ & 97\% &  89\% &  84\% &  77\% &  77\% &  73\% &  70\% &  69\% &  68\%  \\ 
\vspace{-3mm}\\
Flag & 
\multicolumn{1}{c}{$\Theta^\mathrm{F}_1$} & 
\multicolumn{1}{c}{$\Theta^\mathrm{F}_2$} & 
\multicolumn{1}{c}{$\Theta^\mathrm{F}_3$} & 
\multicolumn{1}{c}{$\Theta^\mathrm{F}_4$} & 
\multicolumn{1}{c}{$\Theta^\mathrm{F}_5$} & 
\multicolumn{1}{c}{$\Theta^\mathrm{F}_6$} \\
$p=1$ &  22\% &  18\% &  16\% &  14\% &  12\% &  12\%  \\
$p=2$ &  25\% &  21\% &  18\% &  16\% &  15\% &  14\%   \\
$p=\infty$ & 96\% &  86\% &  79\% &  66\% &  68\% &  69\%  \\ 
\end{tabular}
\end{table}

{
The theoretical formulation allows to place a phantom inside~$\T^2$ in many different positions. Such choice of an orientation of a phantom leads to different choice of projection directions, and thus results different reconstructions when only finitely many Fourier coefficients are recovered. 
This motivated to test, whether the reconstructions improve when data is acquired from several rotated phantoms. We verified this hypothesis by computing the data of the Shepp--Logan phantom from nine different rotational orientations with $20^\circ$ interval and the Flag phantom with six different rotational orientations with $30^\circ$ interval, both with Fourier coefficient radius $r=50$. 

Denote the angles of rotational data sets of the Shepp--Logan phantom with
$$\Theta_i^{SL} = \left\{ (k-1)\cdot 20^\circ \,;\, k\in\{1,\dots,i\} \right\}$$
and of the Flag phantom with 
$$\Theta_i^F = \left\{ (k-1)\cdot 30^\circ \,;\, k\in\{1,\dots,i\} \right\},$$
where $i\in\N$. Furthermore, denote rotations about the point $(0.5,0.5)$ by the symbols $\theta_k \in \Theta_i$. 

The forward solutions were computed for the rotated phantoms as described in subsection~\ref{ssect:numexp.forward}. Then, we reconstructed the rotated phantoms using formula~(\ref{eq:reconstruction}) and rotated the reconstructions back to the natural orientation, i.e., $-\theta_i$ about the point $(0.5,0.5)$. We denote these naturally oriented reconstructions by $f_\text{rec}^{\alpha,s}(x; k)$ for each $k =1,\dots,i$. 
Eventually, we computed the average of the rotated reconstructions
\begin{equation}
    f_{\Theta_i} := \frac{1}{i}\sum_{k=1}^i f_\text{rec}^{\alpha,s}(x; k).
\end{equation} 
%

For the Shepp--Logan phantom, the rotation was computed using Matlab's \texttt{imrotate} with \texttt{crop} option for the Shepp--Logan phantom. For the Flag phantom, rotation about the point $(0.5,0.5)$ was made by changing the coordinates in equations~\eqref{eq:flag1} and~\eqref{eq:flag2}. For both phantoms, the rotation of the reconstruction to natural orientation was computed using \texttt{imrotate}.
}

The reconstructions $f_{\Theta_i}$ 
are shown in Figure~\ref{fig.rotationshepplogan} for the Shepp--Logan phantom and in Figure~\ref{fig.rotationflag} for the Flag phantom. The reconstruction errors~$\epsilon_p^{0,0}$ (eq.~\ref{eq.errorequation}) are tabulated in Table~\ref{table.rotationerr} for both phantoms. With the Flag phantom, the reconstruction errors are computed only on the support of the flag phantom and the errors are ruled out at the points where the Flag phantom vanish. With the Shepp--Logan,~$\epsilon_p^{0,0}$ was computed on the whole grid. The reconstructions were evaluated in a grid of $256 \times 256$ pixels.

With the Shepp--Logan phantom there is clear visual improvement with the increase of rotationally acquired data and also decreasing trend in the errors. With the simpler Flag phantom, reconstruction quality seems to saturate as the decrease of error norms is not as clear as with the Shepp--Logan phantom. Nevertheless, the smallest error norm is given with the highest number of rotational data (shown in Table~\ref{table.rotationerr}) and the visual evaluations support this.

The rotation of the phantom does contribute to the improvement of reconstructions. In other words, the improvement is not merely due to averaging out the zero mean noise. We simulated data from the same rotational orientation of the Shepp--Logan phantom nine times, and the errors were as follows: $\epsilon_1^{0,0}=52\%$, $\epsilon_2^{0,0}=37\%$ and $\epsilon_\infty^{0,0}=66\%$. 
For the Flag phantom with six times from same rotational orientation, the errors were: $\epsilon_1^{0,0}=43\%$, $\epsilon_2^{0,0}=28\%$ and $\epsilon_\infty^{0,0}=76\%$. 
In both cases the errors were higher than what was gained with different rotational orientations, except for~$\epsilon_\infty^{0,0}$ with Shepp--Logan phantom where the error was almost equal. It should also be noted that the use \texttt{imrotate} induces some blurring during the rotation of reconstructions and the Shepp--Logan phantom. Nonetheless, rotational reconstructions approach performed better.

\subsection{Computing times}
\label{ssec:comptimes}

The computing time of the forward system, i.e., data, depends mainly on the cutoff radius~$r$ of the Fourier series and on the number used geodesics in each direction (see discretization of~$x$ in~\ref{ssec:discretization}). In terms of this paper, the radius~$r$ was more of the interest. Discretization relates to the numerical accuracy and the data acquisition accuracy of experimental setup. 
Example computing times~$t_r$ for data on Lenovo P51 laptop with Intel i7-7820HQ CPU and 32~GB of RAM having MATLAB R2017a (The MathWorks, Inc.) with the Shepp--Logan phantom and $\mathcal{A}_1$ are $t_{50}=5.5~\text{min}$ and $t_{100}=51~\text{min}$; and with the Flag phantom and $\mathcal{A}_2$, $t_{50} = 100~\text{min}$ and $t_{100}=62~\text{min}$.
On Lenovo P910 high-end workstation with two Intel Xeon E5-2697 processors and 256~GB RAM having MATLAB version R2016b \mbox{64-bit} (The MathWorks, Inc.), the computation times were of order 
$t_{50}=2.2~\text{min}$,
$t_{100}=15~\text{min}$,
$t_{150}=60~\text{min}$ and
$t_{200}=188~\text{min}$ with the Shepp--Logan phantom and~$\mathcal{A}_1$; and
$t_{50}=1.6~\text{min}$,
$t_{100}=21~\text{min}$,
$t_{150}=79~\text{min}$ and
$t_{200}=242~\text{min}$ with the Flag phantom~$\mathcal{A}_2$.
The analytical integration applied when using the Shepp--Logan phantom with~$\mathcal{A}_1$ explains its faster computations times.

The projection of Radon transform sinogram to the torus~$\mathcal{A}_{\T^2}$ and its reconstruction lasted approximately eight minutes on Lenovo P51. However, the current implementation was not optimized at all and included, among other, three nested for-loops. Hence, here the computational efficiency will increase during further development.

\section{Conclusions}
\label{sec:conclusions}

{
Our main conclusion is that the new inversion method for the X-ray transform based on the torus is numerically applicable.
We emphasize, however, that this is only a first step.
Much is left to be studied, such as the details of the mapping that transforms the planar X-ray transforms onto the torus and application of different regularization methods.
We have developed the previously unapplied theory in a direction relevant to computation and carried out first numerical tests.
}

Our theoretical results were strongly motivated by practical requirements, including new and computationally fast reconstruction formulas from X-ray data in theorem~\ref{thm:invfor}, and rigorous mathematical theory for Tikhonov regularized reconstructions from X-ray data on the flat torus in theorems~\ref{thm:regularization} and~\ref{thm:strategy}.
We gave mathematical formulations of discretized forward and inverse models in section~\ref{sec:implementation}, and considered numerical analysis in section~\ref{ssec:dft}.
The numerical implementation using Matlab was described in section~\ref{sec:numerics}.

The numerical implementation demonstrated the efficacy of Torus CT.
Torus CT is computationally relatively efficient compared with iterative techniques, though still slower than current implementations of the FBP.
An interesting feature of Torus CT is its meshless nature: Once the Fourier coefficients are computed, the reconstruction can be evaluated in any desired grid points.
Currently, theory and the implementation are established in 2 dimensions, which is suitable for slice-wise reconstructions of 3-dimensional structures.
One future research direction could be the development of algorithms and theory in higher dimensional settings.

Data simulation was also computed with the traditional Radon transform corresponding to experimental image acquisition with projection angles preferred by the Torus CT.
Reconstruction quality was promising { but not quite competitive with state-of-the-art methods}.
Some initial work has been conducted with conventional, evenly distributed projection angles, in which case there are various ways to interpolate projection data to directions preferred by Torus CT.
It seems that rotations of a phantom could result sharp reconstructions and might allow reduction in the number of projection directions.
This question should be studied more and with experimental X-ray data.

{
We find that Torus CT opens up a promising path to studying X-ray tomography in a new way.
It provides a useful setup for future study --- both theoretical, numerical, and practical --- of the many remaining questions.
}

\section*{Supplementary material}
\label{sec:supple}

\subsection*{Matlab code}
We provide Creative Commons 4.0 licensed Matlab code files that implement forward model~$\mathcal{A}_1$ (section~\ref{sssec:forwardontorus}) and inverse solutions on torus (section~\ref{ssec:Computationalinversemodel}). 
The code package comprises of three files: \texttt{TorusCTrun.m} is the main script, \texttt{DFT.m} implements discrete Fourier transform (section~\ref{ssec:dft}), and \texttt{LineIntegralOnGrid.m} computes the exact line integral~\eqref{eq:forwardmodel1int} over periodically extended, pixelized phantom.
{ The files} are available at~\cite{RK19}.

\bibliographystyle{abbrv}
\bibliography{sample}

\end{document}